\documentclass[a4paper]{amsart}
\usepackage{amssymb, enumerate}
\usepackage{stmaryrd}
\usepackage{directory}
\usepackage{tikz}
\usepackage[all]{xy}

\usepackage{hyperref, aliascnt}
\usepackage{mathtools}
\def\today{\number\day\space\ifcase\month\or   January\or February\or
   March\or April\or May\or June\or   July\or August\or September\or
   October\or November\or December\fi\   \number\year}

\theoremstyle{definition}
\newtheorem{lma}{Lemma}[section]

\newaliascnt{thmCt}{lma}
\newtheorem{thm}[thmCt]{Theorem}
\aliascntresetthe{thmCt}

\newaliascnt{corCt}{lma}
\newtheorem{cor}[corCt]{Corollary}
\aliascntresetthe{corCt}

\newaliascnt{propCt}{lma}
\newtheorem{prop}[propCt]{Proposition}
\aliascntresetthe{propCt}

\newtheorem*{thm*}{Theorem}
\newtheorem*{qst*}{Question}
\newtheorem*{cor*}{Corollary}
\newtheorem*{prop*}{Proposition}

\newcounter{theoremintro}
\newtheorem{thmintro}[theoremintro]{Theorem}

\newaliascnt{pgrCt}{lma}

\aliascntresetthe{pgrCt}

\newaliascnt{dfCt}{lma}
\newtheorem{df}[dfCt]{Definition}
\aliascntresetthe{dfCt}

\newaliascnt{remCt}{lma}
\newtheorem{rem}[remCt]{Remark}
\aliascntresetthe{remCt}

\newaliascnt{remsCt}{lma}

\aliascntresetthe{remsCt}

\newaliascnt{egCt}{lma}
\newtheorem{eg}[egCt]{Example}
\aliascntresetthe{egCt}

\newaliascnt{egsCt}{lma}

\aliascntresetthe{egsCt}

\newaliascnt{qstCt}{lma}
\newtheorem{qst}[qstCt]{Question}
\aliascntresetthe{qstCt}

\newaliascnt{pbmCt}{lma}
\newtheorem{pbm}[pbmCt]{Problem}
\aliascntresetthe{pbmCt}

\newaliascnt{notaCt}{lma}
\newtheorem{nota}[notaCt]{Notation}
\aliascntresetthe{notaCt}

\newaliascnt{cnjCt}{lma}
\newtheorem{cnj}[cnjCt]{Conjecture}
\aliascntresetthe{cnjCt}

\newcommand{\beq}{\begin{equation}}
\newcommand{\eeq}{\end{equation}}
\newcommand{\beqa}{\begin{eqnarray*}}
\newcommand{\eeqa}{\end{eqnarray*}}
\newcommand{\bal}{\begin{align*}}
\newcommand{\eal}{\end{align*}}
\newcommand{\bi}{\begin{itemize}}
\newcommand{\ei}{\end{itemize}}
\newcommand{\be}{\begin{enumerate}}
\newcommand{\ee}{\end{enumerate}}

\newcommand{\ep}{\varepsilon}

\newcommand{\Q}{{\mathbb{Q}}}

\newcommand{\Z}{{\mathbb{Z}}}
\newcommand{\R}{{\mathbb{R}}}
\newcommand{\C}{{\mathbb{C}}}
\newcommand{\N}{{\mathbb{N}}}

\newcommand{\K}{{\mathcal{K}}}

\newcommand{\U}{{\mathcal{U}}}

\newcommand{\T}{{\mathbb{T}}}

\newcommand{\KK}{\mathrm{K}}
\newcommand{\Lat}{\mathrm{Lat}}
\newcommand{\QT}{\mathrm{QT}}
\newcommand{\CatCu}{\mathbf{Cu}}
\newcommand{\V}{\mathrm{V}}

\newcommand{\Cu}{{\mathrm{Cu}}}

\newcommand{\id}{{\mathrm{id}}}

\newcommand{\spec}{{\mathrm{sp}}}

\newcommand{\supp}{{\mathrm{supp}}}

\newcommand{\Ell}{{\mathrm{Ell}}}

\newcommand{\Lsc}{\mathrm{Lsc}}
\newcommand{\NN}{\overline{\mathbb{N}}}
\newcommand{\W}{\mathrm{W}}
\newcommand{\Paths}{\mathrm{Paths}}

\newcommand{\CatPom}{\mathrm{\textbf{PoM}}}
\newcommand{\vect}[1]{\textbf{#1}}
\newcommand{\filter}{\mathcal{U}}

\newcommand{\ihom}[1]{\llbracket #1 \rrbracket}

\newcommand{\ca}{C$^*$-algebra}
\newcommand{\cas}{$C^*$-algebras}
\newcommand{\uca}{unital $C^*$-algebra}

\newcommand{\I}{\infty}

\date{\today}
\title[]{The modern theory of Cuntz semigroups of C$^*$-algebras}

\thanks{The first named author was supported by a starting grant of 
the Swedish Research Council. The second named author was partially supported by MINECO (grant No. PID2020-113047GB-I00).}

\author{Eusebio Gardella}
\address[Eusebio Gardella]
{Department of Mathematical Sciences, Chalmers University of
Technology and University of Gothenburg, Gothenburg SE-412 96, Sweden.}
\email{gardella@chalmers.se}
\urladdr{www.math.chalmers.se/~gardella}

\author{Francesc Perera}
\address[Francesc Perera]{Departament de Matem\`{a}tiques,
Universitat Aut\`{o}noma de Barcelona,
\linebreak 08193 Bellaterra, Barcelona, Spain, and
Centre de Recerca Matem\`atica, Edifici Cc, Campus de Bellaterra,  08193 Cerdanyola del Vall\`es, Barcelona, Spain}
\email[]{perera@mat.uab.cat}
\urladdr{https://mat.uab.cat/web/perera}

\begin{document}

\begin{abstract}
We give a detailed introduction to the theory of Cuntz semigroups
for C$^*$-algebras. Beginning with the most basic definitions and 
technical lemmas, 
we present several results of historical importance, 
such as Cuntz's theorem on the existence of quasitraces, 
R{\o}rdam's proof that $\mathcal{Z}$-stability implies 
strict comparison, and Toms' example of a non 
$\mathcal{Z}$-stable simple, nuclear C$^*$-algebra.
We also give the reader an extensive overview of the state of 
the art and the modern approach to the theory, including the recent results for C$^*$-algebras of stable rank one (for example, the Blackadar-Handelman conjecture and the 
realization of ranks), as well as the abstract
study of the Cuntz category \textbf{Cu}. 
\end{abstract}

\maketitle

\begin{center}\emph{Dedicated to the memory of Eberhard Kirchberg.}
\end{center}
\vspace{.5cm}

\tableofcontents

\renewcommand*{\thetheoremintro}{\Alph{theoremintro}}

\section{Introduction}

The Cuntz semigroup is an invariant for \ca s whose origin 
can be traced back to the seminal work of Joachim Cuntz \cite{Cun_dimension_1978} on the existence of quasitraces on
simple, stably finite C$^*$-algebras. 
The Cuntz semigroup $\Cu(A)$ of a \ca\ $A$ resembles the Murray-von Neumann semigroup $\V(A)$, but 
is constructed using positive elements instead of projections
(and a suitable equivalence relation).  
The comparison between the Cuntz semigroup and the K-groups
(particularly K$_0$) puts in perspective the advantages and 
disadvantages of these invariants. 
Arguably one of the biggest advantages of K-theory is its
computability, as there are for instance several 6-term exact sequences that are very useful in
a number of situations. On the other hand, many C$^*$-algebras
do not contain any nontrivial projections, and thus $\V(A)$
and K$_0(A)$ may in general contain very little information 
about $A$ outside the class of real rank zero C$^*$-algebras, which is a class where projections abound. For example,
the complex numbers $\C$, continuous functions on the Hilbert cube $[0,1]^\N$, the Jiang-Su algebra $\mathcal{Z}$,
and the suspension of the Calkin algebra $S\mathcal{Q}$ 
all have the same 
K-theory. Moreover, the K-groups of a C$^*$-algebra do not contain any information
about its ideal structure, which helps explain why virtually all classification results that only use the 
K-theory of the algebra must assume simplicity.
The Cuntz semigroup, on the other hand, always
contains plenty of information about the C$^*$-algebra, essentially
because every
C$^*$-algebra contains a great deal of positive elements. 
(Exactly what
kind of information about $A$ is encoded in $\Cu(A)$ is not 
completely clear, but we give many concrete instances in the 
theorems below.)
The Cuntz semigroup is unfortunately rather difficult to compute,
which makes its rich information sometimes difficult to access.
This is perhaps not entirely surprising, given how intrincate 
its structure is. Interestingly, in certain instances enough information suffices to distinguish algebras without the need of a full computation. For example, the Cuntz semigroup distinguishes
the algebras $\C$, $C([0,1]^\N)$, $\mathcal{Z}$ and 
$S\mathcal{Q}$ mentioned above. 
(We do not have an explicit computation of
either $\Cu(C([0,1]^\N))$ or $\Cu(S\mathcal{Q})$, but enough about them is known to claim that they are different.)

The Cuntz semigroup is intimately related with classification, particularly with the classification program of simple, nuclear C$^*$-algebras initiated by George Elliott. The original conjecture aimed at classifying all simple, separable, unital, nuclear C$^*$-algebras using K-theoretical data, conveniently encoded in the Elliott invariant $\Ell$ that, loosely speaking, consists of the 
$\KK_0$-group, the topological $\KK_1$-group, the trace simplex, and the pairing between projections and traces. 
More precisely then, it was asked whether an isomorphism of invariants $\Ell(A)\cong \Ell(B)$ could be lifted to an isomorphism of the algebras $A\cong B$. 
The relevance of the Cuntz semigroup in the modern theory of C$^*$-algebras was made evident in the celebrated work of 
Toms \cite{Tom_classification_2008}, where he constructed two 
simple, separable, nuclear, unital \ca s $A$ and $B$ which 
satisfy $\mathrm{Ell}(A)\cong \Ell(B)$ and $A\ncong B$. These
algebras constitute a counterexample to Elliott's conjecture, 
and were distinguished using the Cuntz semigroup; see 
Section~\ref{sec:Toms} for an exposition of Toms' examples and more details of the classification program.\footnote{Toms' counterexample to the Elliott conjecture was not the first one. Indeed, R{\o}rdam 
had earlier constructed \cite{Ror_simple_2003} a simple, nuclear 
C$^*$-algebra $A$ containing a finite and an infinite projection.
It follows that $A$ and $A\otimes\mathcal{Z}$ have the same Elliott
invariant, but are not isomorphic. 
Prior to R{\o}rdam's construction, Villadsen \cite{Vil_range_1995} had constructed 
examples of simple, separable, nuclear \cas{} which agreed on the Elliott invariant but were not isomorphic.
The relevance of Toms' example
in the context of these notes 
stems from the fact that he used the Cuntz semigroup
to distinguish C$^*$-algebras with isomorphic Elliott invariants.}
The work of Toms motivated the systematic 
study of the Cuntz semigroup, which was initiated by
Coward, Elliott and Ivanescu in 
\cite{CowEllIva_Cuntz_2008}. 
Since then, 
numerous papers have been written about the Cuntz semigroup
developing a categorical framework for their study, thus 
unveiling fascinating features of this invariant.

The Cuntz semigroup has been successfully used to classify 
certain classes of C$^*$-algebras, as well as maps between them.
Ciuperca and Elliott's classification \cite{CiuEll_remark_2008} of AI-algebras was ultimately greatly generalized by Robert \cite{Rob_classification_2012}, 
who classified certain direct limits of
one-dimensional NCCW-complexes. These classification results 
are obtained as consequences (via an intertwining argument) 
of general theorems classifying
homomorphisms from said algebras into arbitrary \ca s of stable
rank one. A basic form of these classification results for maps
is the work \cite{RobSan_classification_2010} 
of Robert and Santiago, where they 
show that two homomorphisms from $C_0((0,1])$ into a \ca\ of
stable rank one are approximately unitarily equivalent if and only if they are equal at the level of the Cuntz semigroup.
It should be noted that none of these results requires either
the domain or codomain algebra to be simple; this is 
perhaps not so surprising considering the fact that the Cuntz semigroup encodes the ideal lattice of the algebra (see below) as well as other structural aspects (see, for example, \cite{OrtPerRor_corona_2012} and also \cite{RobRor_divisibility_2013}).

The goal of this survey is to introduce the reader to this 
rich theory, beginning with a detailed exposition of the basics,
and proving some of the most celebrated results. We will also 
give an overview of the modern theory of Cuntz semigroups, 
particularly the spectacular recent developments for 
C$^*$-algebras of stable rank one.

In the rest of this introduction, we give a summary of the main 
results discussed in this work. As it turns out,
the Cuntz semigroup of a C$^*$-algebra is 
quite a special kind of ordered semigroup: 
for example, suprema of increasing sequences always exist, and 
addition is compatible with suprema and with 
the so-called compact containment relation $\ll$; see \autoref{sec:waybelow}. Thus, the Cuntz semigroup $\Cu(A)$ naturally
belongs to a subcategory of positively ordered monoids:

\begin{thmintro}\label{thmintro:CatCu}(Coward-Elliott-Ivanescu
\cite{CowEllIva_Cuntz_2008}).
There is a category \textbf{Cu} of positively ordered monoids
to which $\Cu(A)$ belongs for every C$^*$-algebra $A$. 
Moreover,
the Cuntz semigroup determines a functor 
$\Cu\colon \mathbf{C^*}\to \CatCu$ which respects 
(countable) direct limits.
\end{thmintro}

In fact, considering the Cuntz semigroup as an ordered set, it becomes an \emph{$\omega$-domain}, that is, a sequentially complete partially ordered set which is also \emph{$\omega$-continuous}; see \cite{Kei_cuntz_2017}, and also \cite{GieHofKeiLawMisSco_continuous_2003,Thi_notes}. While the work \cite{CowEllIva_Cuntz_2008} only considered
\emph{countable}
direct limits, it was later shown in \cite{AntPerThi_tensor_2018} 
that this assumption is not necessary: $\CatCu$ posseses arbitrary
direct limits, and $\Cu$ preserves them. 
More properties of the category \textbf{Cu} and the functor Cu 
will be addressed in Theorem~\ref{thmintro:CatCu2}. 

The objects in the category \textbf{Cu} are partially
ordered semigroups satisfying certain axioms (see 
\autoref{df:CatCu}) which are enjoyed by $\Cu(A)$ for all 
\ca s $A$. The goal of describing precisely which partially
ordered semigroups arise from \ca s has led to the discovery of
five additional axioms, but even these do not entirely
describe the range of the invariant. Obtaining a complete
exlicit description is an
extremely complicated task, 
and we are very far from achieving it.
This should be compared to the situation with K-theory,
where it is not so hard to show that every pair of abelian groups arises as the K-groups of a \ca.

Theorem~\ref{thmintro:CatCu} is important
since it grants us access
to categorical methods in the study of (abstract) 
Cuntz semigroups. This perspective has been extremely fruitful,
and some of the most recent applications are discussed in Section~\ref{sec:Axioms}; see also Theorem~\ref{thmintro:CatCu2} below.

As mentioned before, the ideal structure of $A$ can be read 
off of $\Cu(A)$: 

\begin{thmintro}(Ciuperca-Robert-Santiago \cite{CiuRobSan_ideals_2010}).
The Cuntz semigroup of a C$^*$-algebra encodes its
ideal-lattice structure, as well as the Cuntz semigroups
of every ideal and every quotient. More explicitly, the
assignment $I\mapsto \Cu(I)$ defines a lattice isomorphism between
the ideals of $A$ and the ideals of $\Cu(A)$, and 
there is a canonical isomorphism $\Cu(A)/\Cu(I)\cong \Cu(A/I)$.
\end{thmintro}

The above theorem should also be compared to analogous 
statements in K-theory:  
in general, the K-groups of $A$ do 
not contain any information
about the ideal structure of $A$. 

Another major part of the structure of a C$^*$-algebra
which is encoded in its Cuntz semigroup is its (quasi)tracial 
state space.

\begin{thmintro}\label{thmintro:QT} (Blackadar-Handelman \cite{BlaHan_dimension_1982}).
Let $A$ be a \uca. Then there is a natural affine bijection 
between the set of all
quasitracial states on $A$ and the set of all normalized
functionals on $\Cu(A)$. Given a quasitrace $\tau\in \QT(A)$, 
the corresponding functional is
\[d_\tau([a])=\lim_{n\to\I} \tau\big( a^{\frac{1}{n}}\big)\]
for all $a\in A_+$ (and extended naturally to positive
elements in $A\otimes\K$).
\end{thmintro}

By the work of Elliott, Robert and Santiago,
the natural bijection described in the theorem above extends
to a bijection between the set of all lower-semicontinuous quasitraces on $A$
and the set of all functionals on $\Cu(A)$; see
\cite{EllRobSan_cone_2011}.

The fact that a finite von Neumann factor admits a 
faithful trace is a fundamental result in the study of 
factors. In the C$^*$-algebra setting, Cuntz used a precursor of
Theorem~\ref{thmintro:QT} to show the following version of that
result:

\begin{thmintro}(Cuntz \cite{Cun_dimension_1978}).
Let $A$ be a simple, unital \ca. Then $A$ is stably finite 
if and only if it admits a faithful quasitrace.
\end{thmintro}

For the most part of the last two decades, classification has 
revolved around the Jiang-Su algebra $\mathcal{Z}$ and 
C$^*$-algebras that absorb it tensorially; such algebras
are called $\mathcal{Z}$-\emph{stable}. The main 
reason for this is the fact (see \autoref{rem:ElltensorZ}) that
the Elliott invariant cannot distinguish between $A$ and 
$A\otimes\mathcal{Z}$. Thus, only $\mathcal{Z}$-stable C$^*$-algebras
can be expected to be classified using Ell.
The computation of the Cuntz semigroup of
$\mathcal{Z}$ was used by R{\o}rdam to show that Cuntz 
semigroups of $\mathcal{Z}$-stable \ca s are well-behaved:

\begin{thmintro}\label{thmintro:Rordam}(R{\o}rdam \cite{Ror_stable_2004}).
Let $A$ be a separable, unital $\mathcal{Z}$-stable C$^*$-algebra. Then $\Cu(A)$ is almost
unperforated; equivalently, $A$ has strict comparison.
\end{thmintro}

Perhaps surprisingly, it is conjectured that the converse of Theorem~\ref{thmintro:Rordam} is true in the simple, nuclear setting: 
this is the only implication that remains
open in the Toms-Winter conjecture. The conditions of
$\mathcal{Z}$-stability and strict comparison 
should be regarded as
C$^*$-algebraic counterparts of McDuffness and the fact that in a
II$_1$-factor, the order on projections is determined by the 
unique trace. While those conditions are always satisfied for 
hyperfinite factors, $\mathcal{Z}$-stability and strict comparison are not automatic for simple, nuclear C$^*$-algebras,
as Toms showed:

\begin{thmintro}(Toms \cite{Tom_classification_2008}).
There exists a simple, separable, nuclear, unital C$^*$-algebra
$A$ which satisfies $\Ell(A)\cong \Ell(A\otimes\mathcal{Z})$ and
$\Cu(A)\ncong \Cu(A\otimes\mathcal{Z})$, so in particular $A\ncong A\otimes\mathcal{Z}$. 
\end{thmintro}

The way that Toms distinguished $\Cu(A)$ from $\Cu(A\otimes\mathcal{Z})$
was using the property of almost unperforation via Theorem~\ref{thmintro:Rordam}.
The explicit computation of $\Cu(A)$, for the C$^*$-algebra $A$
constructed by Toms, is still unknown; see \autoref{pbm:Cu(Toms)}.
On the other hand, the difficult task of computing Cuntz semigroups becomes much simpler for $\mathcal{Z}$-stable C$^*$-algebras:

\begin{thmintro}(Brown-Toms \cite{BroTom_three_2007}, Brown-Perera-Toms \cite{BroPerTom_cuntz_2008}).
\label{thm:BPT}
Let $A$ be a simple, separable, unital, stably finite $\mathcal{Z}$-stable 
C$^*$-algebra. Then  
\[\Cu(A)\cong \mathrm{V}(A) \sqcup 
\mathrm{LAff}(\QT(A))_{++}.\]
In particular, the pair $(\Cu(A),\mathrm{K}_1(A))$ 
is equivalent to $\Ell(A)$.
\end{thmintro}

In the theorem above, LAff(QT$(A))_{++}$ denotes the set of all
lower semicontinuous affine functions QT$(A)\to (0,\infty]$.
For the class of algebras in Theorem \ref{thm:BPT}, the group $\KK_1(A)$ can also be recovered from Cuntz semigroup data,
namely from $\Cu(A\otimes C(\mathbb{T}))$;
see \autoref{thm:zstable}.

A recurrent theme in these notes is the fact that C$^*$-algebras
of stable rank one have particularly well behaved Cuntz semigroups. 
This is made already apparent in Subsection~2.2, where we show that 
Cuntz comparison in the stable rank one case takes a form which is
very close to Murray-von Neumann subequivalence for projections;
see \autoref{prop:CtzCompsr1}. We also report on some of the most
recent results on Cuntz semigroups of C$^*$-algebras of stable rank one,
including the following: 

\begin{thmintro}\label{thmintro:CatCu2}(Antoine-Perera-Robert-Thiel
\cite{AntPerRobThi_stable_2022}).
Let $A$ be a separable, unital \ca\ of stable 
rank one without finite-dimensional quotients. 
Then $\Cu(A)$ has the Riesz
interpolation property and it is an inf-semillatice ordered
semigroup. Moreover:
\be[{\rm (i)}]\item (Realization of ranks). For every 
$f\in \mbox{LAff}(\QT(A))_{++}$, there exists $a\in (A\otimes\K)_+$ such that
$d_\tau([a])=f(\tau)$ for all $\tau\in \QT(A)$.
\item (Blackadar-Handelman conjecture). The space $\mathrm{DF}(A)$ of 
dimension functions on $A$ is a Choquet simplex.
\ee
\end{thmintro}

The developments that led to 
Theorem~\ref{thmintro:CatCu2} ran parallel to, and also 
benefited from, the advances made in the systematic study of the category $\CatCu$. Indeed,
Theorem~\ref{thmintro:CatCu} opened the doors for an abstract 
study of the category $\CatCu$ and of the functor $\Cu$. 
In this setting, it is particularly important to establish 
the existence of some standard categorical 
constructions in $\CatCu$
such as direct limits, tensor products or products. This naturally
leads one to consider two larger auxiliary categories 
\textbf{W} and \textbf{Q}, which contain \textbf{Cu} as a 
full subcategory. These larger categories have the advantage that 
the constructions that we are interested in exist in them (for example, 
it is not so difficult to 
prove that \textbf{W} has arbitrary direct limits
and tensor products).
One can show (see \autoref{thm:completion} and \autoref{prp:CuCoreflQ}) that there exist
functors $\gamma\colon \mathrm{\textbf{W}} \to \CatCu$ and 
$\tau\colon  \mathrm{\textbf{Q}} \to \CatCu$ which are,
respectively, a reflector and coreflector to
the canonical inclusions. Using these, it is possible
to transfer constructions from \textbf{W} and \textbf{Q} back to 
$\CatCu$. This is an essential ingredient in the following:

\begin{thmintro}(Antoine-Perera-Thiel
\cite{AntPerThi_tensor_2018}).
The category \textbf{Cu} is closed, symmetric, monoidal, and 
bicomplete. The Cuntz semigroup functor preserves directed limits and coproducts, and, suitably interpreted, also products and ultraproducts.
\end{thmintro}

We almost exclusively work with the picture of the Cuntz
semigroup of a \ca\ $A$ which uses positive elements in
$A\otimes\K$. We should, however, point out that there are at 
least two alternative presentations of $\Cu(A)$
which may be more convenient in some contexts: 
the one using Hilbert modules, which was first considered in 
\cite{CowEllIva_Cuntz_2008} and is 
only very briefly presented here in the comments before
\autoref{prop:key}, and the one using open projections,
which was given in \cite{OrtRorThi_cuntz_2011}, and is not 
covered here. 

We have made these notes as self-contained as possible, assuming only
a very elementary background in C$^*$-algebras and functional calculus.
As such, we hope that this survey will be useful to young mathematicians
who wish to learn the theory of Cuntz semigroups, as well as to those
researchers who want to see a streamlined presentation of the latest
results in the structure theory of Cuntz semigroups.
Even though we made an effort to cover a large variety of 
topics on Cuntz semigroups,
space constraints and the magnitude of the 
existent literature make it impossible to 
be exhaustive, and it was inevitable to make some 
omissions. On the other hand, 
there are two other introductions to the subject 
\cite{AraPerTom_survey_2011,Thi_notes}, with somewhat different approaches. 
Indeed, \cite{AraPerTom_survey_2011}, which was 
written over a decade ago, focuses primarily
on the classical (or uncompleted) Cuntz semigroup W, 
and makes extensive use of the Hilbert module picture. 
On the other hand, \cite{Thi_notes}, which
is more recent, regards Cu-semigroups as domains with additional
structure, and uses methods from lattice and category theory. 
Both references are good complements to this survey.

\vspace{0.2cm}

This manuscript is an expanded version of the lectures notes 
which were prepared as supporting material for the minicourse 
\emph{An introduction to the modern theory of Cuntz semigroups},
which took place at the University of Kiel in September 2022,
as part of the workshop \emph{Cuntz semigroups}. The authors would
like to thank Hannes Thiel for the invitation to deliver this 
course, and to both Ramon Antoine and Hannes Thiel for several discussions on these topics. Thanks are also extended to Leonel Robert and Andrew Toms for useful comments after a first version was circulated.

\section{Comparison of positive elements}

The following is the definition on which the theory of Cuntz
semigroups is based.

\begin{df}\label{df:CuSubeq}
Let $A$ be a C$^*$-algebra and let $a,b\in A_+$. We say that $a$ is
\emph{Cuntz subequivalent} to $b$ (in $A$), denoted $a\precsim b$ 
(or $a\precsim_A b$ if there is a need to specify the ambient 
C$^*$-algebra), if there exists a sequence $(r_n)_{n\in\N}$ in $A$
such that $\lim\limits_{n\to\I} \|r_nbr_n^*-a\|=0$. 
Equivalently, for every $\ep>0$ there exists $r\in A$ such that
$\|rbr^*-a\|<\ep$.

We say that $a$ and $b$ are \emph{Cuntz equivalent}, written 
$a\sim b$, if $a\precsim b$ and $b\precsim a$.
\end{df}

It is easy to see that 
if $a\precsim b$ and $b\precsim c$, then $a\precsim c$.

In order to give a feeling for the relation $\precsim$, we 
first look at the case of commutative \ca s. For $f\in C(X)$,
we denote its \emph{open support} by 
\[\supp_{\mathrm{o}}(f)=\{x\in X\colon f(x)\neq 0\}.\]

\begin{prop}\label{prop:CtzCompCX}
Let $X$ be a locally compact, Hausdorff space, and let $f,g\in C_0(X)$
be positive functions. Then $f\precsim g$ if and only if 
\[\supp_{\mathrm{o}}(f)\subseteq \supp_{\mathrm{o}}(g).\]
Equivalently, $g(x)=0$ implies $f(x)=0$, for all $x\in X$.
\end{prop}
\begin{proof}
Assuming that $f\precsim g$, choose a sequence $(r_n)_{n\in\N}$ in
$C_0(X)$ such that $\lim\limits_{n\to\I} r_ngr_n^*=f$ uniformly on $X$.
For $x\in X$, it is then 
easy to see that $g(x)=0$ implies $f(x)=0$.

Conversely, suppose that 
$\supp_{\mathrm{o}}(f)\subseteq \supp_{\mathrm{o}}(g)$
and let $\ep>0$. Set $K=\{x\in X\colon f(x)\geq \ep\}$, which is a 
compact set on which $g$ is strictly positive. Find $\delta>0$ such
that $g(x)\geq \delta$ for all $x\in K$, and set 
\[U=\Big\{x\in X\colon g(x)>\tfrac{\delta}{2}\Big\}.\]
Then $U$ is open in $X$ and $K\subseteq U$. By Urysohn's lemma, 
there is a positive function $h\in C_0(X)$ which is identically 
equal to 1 on $K$ and
vanishes precisely outside of $U$. Define $s\in C_0(X)$
by 
\[s(x)=\begin{cases}
        \frac{f(x)}{g(x)} h(x), &\mbox{ if } x\in U;\\
        0, &\mbox{ else}.
       \end{cases}
\]
Setting $r=s^{\frac{1}{2}}$, we get $r=r^*$ and $\|f-rgr\|<\ep$, as desired. 
\end{proof}

For an element $a$ in a \ca\ $A$, we denote by $\spec(a)$ its spectrum. We now extract some very useful consequences of \autoref{prop:CtzCompCX}:

\begin{cor}\label{cor:FuncCalc}
Let $A$ be a \ca, let $a\in A_+$ and let $f\colon [0,\I)\to [0,\I)$
be continuous and satisfy $f(0)=0$. Then:
\be[{\rm (i)}]
\item We have $f(a)\precsim a$. 
\item If $f(t)>0$ for $t>0$, then $f(a)\sim a$.
\item We have $a\sim a^\lambda$ and $a\sim \lambda a$ for $\lambda\in (0,\infty)$.
\item For $x\in A$, we have $x^*x\sim xx^*$.
\ee
\end{cor}
\begin{proof} 
(i) Note that $a$ and $f(a)$ belong to $C^*(a)\cong C_0(\spec(a))$,
and under this identification these elements correspond to 
$\id_{\spec(a)}$ and $f|_{\spec(a)}$, respectively. The conclusion
then follows from \autoref{prop:CtzCompCX} since $x=0$ implies $f(x)=0$. 

(ii) Follows from (i) by taking $f^{-1}$.

(iii) Follows from (ii) by taking either $f(t)=t^\lambda$
or $f(t)=\lambda t$, respectively.

(iv) We have
\[xx^*\stackrel{\mathrm{(iii)}}{\sim} (xx^*)^2 = x(x^*x) x^*\precsim x^*x.\]
Analogously, we get $x^*x\precsim xx^*$, as desired.
\end{proof}

Recall that a subalgebra $B$ of a \ca\ $A$ is said to be 
\emph{hereditary} if whenever $b\in B_+$ and $a\in A_+$ satisfy $a\leq b$, then $a\in B$.
Given $b\in A_+$, the smallest hereditary subalgebra of $A$ containing
$b$ is $A_b:=\overline{bAb}$. We will use, without proof, the fact that
$(b^{\frac{1}{n}})_{n\in \N}$ is an approximate identity for $A_b$. For arbitrary elements $a,b$ in a C$^*$-algebra $A$, we shall as customary write $a\approx_\varepsilon b$ to mean that $\Vert a-b\Vert<\varepsilon$.

\begin{prop}\label{prop:InHerSubalg}
Let $A$ be a \ca\ and let $a,b\in A_+$. If $a\in A_b$, then 
$a\precsim b$. In particular, if $a\leq b$ then $a\precsim b$.
\end{prop}
\begin{proof}
Suppose that $a\in A_b$, and let $\ep>0$. 
Choose $n\in\N$ such that $a\approx_{\frac{\ep}{2}} b^{\frac{1}{n}}ab^{\frac{1}{n}}$.
Using part~(iv) of \autoref{cor:FuncCalc} at the second step, and 
part~(iii) of at the fourth step, we get
\[b^{\frac{1}{n}}ab^{\frac{1}{n}}=(b^{\frac{1}{n}}a^{\frac{1}{2}})(a^{\frac{1}{2}}b^{\frac{1}{n}})\sim a^{\frac{1}{2}}b^{\frac{2}{n}}a^{\frac{1}{2}}\precsim b^{\frac{2}{n}}\sim b.\]
Choose $r\in A$ with $rbr^*\approx_{\frac{\ep}{2}} b^{\frac{1}{n}}ab^{\frac{1}{n}}$.
Then 
\[a\approx_{\frac{\ep}{2}} b^{\frac{1}{n}}ab^{\frac{1}{n}} \approx_{\frac{\ep}{2}} rbr^*,                                      \]
and hence $a\precsim b$, as desired.
\end{proof}


Given $\ep>0$, let $f_\ep\colon [0,\infty)\to [0,\I)$ be given
by $f_\ep(t)=\max\{t-\ep,0\}$. Given $a\in A_+$, we write
$(a-\ep)_+$ for $f_\ep(a)$. The element $(a-\ep)_+$ is usually
referred to as the \emph{$\ep$ cut-down} of $a$. Since $(a-\varepsilon)_+\leq a$ for any $\varepsilon>0$, we get from \autoref{prop:InHerSubalg} that $(a-\varepsilon)_+\precsim a$.
The following lemma is a generalization of this fact.

\begin{lma}\label{lma:CutDownDistance}
Let $A$ be a \ca, let $\ep>0$, and let $a,b\in A_+$ with 
$\|a-b\|<\ep$. Then there exists a contraction 
$r\in A$ with $rbr^*=(a-\ep)_+$.
In particular, $(a-\ep)_+\precsim b$.
\end{lma}
\begin{proof} 
We only prove the last assertion; the first one is rather
involved, and a proof can be found in \cite[Lemma 2.2]{KirRor_infinite_2002}.

Note that $\|a-b\|<\ep$ implies 
that $a-\ep\leq b$. Multiplying on both sides
by $(a-\ep)_+$, we get
\begin{align*}\tag{2.1}(a-\ep)_+(a-\ep)(a-\ep)_+\leq (a-\ep)_+b(a-\ep)_+.
\end{align*}
Using part~(iii) of \autoref{cor:FuncCalc} at the first step, we get 
\[(a-\ep)_+\sim (a-\ep)_+^3=
(a-\ep)_+(a-\ep)(a-\ep)_+\stackrel{(2.1)}{\leq} (a-\ep)_+b(a-\ep)_+
\precsim b,
\]
as desired.
\end{proof}

The reader is encouraged to prove the first assertion in
\autoref{lma:CutDownDistance} in the case where $a$ and $b$ commute. 


We will sometimes need the following strengthening of 
part~(iv) of \autoref{cor:FuncCalc}. We omit its proof, which
can be found in \cite[Corollary 2.5]{Thi_notes}.

\begin{lma}\label{lma:CutDwnSymm} 
Let $A$ be a \ca, let $x\in A$ and let $\ep>0$. Then 
\[(x^*x-\ep)_+\sim (xx^*-\ep)_+.\]
\end{lma}

The following is one of the most used technical results 
about Cuntz comparison; see \cite[Proposition 2.4]{Ror_structure_1992}.

\begin{thm}\label{thm:Rordam} (R{\o}rdam's lemma.)
Let $A$ be a \ca\ and let $a,b\in A_+$. Then the following are equivalent:
\be[{\rm (i)}]\item $a\precsim b$;
\item For every $\ep>0$ we have $(a-\ep)_+\precsim b$;
\item For every $\ep>0$ there exists $\delta>0$ such that
$(a-\ep)_+\precsim (b-\delta)_+$;
\item For every $\ep>0$ there exist $\delta>0$ and $x\in A$ such that
$(a-\ep)_+=x^*x$ and $xx^*\in A_{(b-\delta)_+}$;
\item For every $\ep>0$ there exist $\delta>0$ and $r\in A$ such that
$(a-\ep)_+=r (b-\delta)_+r^*$.
\ee
\end{thm}
\begin{proof}
(i) implies (ii) since $(a-\varepsilon)_+\precsim a$, as observed above. Conversely, let $\ep>0$.
Since $(a-\tfrac{\ep}{2})_+\precsim b$, there is $r\in A$ with
\[rbr^*\approx_{\frac{\ep}{2}} \big(a-\tfrac{\ep}{2}\big)_+
\approx_{\frac{\ep}{2}} a,
\]
so (ii) implies (i). 
That (iii) implies (ii) follows from $(b-\delta)_+\leq b$.

We now show that (iv) implies (iii). Let $\ep>0$ and find $\delta>0$ and 
$x\in A$ with $(a-\ep)_+=x^*x$ and $xx^*\in A_{(b-\delta)_+}$.
Using part~(iv) of \autoref{cor:FuncCalc} at the second step and 
using \autoref{prop:InHerSubalg} at the third step, we get
\[(a-\ep)_+=x^*x\sim xx^*\precsim (b-\delta)_+,\]
as desired. To show that (v) implies (iv), let $\ep>0$ and find 
$\delta>0$ and $r\in A$ with $(a-\ep)_+=r (b-\delta)_+r^*$.
Set $x=(b-\delta)_+^{\frac{1}{2}}r^*$. Then 
\[(a-\ep)_+=x^*x \ \ \mbox{ and } \ \ xx^*=(b-\delta)^{\frac{1}{2}}r^*r(b-\delta)^{\frac{1}{2}}\in A_{(b-\delta)_+},\]
as desired.
Finally, we show that (i) implies (v). Let $\ep>0$ and choose $c\in A$
with $cbc^*\approx_{\frac{\ep}{2}} a$. By continuity of functional calculus, there
is $\delta>0$ such that 
$cbc^*\approx_{\frac{\ep}{2}} c(b-\delta)_+c^*$. By
\autoref{lma:CutDownDistance}, there is $d\in A$ such that 
$(a-\ep)_+=dc(b-\delta)_+c^*d^*$. Then $r=dc$ satisfies 
$(a-\ep)_+=r(b-\delta)_+r^*$.
\end{proof}

We will later want to compute the Cuntz semigroup of $\C$, for which we will
need the following useful lemma. Recall that projections $p,q$ in a C$^*$-algebra $A$ are said to be \emph{Murray-von Neumann equivalent}, in symbols $p\sim_\mathrm{MvN}q$, provided there is a partial isometry $v\in A$ such that $p=v^*v$ and $q=vv^*$. We also say that $p$ is \emph{Murray-von Neumann subequivalent} to $q$ if there is a projection $p'\leq q$ such that $p\sim_\mathrm{MvN}p'$.

\begin{lma}\label{lma:CtzCompPjns}
Let $A$ be a \ca\ and let $p,q\in A$ be projections. Then
$p\precsim q$ if and only if $p\precsim_{\mathrm{MvN}} q$. 
\end{lma}
\begin{proof} 
It is clear that $p\precsim_{\mathrm{MvN}} q$ implies $p\precsim q$.
Conversely, assume that $p\precsim q$. For $\ep=\frac{1}{2}$, use 
part~(iv) of R{\o}rdam's lemma (\autoref{thm:Rordam}) to find 
$\delta>0$ and $x\in A$ such that $x^*x=(p-\frac{1}{2})_+$ and 
$xx^*\in A_{(q-\delta)_+}$. Since $p$ and $q$ are projections, we
have $(p-\frac{1}{2})_+=\frac{1}{2}p$ and similarly 
$(q-\delta)_+=(1-\delta)q\sim q$. In particular, 
$A_{(q-\delta)_+}=A_q=qAq$.
It follows that 
$v=\sqrt{2}x$ is a partial isometry satisfying $v^*v=p$
and $vv^*\in qAq$, so $vv^*\leq q$ as desired.
\end{proof}

Note, however, that the previous lemma does not imply that 
$p\sim q$ if and only if $p\sim_{\mathrm{MvN}} q$, since in general
$p\sim_{\mathrm{MvN}} q$ is not equivalent to 
$p\precsim_{\mathrm{MvN}} q\precsim_{\mathrm{MvN}} p$.
This \emph{is} the case if $A$ is stably finite\footnote{Recall that a \ca\ $A$ is said to be \emph{finite} if
every partial isometry in $A$ is a unitary. Moreover, $A$ is said to 
be \emph{stably finite} if $M_n(A)$ is finite for all $n\in\N$.
Equivalently, no projection in $M_n(A)$ is equivalent to a proper subprojection of itself. Now, $p$ and $q$ are projections in some 
matrix algebra over $A$ satisfying $p\precsim_{\mathrm{MvN}} q\precsim_{\mathrm{MvN}} p$, implemented say as $v^*v=p$ 
and $vv^*\leq q$, and $w^*w=q$ and $ww^*\leq p$, then
it follows that $w^*v^*vw=p$ and $vww^*v^*\leq p$. By stable finiteness, we must have $vww^*v^*=p$,
which implies that $vv^*=q$ and $ww^*=q$, so $p\sim_{\mathrm{MvN}}q$.}.

\subsection{Cuntz comparison and stable rank one}
C$^*$-algebras of stable rank one play an important role in these notes, since their Cuntz semigroups are particularly
well-behaved. In this subsection, we show that Cuntz
comparison in C$^*$-algebras of stable rank one takes a form 
which is very similar to Murray-von Neumann subequivalence
for projections; see \autoref{prop:CtzCompsr1}.

We begin with a general discussion for the reader to develop
some intuition around the notion of stable rank one.
For a C$^*$-algebra $A$, its \emph{minimal unitization} $\widetilde{A}$ is defined to be $A$ when $A$ is unital, and 
$A\oplus\C$ otherwise. In the latter case, addition is defined componentwise whilst the product is given by $(a,\lambda)(b,\mu)=(ab+\mu a+\lambda b, \lambda\mu)$, so that $(0,1)$ becomes the unit and $A$ sits inside $\widetilde{A}$ as a closed, two-sided ideal.

\begin{df}\label{df:sr1}
A unital C$^*$-algebra $A$ is said to have \emph{stable rank one} if the set $\mathrm{GL}(A)$ of invertible elements in $A$ is dense in $A$. Moreover, a nonunital
C$^*$-algebra is said to have stable rank one if its minimal unitization does.
\end{df} 

The class of C$^*$-algebras with stable rank one is pleasantly large. It was shown by R\o rdam in \cite{Ror_stable_2004} that a unital, simple, stably finite and $\mathcal{Z}$-stable C$^*$-algebra has stable rank one, hence this applies to all classifiable stably finite C$^*$-algebras. 
Stable rank one also holds outside of the classifiable class, 
for example for crossed products of the form $C(X)\rtimes\Z^n$ for 
a free, minimal action of $\Z^n$ on a compact Hausdorff space $X$, 
regardless of whether $C(X)\rtimes\Z^n$ is $\mathcal{Z}$-stable or 
not; see \cite{GuaNiu_stable_2020}. (See also \autoref{sec:JS} below for more information on the C$^*$-algebra $\mathcal{Z}$.)

As the name indicates, the stable rank is an integer valued number that is associated to any C$^*$-algebra and can be with advantage thought of as a noncommutative dimension theory. In fact, one has that $\mathrm{sr}(C(X))=\big[\frac{\dim(X)}{2}\big]+1$; see \cite{Rie_dimension_1983}. The notion of stable rank in \cite{Rie_dimension_1983} agrees with that of Bass, as shown in \cite{HerVas_stable_1984}.

We record here some general facts about stable rank one
that will be needed throughout.

\begin{prop}\label{prop:sr1}
Let $A$ be a C$^*$-algebra of stable rank one.
\be
\item[(i)] $A$ is automatically stably finite.
\item[(ii)] If $B\subseteq A$ is a hereditary subalgebra,
then $B$ has stable rank one.
\item[(iii)] $A\otimes\K$ has stable rank one.
\item[(iv)] If $p,q\in A$ satisfy $p\sim_{\mathrm{MvN}} q$, then there is a unitary $u\in \widetilde{A}$
satisfying $p=uqu^*$.
\ee
\end{prop}

We need a technical lemma; see 
\cite[Lemma~2.4]{CiuEllSan_inductive_2011} for the proof.

\begin{lma}\label{lma:CES}
Let $A$ be a \ca\ and let $B\subseteq A$ be a hereditary
subalgebra of stable rank one. Let $\delta>0$ 
and let $x,y\in A$ satisfy 
\[xx^*,yy^*\in B, \ \ \ x^*x\in A_{y^*y} \ \ \ 
\mbox{ and } \ \ \ \|x^*x-y^*y\|<\delta.\]
Then there is a unitary $u\in \widetilde{B}$ such that
$\|x-uy\|<\sqrt{\delta}$.
\end{lma}

The following is one of the basic technical results
that makes the study of Cuntz comparison in \ca s of 
stable rank one particularly accessible.
(Another one will be given in \autoref{thm:StableUnitImpl}.)

\begin{prop}\label{prop:CtzCompsr1}
(\cite[Proposition~2.5]{CiuEllSan_inductive_2011}).
Let $A$ be a \ca\ of stable rank one, and let 
$a,b\in A_+$. Then $a\precsim b$ if and only if 
there exists $x\in A$ satisfying $x^*x=a$ and $xx^*\in A_b$.
\end{prop}
\begin{proof}
If there exists $x\in A$ as in the statement, then
part~(iv) of \autoref{cor:FuncCalc} gives
\[a=x^*x\sim xx^*\in A_b,\]
and thus $a\precsim b$ by \autoref{prop:InHerSubalg}.

Conversely, assume that $a\precsim b$.
By R{\o}rdam's lemma (\autoref{thm:Rordam}), 
for every $n\in\N$ there exists $y_n\in A$ such that
\[\big(a-\tfrac{1}{2^{2n}}\big)_+=y_n^*y_n \ \ \mbox{ and } \ \ 
 y_ny_n^*\in A_b.
\]
Given $n\in\N$, apply \autoref{lma:CES} to $y_n$ and 
$y_{n+1}$ to obtain a unitary $u_n$ satisfying
$\|y_n-u_ny_{n+1}\|<\tfrac{1}{2^n}$.
Set $x_n=u_1\cdots u_{n-1}y_n$. Then $\|x_n-x_{n+1}\|<\tfrac{1}{2^n}$, and thus the sequence $(x_n)_{n\in\N}$
has a limit $x\in A$. Moreover,
\[x^*x=\lim_{n\to\I}x_n^*x_n=\lim_{n\to\I}\big(a-\tfrac{1}{2^{2n}}\big)_+=a,\]
and similarly $xx^*=\lim\limits_{n\to\I} x_nx_n^*\in A_b$, as desired.
\end{proof}

\section{The Cuntz semigroup}
The following is the object we will study in these notes.

\begin{df}\label{df:CuA}
Let $A$ be a \ca. The \emph{Cuntz semigroup} of $A$ is defined as 
\[\Cu(A)=(A\otimes\K)_+/\!\sim,\]
where $\sim$ stands for the Cuntz equivalence relation. For a positive element $a\in (A\otimes\K)_+$, we denote by $[a]$ its Cuntz equivalence class, hence $\Cu(A)=\big\{[a]\colon a\in (A\otimes\K)_+\big\}$.
There is a natural partial order defined on $\Cu(A)$, namely $[a]\leq [b]$ if $a\precsim b$. (The element $0:=[0]$ is the minimal element in $\Cu(A)$.) We define an addition on $\Cu(A)$ by setting\footnote{Here, just like it is done in K-theory, 
we are implicitly fixing an isomorphism $M_2\otimes\K\cong \K$ and
using it to identify $\Cu(A)$ with $\Cu(M_2(A))$; one can check that
this identification does not depend on the isomorphism we fixed.}
\[[a]+[b]=\big[\begin{psmallmatrix}a & 0\\0 & b\end{psmallmatrix}\big]
.\]
\end{df}

We also denote $\begin{psmallmatrix}a & 0\\0 & b\end{psmallmatrix}$ by $a\oplus b$. 

\begin{lma}
\label{lma:orth} Let $A$ be a C$^*$-algebra, and let $a,b\in A_+$. Then $a+b\precsim a\oplus b$. If $ab=0$, then $a+b\sim a\oplus b$.
\end{lma}
\begin{proof}
Note that 
\[
\begin{psmallmatrix} (aa^*)^{\frac{1}{n}} & (bb^*)^{\frac{1}{n}}\\0 & 0\end{psmallmatrix}\Big(\begin{smallmatrix}a & 0\\0 & b\end{smallmatrix}\Big)\begin{psmallmatrix} (a^*a)^{\frac{1}{n}} & 0\\(b^*b)^{\frac{1}{n}} & 0\end{psmallmatrix}\to \begin{psmallmatrix}a+b & 0\\0 & 0\end{psmallmatrix},
\] 
hence $a+b\sim \begin{psmallmatrix}a+b & 0\\0 & 0\end{psmallmatrix}\precsim a\oplus b$. On the other hand, if $ab=0$ and we let $x=\begin{psmallmatrix}a^{\frac{1}{2}} & b^{\frac{1}{2}}\end{psmallmatrix}$, then $xx^*=a+b$ and $x^*x=\begin{psmallmatrix}a & 0\\0 & b\end{psmallmatrix}$.
\end{proof}

 It is not difficult to check that $[a]+[b]=[b]+[a]$ and 
that $[a]+0=[a]$ in $\Cu(A)$ for all $[a], [b]\in \Cu(A)$. More interestingly, addition and order
are compatible in $\Cu(A)$, in the sense that $[a_1]\leq [b_1]$
and $[a_2]\leq [b_2]$ imply $[a_1]+[a_2]\leq [b_1]+[b_2]$.
This gives $\Cu(A)$ the structure of a positively ordered monoid\footnote{A \emph{monoid} is a semigroup with a neutral element, and it is said to be \emph{positively ordered} if every
element dominates the neutral element. The reasons for having termed $\Cu(A)$ a semigroup are only historical.}.

We turn to the first computation of a Cuntz semigroup.

\begin{eg}\label{eg:CuC}
Let us compute $\Cu(\C)\cong \Cu(M_n)\cong \Cu(\K)$.
We will show that the rank map 
$\mathrm{rk}\colon \K_+\to \{0,1,\ldots,\infty\}=:\overline{\N}$
induces an ordered semigroup isomorphism 
\[\mathrm{rk}\colon \Cu(\K)\to \overline{\N}.\]
To show this, given $a,b\in \K_+$ we will prove that
$\mathrm{rk}(a)\leq \mathrm{rk}(b)$ if and only if 
$a\precsim b$. 

Let $a\in \K_+$. By the Spectral Theorem, there are scalars $\lambda_n\geq 0$ and finite rank projections $p_n\in\K$ for 
$n\in\N$, such that 
\[a=\sum_{n=0}^\infty \lambda_np_n.\]
Moreover, $\mathrm{rk}(a)<\infty$ if and only if there is $m$ such that $\lambda_n=0$
for all $n\geq m$. In this case, part~(ii) of
\autoref{cor:FuncCalc} implies that $a\sim \sum_{n=1}^m p_n=:p_a$.
Note that $\mathrm{rk}(p_a)=\mathrm{rk}(a)$, and recall that Murray-von
Neumann subequivalence for projections in $\K$ is determined by 
the rank. 
Using \autoref{lma:CtzCompPjns} at the second step, for finite rank elements $a,b\in \K_+$
we have
\[a\precsim b  \ \ \Leftrightarrow \ \ p_a\precsim p_b \ \
\Leftrightarrow \ \ p_a\precsim_{\mathrm{MvN}} p_b  \ \ 
\Leftrightarrow \ \ \mathrm{rk}(p_a)\precsim \mathrm{rk}(p_b) \ \ 
\Leftrightarrow \ \ \mathrm{rk}(a)\precsim \mathrm{rk}(b),\]
as desired. 
Now, if $\mathrm{rk}(b)=\infty$ and $a\in\K_+$ is arbitrary, we will
show that $a\precsim b$. By part~(ii) of R{\o}rdam's lemma
(\autoref{thm:Rordam}), it suffices to show that $(a-\ep)_+\precsim b$
for all $\ep>0$. Note that $(a-\ep)_+$ has finite rank and is therefore
Cuntz equivalent to a finite-rank projection $p\in\K$, by
the previous part.
Since $A_b$ contains projections of arbitrary large rank, there is 
$q\in A_b$ with $p\sim_{\mathrm{MvN}}q$. Then 
\[(a-\ep)_+ \sim p \precsim_{\mathrm{MvN}}q \precsim b,\]
which implies the result by \autoref{lma:CtzCompPjns}.
\end{eg}

We isolate the following convenient corollary. Recall that an \emph{order-embedding}\footnote{This is stronger than being order-preserving and injective.} $\phi\colon S\to T$
between ordered sets is a map satisfying $\phi(s)\leq \phi(s')$
if and only if $s\leq s'$.

\begin{cor}\label{cor:VACuA}
Let $A$ be a \ca. Then there is a natural semigroup
map $\iota\colon \V(A)\to \Cu(A)$. If $A$ is stably finite, 
then $\iota$ is an order embedding.
\end{cor}
\begin{proof} The first part is obvious since every projection is 
a positive element. If $A$ is stably finite, then 
the comments above show that $\V(A)$ is an ordered semigroup and 
the last claim follows immediately from 
\autoref{lma:CtzCompPjns}.
\end{proof}

\begin{rem} If $A$ is stably finite, then the 
 image of $\iota$ in $\Cu(A)$ is precisely the set of the so-called \emph{compact
 elements} of $\Cu(A)$; see \autoref{df:CompCont} and \autoref{rem:CuCAR}.
\end{rem}

\begin{nota} Given $a,b\in A_+$, we introduce the following notation:
\bi\item $a\sim_u b$ if there is a unitary $u\in \widetilde{A}$ with $uau^*=b$;
\item $a\subseteq b$ if $a\in A_b$;
\item $a\subseteq_u b$ if there is a unitary $u\in \widetilde{A}$ such that 
$uau^*\in A_b$.
\ei
\end{nota}

The following theorem shows that Cuntz comparison in stable 
\ca s is \emph{unitarily implemented}. The result holds more
generally for \ca s of \emph{weak stable rank one}: by definition, a C$^*$-algebra $A$ has weak stable rank one if $A\subseteq \overline{\mathrm{GL}(\widetilde{A})}$. By 
\cite[Lemma 4.3.2]{BlaRobTikTomWin_algebraic_2012} every stable \ca\ has weak stable rank one. 

\begin{thm}\label{thm:StableUnitImpl}
Let $A$ be a \ca\ with weak stable rank one 
and let $a,b\in A_+$. Then
$a\precsim b$ if and only if for every $\ep>0$ 
we have $(a-\ep)_+\subseteq_u b$.
\end{thm}
\begin{proof} 
Note that the ``only if'' implication is 
 true in full generality: indeed, given $\ep>0$, find $u\in\U(\widetilde{A})$ as in the statement. With $x=u(a-\ep)_+^{\frac{1}{2}}\in A$, we use part~(iv) of \autoref{cor:FuncCalc} at the second step, and \autoref{prop:InHerSubalg} at the last one to get 
\[ (a-\ep)_+=x^*x\sim xx^*= u(a-\ep)_+u^* \precsim b.\]
Then $a\precsim b$ by R{\o}rdam's lemma (\autoref{thm:Rordam}).

We only sketch the proof of the converse, so assume that 
$A$ has weak stable rank one.
For $c,d\in A_+$, we write $c\sim_u d$ if there is a unitary 
$u\in\U(\widetilde{A})$ such that $c=udu^*$. 
From this, one shows that for every $x\in A$ and 
every $\ep>0$ we have
\begin{align*}\tag{2.2}
(x^*x-\ep)_+\sim_u(xx^*-\ep)_+.
\end{align*}
(Note that this is a strengthening of \autoref{lma:CutDwnSymm}.)
Assume now that $a\precsim b$ and let $\ep>0$. By R{\o}rdam's 
lemma (\autoref{thm:Rordam}), there exists $x\in A$ such that 
\[(a-\tfrac{\ep}{2})_+=x^*x \ \ \mbox{ and } \ \ xx^*\in A_b.\]
Then
\[(a-\ep)_+=(x^*x-\tfrac{\ep}{2})_+\stackrel{(2.2)}{\sim_u}(xx^*-\tfrac{\ep}{2})_+\in A_{b}.\]
Therefore there is $u\in\U(\widetilde{A})$ with $u(a-\ep)_+u^*\in A_b$, as desired.
\end{proof}

The following theorem, originally obtained by Coward, Elliott
and Ivanescu in \cite{CowEllIva_Cuntz_2008} using Hilbert C$^*$-modules,
was the first result about the internal structure of 
Cuntz semigroups, and shows that they are rather special
ordered semigroups.

\begin{thm}\label{thm:suprema}
Let $A$ be a \ca. Then every increasing sequence in $\Cu(A)$ has
a supremum. 
\end{thm}
\begin{proof} Let $(a_n)_{n\in\N}$ be a sequence in 
$(A\otimes\K)_+$ satisfying $a_n\precsim a_{n+1}$ for all $n\in\N$.

\textbf{Case 1:}\emph{ $(a_n)_{n\in\N}$ is increasing in $A\otimes\K$
and has a limit $a=\lim\limits_{n\to\I} a_n$. Then $[a]$ is the supremum of $([a_n])_{n\in\N}$.} In this case
$[a]\in \Cu(A)$ is an upper bound for $([a_n])_{n\in\N}$. Let 
$[b]\in\Cu(A)$ is another upper bound; we want to show that
$[a]\leq [b]$. By R{\o}rdam's lemma, it suffices to show that
for every $\ep>0$ we have $(a-\ep)_+\precsim b$. For the $\ep>0$
given, find $n\in\N$ such that $\|a-a_n\|<\ep$, so that we
have $(a-\ep)_+\precsim a_n$ by \autoref{lma:CutDownDistance}.
Then $(a-\ep)_+\precsim a_n\precsim b$, as desired.

\textbf{Case 2:}\emph{ for every $n\in\N$ we have $a_n\subseteq a_{n+1}$ (which means $a_n\in A_{a_{n+1}}$).} In this case, for $n\in\N$ we set 
\[b_n=\sum_{k=1}^n \frac{a_k}{\|a_k\| 2^k} \ \ \mbox{ and } \ \ 
 b=\sum_{k=1}^\I \frac{a_k}{\|a_k\| 2^k}.
\]
It is clear that $b_n\leq b_{n+1}$ and that $b=\lim\limits_{n\to\I}b_n$. Note that 
\[a_n\sim \frac{a_n}{\|a_n\|2^n} \leq b_n,\]
and hence $a_n\precsim b_n$ by \autoref{prop:InHerSubalg}. 
On the other hand,
since $a_1,\ldots,a_n$ belong to the hereditary subalgebra generated by $a_n$, the same is true for $b_n$. Thus 
$b_n\precsim a_n$ by \autoref{prop:InHerSubalg} and 
therefore $a_n\sim b_n$. It follows from Case~1 that $[b]$ is 
the supremum of $([a_n])_{n\in\N}$. 

\textbf{Case 3:}\emph{ for every $n\in\N$ we have $a_n\subseteq_u a_{n+1}$.} 
This case is easy to reduce 
to the previous one.
For each $n\in\N$, choose $u_n\in\U(\widetilde{A\otimes\K})$
such that $u_na_nu_n^*\in 
A_{a_{n+1}}$. Set $b_1=a_1$ and 
\[b_n=u_1^*u_2^*\cdots u_n^* a_n u_n \cdots u_2u_1.\]
Then $b_n\sim_u a_n$, so $[b_n]=[a_n]$. One readily checks that 
$b_n\in A_{b_{n+1}}$, and hence Case~2 implies the result in this case. 

\textbf{Case 4:}\emph{ the sequence $(a_n)_{n\in\N}$ is arbitrary.}
Using R{\o}rdam's lemma repeatedly together with 
\autoref{thm:StableUnitImpl}, for every $n\in\N$ we can find a 
sequence $(\ep_k^{(n)})_{k\in\N}$ which decreases to zero and 
such that 
\[(a_n-\ep_{k}^{(n)})_+\subseteq_u (a_{n+1}-\ep_{k}^{(n+1)})_+\]
for all $k,n\in\N$. We represent this graphically as follows:
\begin{align*}\xymatrix{
a_1  \ar@{}[r]|{\precsim} & a_2 \ar@{}[r]|{\precsim} & a_3 \ar@{}[r]|{\precsim} & \cdots \\
\vdots  \ar@{}[u]|{\rotatebox[origin=c]{90}{$\leq$}}& 
\vdots  \ar@{}[u]|{\rotatebox[origin=c]{90}{$\leq$}}& 
\vdots  \ar@{}[u]|{\rotatebox[origin=c]{90}{$\leq$}}& 
\cdots\\
(a_1-\ep^{(1)}_3)_+  \ar@{}[u]|{\rotatebox[origin=c]{90}{$\leq$}}\ar@{}[r]|{\subseteq_u} & 
(a_2-\ep^{(2)}_3)_+ \ar@{}[u]|{\rotatebox[origin=c]{90}{$\leq$}} \ar@{}[r]|{\subseteq_u} & 
(a_3-\ep^{(3)}_3)_+ \ \ \  \ar@{}[u]|{\rotatebox[origin=c]{90}{$\leq$}}\ar@{}[r]|{\subseteq_u} & \cdots\\
(a_1-\ep^{(1)}_2)_+  \ar@{}[u]|{\rotatebox[origin=c]{90}{$\leq$}}\ar@{}[r]|{\subseteq_u} & 
(a_2-\ep^{(2)}_2)_+ \ar@{}[u]|{\rotatebox[origin=c]{90}{$\leq$}} \ar@{}[r]|{\subseteq_u} & 
(a_3-\ep^{(3)}_2)_+ \ \ \  \ar@{}[u]|{\rotatebox[origin=c]{90}{$\leq$}}\ar@{}[r]|{\subseteq_u} & \cdots\\
(a_1-\ep^{(1)}_1)_+  \ar@{}[u]|{\rotatebox[origin=c]{90}{$\leq$}}\ar@{}[r]|{\subseteq_u} & 
(a_2-\ep^{(2)}_1)_+ \ar@{}[u]|{\rotatebox[origin=c]{90}{$\leq$}} \ar@{}[r]|{\subseteq_u} & 
(a_3-\ep^{(3)}_1)_+ \ \ \ \ar@{}[u]|{\rotatebox[origin=c]{90}{$\leq$}}\ar@{}[r]|{\subseteq_u} & \cdots
}\end{align*}

Set $b_n=(a_n-\ep^{(n)}_n)_+$. Then $b_n\subseteq_u b_{n+1}$ for all
$n\in\N$, and by Case~3 the supremum of $([b_n])_{n\in\N}$ 
exists in $\Cu(A)$, say $[b]$. Then $[a_n]\leq [b]$, 
since 
\[[a_n]\stackrel{\textrm{\textbf{Case~1}}}{=}\sup\limits_{k\in\N} [(a_n-\ep^{(n)}_k)_+]\ \leq \ \sup\limits_{k\in\N} [(a_k-\ep^{(k)}_k)_+] \ =\ [b].
\]
(To justify the second step: one can without loss of generality assume
that $k> n$. Then $(a_n-\ep^{(n)}_k)_+\precsim (a_k-\ep^{(n)}_k)_+$
since $a_n\precsim a_k$, and $(a_k-\ep^{(n)}_k)_+\precsim (a_k-\ep^{(k)}_k)_+$ because $\ep_k^{(n)}\geq \ep_k^{(k)}$.) Thus $[b]$
is an upper bound of the sequence. To see that it is the smallest,
let $[c]$ be another upper bound. Then
\[(a_n-\ep)_+\precsim a_n \precsim c\]
for any $\varepsilon>0$, and thus $b\precsim c$. We conclude that $[b]$ is the supremum of
$([b_n])_{n\in\N}$.
\end{proof}

\begin{rem}\label{rem:CutDownsConverge}
It follows from Case~1 above that for all $a\in (A\otimes \K)_+$ we have
\[[a]=\sup\limits_{\ep>0}[(a-\ep)_+].\]
\end{rem}

\begin{rem}\label{rem:ReplaceIncrSeq}
The proof of \autoref{thm:suprema} shows that if $(a_n)_{n\in\N}$ is any sequence in $(A\otimes\K)_+$ which is increasing in 
$\Cu(A)$, then there exists an increasing sequence $(b_n)_{n\in\N}$ in $(A\otimes\K)_+$ which converges to an element representing
$\sup\limits_{n\in\N}[a_n]$, and satisfies $b_n\precsim a_n$
for all $n\in\N$. (However, one cannot in general arrange
that $a_n\sim b_n$.)
\end{rem}

\begin{prop}\label{prop:CutDownWayBelow}
Let $A$ be a \ca, let $(a_n)_{n\in\N}$ be a sequence in
$(A\otimes\K)_+$, let $a\in (A\otimes\K)_+$ and let $\ep>0$.
If 
\[[a]\leq \sup\limits_{n\in\N}[a_n],\]
then there exists $n\in\N$ with $[(a-\ep)_+]\leq [a_n]$.
\end{prop}
\begin{proof}
Use \autoref{rem:ReplaceIncrSeq} to find an increasing 
sequence $(b_n)_{n\in\N}$ in $(A\otimes\K)_+$ with limit $b$
satisfying $[b]=\sup\limits_{n\in\N}[a_n]$ and $b_n\precsim a_n$
for all $n\in\N$.
Since $a\precsim b$, by R{\o}rdam's lemma (\autoref{thm:Rordam})
there exists $\delta>0$ such that 
$(a-\ep)_+\precsim (b-\delta)_+$. 
Find $n\in\N$ such that $\|b-b_n\|<\delta$, so that 
$(b-\delta)_+\precsim b_n$ by \autoref{lma:CutDownDistance}.
Then 
\[(a-\ep)_+\precsim (b-\delta)_+\precsim b_n\precsim a_n,\]
as desired.
\end{proof}

The cut-down is necessary in \autoref{prop:CutDownWayBelow}:
for example, let $a_n\in C([0,1])$ be a positive function supported on 
$\big[\frac{1}{n}, 1-\frac{1}{n}\big]$, such that $(a_n)_{n\in\N}$
converges to a function $a$ vanishing only on $0$ and $1$. 
Then $[a]=\sup\limits_{n\in\N}[a_n]$, but there is no $n$ such that
$[a]\leq [a_n]$ by \autoref{prop:CtzCompCX}.

\section{Compact containment and the category \textbf{Cu}}
\label{sec:waybelow}

In this section, we will explore the order-theoretic aspects of 
$\Cu(A)$ in more detail. For this, some abstraction will be 
necessary, and we will often work with (partially) 
ordered semigroups $(S,\leq)$, or even just ordered sets; 
the example we will always have in mind is $(\Cu(A),\leq)$. A more general version of the following definition appears in \cite[Definition I-1.1]{GieHofKeiLawMisSco_continuous_2003}.

\begin{df}\label{df:CompCont}
Let $(S,\leq)$ be an ordered set. We define an additional 
relation $\ll$ on $S$, called \emph{(sequential) compact
containment}, as follows: for $s,t\in S$, we
set $s\ll t$ if whenever $(x_n)_{n\in\N}$ is an increasing
sequence in $S$ with supremum $x$ satisfying $t\leq x$, then 
there exists $n\in\N$ such that $s\leq x_n$. 

We say that $s\in S$ is \emph{compact} if $s\ll s$.
\end{df}

\begin{rem}\label{rem:CutDownWayBelow}
\autoref{prop:CutDownWayBelow} shows precisely
that
$[(a-\ep)_+]\ll [a]$ for all $\ep>0$ and all $a\in (A\otimes\K)_+$.
 \end{rem}

It is easy to see that $s\ll t$ implies $s\leq t$, but the 
converse is not true, even in Cuntz semigroups of (commutative) C$^*$-algebras; see \autoref{prop:CompactContCX}. 
In fact, the relation $\ll$ is an example of what is called an \emph{auxiliary relation}; see \autoref{pgr:aux}.
We now show that compact containment in C$^*$-algebras can be characterized using
cut-downs:

\begin{prop}\label{prop:WayBelowCutDown}
Let $A$ be a \ca\ and let $a,b\in (A\otimes\K)_+$. Then 
$[a]\ll [b]$ if and only if there exists $\ep>0$ such that
$[a]\leq [(b-\ep)_+]$. In particular, $[a]$ is compact if 
and only if there exists $\ep>0$ with $a\sim (a-\ep)_+$.
\end{prop}
\begin{proof} Suppose that $[a]\ll [b]$. Since 
$\sup\limits_{n\in\N}[(b-\frac{1}{n})_+]=[b]$ by \autoref{rem:CutDownsConverge}, it follows that there exists
$n\in\N$ with $[a]\leq [(b-\frac{1}{n})_+]$. The 
converse follows from \autoref{rem:CutDownWayBelow}.
\end{proof}


Using the above proposition together with \autoref{prop:CtzCompCX}, 
compact containment in commutative \ca s can be
easily characterized in terms of open supports. We
leave the proof as an exercise:

\begin{prop}\label{prop:CompactContCX}
Let $X$ be a compact Hausdorff space and let $a,b\in C(X)_+$. 
Then $a\ll b$ if and only if 
\[\overline{\supp_{\mathrm{o}}(a)}\subseteq \supp_{\mathrm{o}}(b).\]
In particular, $a\in C(X)$ is compact if and only if
$\supp_{\mathrm{o}}(a)$ is compact.
\end{prop}


\begin{df}\label{df:CatCu}
Let $(S,\leq)$ be a positively ordered monoid. We say that $S$
is an \emph{(abstract) Cuntz semigroup}, or just a Cu-\emph{semigroup}, if it satisfies the following so-called axioms:
\be
\item[(O1)] Every increasing sequence has a supremum.
\item[(O2)] For every $s\in S$, there is a sequence 
$(s_n)_{n\in\N}$ in $S$ with $s_n\ll s_{n+1}$ and 
$s=\sup\limits_{n\in\N} s_n$.
\item[(O3)] If $s\ll t$ and $s'\ll t'$, then $s+s'\ll t+t'$.
\item[(O4)] If $(s_n)_{n\in\N}$ and $(t_n)_{n\in\N}$ are increasing sequences, then \[\sup\limits_{n\in\N} (s_n+t_n)
=\sup\limits_{n\in\N} s_n+ \sup\limits_{n\in\N} t_n.\]
\ee

Given $\Cu$-semigroups $S$ and $T$, a Cu-\emph{morphism} between
them is a map $f\colon S\to T$ preserving addition, neutral element, order $\leq$, suprema of increasing sequences, and also
the compact containment relation $\ll$. Maps that preserve all
the structure except possibly for $\ll$ are called \emph{generalized $\Cu$-morphisms}. They are also relevant,
as we will see in \autoref{sec:Structure}.

We denote by \textbf{Cu} the category whose objects are Cu-semigroups and whose morphisms are Cu-morphisms. The set of $\Cu$-morphisms between two semigroups $S$ and $T$ will be denoted by $\mathbf{Cu}(S,T)$, and the set of generalized $\Cu$-morphisms will be denoted by $\mathbf{Cu}[S,T]$.
\end{df}

The following result, due to Coward, Elliott and Ivanescu
\cite{CowEllIva_Cuntz_2008},
was arguably the beginning of the systematic study of Cuntz
semigroups. 

\begin{thm}\label{thm:CuAinCu}
Let $A$ be a \ca. Then $\Cu(A)$ is a Cu-semigroup. Moreover, if 
$\varphi\colon A\to B$ is a $\ast$-homomorphism between \ca s, 
then $\varphi$ naturally induces a Cu-morphism $\Cu(\varphi)\colon \Cu(A)\to \Cu(B)$. In other words, $\Cu$ is a functor from 
the category \textbf{C$^*$} of C$^*$-algebras to \textbf{Cu}.
\end{thm}
\begin{proof}
Most of the work has already been done. (O1) is \autoref{thm:suprema}, while
(O2) follows from \autoref{rem:CutDownsConverge} and 
\autoref{rem:CutDownWayBelow}.
To verify (O3), let $a,a',b,b'\in (A\otimes\K)_+$ satisfy
$[a]\ll [b]$ and $[a']\ll [b']$. Upon identifying
$[a]$ with the class of $\begin{psmallmatrix}a & 0\\0 & 0\end{psmallmatrix}\in A\otimes\K\otimes M_2\cong A\otimes\K$, and 
similarly for $a',b,b'$, we
may assume that $a\perp a'$ and $b\perp b'$. 
Use \autoref{prop:WayBelowCutDown} 
to find $\ep>0$ such that $[a]\leq [(b-\ep)_+]$ and $[a']\leq 
[(b'-\ep)_+]$. Using that $b\perp b'$ at the second step, we get 
\[a+a'\precsim (b-\ep)_+ + (b'-\ep)_+ =(b+b'-\ep)_+.\]
Note that $[(b+b'-\ep)_+]\ll [b+b']$ by \autoref{prop:CutDownWayBelow}. Using this at the second step, that $a\perp a'$ 
at the first step, and that 
$b\perp b'$ at the last step (both in combination with 
\autoref{lma:orth}, we get 
\[
[a]+[a']=[a+a']\ll [b+b']=[b]+[b'],
\]
as desired.

Finally, to verify (O4), 
let $(s_n)_{n\in\N}$ and $(t_n)_{n\in\N}$ be increasing sequences 
in $\Cu(A)$.
Since $s_n+t_n\leq \sup\limits_{n\in\N} s_n + \sup\limits_{n\in\N} t_n$, we 
get 
\[\sup\limits_{n\in\N} (s_n+t_n)
\leq \sup\limits_{n\in\N} s_n+ \sup\limits_{n\in\N} t_n.\]
To show the converse inequality, use \autoref{rem:ReplaceIncrSeq}
to find increasing, norm-con\-ver\-gent sequences $(a_n)_{n\in\N}$ 
and $(b_n)_{n\in\N}$ in $(A\otimes\K)_+$, with limits $a$ and $b$,
satisfying $[a_n]\leq s_n$ and $[b_n]\leq t_n$ for all $n\in\N$, 
and such that $[a]=\sup\limits_{n\in\N} s_n$ and $[b]=\sup\limits_{n\in\N}t_n$. 
Then
\[\sup\limits_{n\in\N} s_n + \sup\limits_{n\in\N} t_n=[a\oplus b] = \sup\limits_{n\in\N}
 [a_n\oplus b_n]=\sup\limits_{n\in\N}
 \big([a_n]+[ b_n]\big)\leq \sup\limits_{n\in\N}
 \big(s_n+t_n\big),
\]
as desired. This finishes the proof.
\end{proof}

\begin{rem}\label{rem:Scott}
Although there is no topology in a Cu-semigroup, in view of (O1) one usually thinks of Cu-semigroups as being complete in
a suitable order-topology, called the \emph{Scott topology}.
A base for this topology is given by those upward-hereditary sets 
$U$ (that is, $s\in U$ and $s\leq t$ imply $t\in U$) such that for 
every $s\in U$ there exists $s'\in U$ with $s'\ll s$. 
With respect to this topology, an increasing sequence in 
$\Cu(A)$ actually \emph{converges} to its supremum. Regarding 
increasing sequences as the order-theoretic analogues of Cauchy
sequences, (O1) states that Cu-semigroups are complete. 
\end{rem}

\begin{df}[{\cite[Definition 2.3]{WinZac_completely_2009}}]
We say that a linear map $\varphi\colon A\to B$ between C$^*$-algebras $A$ and $B$ is a \emph{completely positive contractive map of order-zero}, in short a cpc$_\perp$ map, in case the natural extensions of $\varphi$ to matrices are all positive and contractive, and $\varphi$ preserves orthogonality, in the sense that $\varphi(a)\varphi(b)=0$ whenever $a,b\in A_+$ satisfy $ab=0$.
\end{df}

Just as homomorphisms are the natural models for $\Cu$-morphisms, cpc$_\perp$ maps are natural models for generalized $\Cu$-morphisms, as the result below shows (we omit its proof):

\begin{prop}[{see \cite[Corollary 4.5]{WinZac_completely_2009} and \cite[Proposition 2.2.7]{AntPerThi_tensor_2018}}]
Let $\varphi\colon A\to B$ be a cpc$_\perp$ map. Then $\varphi$ induces a generalized $\Cu$-morphism $\Cu(\varphi)\colon \Cu(A)\to \Cu(B)$, given by $\Cu(\varphi)([a])=[\varphi(a)]$ for 
all $a\in (A\otimes\K)_+$.
\end{prop}

Perhaps a natural question at this point is whether \emph{every}
Cu-semigroup is the Cuntz semigroup of a C$^*$-algebra. The answer
is unfortunately negative. The following, due to Bosa and Petzka, 
is the smallest example:

\begin{eg}(\cite[Example~5.3]{BosPet_comparison_2018}).
Set $S=\{0,1,\infty\}$ with the usual order and addition ($1+1=\infty$).
Then there is no C$^*$-algebra $A$ with $\Cu(A)=\{0,1,\infty\}$. This is
however surprisingly difficult to prove.
\end{eg}

In fact there are more axioms that $\Cu(A)$ always satisfies, 
and some of
these will be discussed in Section~\ref{sec:Axioms}. Describing precisely which 
Cu-semigroups arise as $\Cu(A)$ is extremely difficult and 
currently considered to be out of reach, although this is 
possible (and usually tedious) for some specific classes of 
C$^*$-algebras such as AF-algebras 
\cite{AntPerThi_tensor_2018}, AI-algebras 
\cite{Vil_local_2022}, certain commutative
C$^*$-algebras with 2-dimensional spectrum \cite{Rob_spaces_2013}, and simple $\mathcal{Z}$-stable
\ca s (see \autoref{thm:CuAtimesZ}). 

The functor $\Cu\colon \mathbf{C^*}\to \mathbf{Cu}$ has many 
nice properties and preserves a number of constructions, including
direct limits, short exact sequences, direct sums, direct
products, and ultraproducts. These claims have to be suitably
interpreted: there are categorical versions of the above notions,
which can be shown to always exist in $\CatCu$ (their existence
in \textbf{C$^*$} is known), and the functor $\Cu$ preserves these. 
The study of the category $\CatCu$ in itself (or some 
subcategory of it) is crucial in this setting, and allows one to
better understand the functor Cu;
this is explored in detail in \autoref{sec:Structure}.
Thus, and even if one is only
interested in $\Cu(A)$, studying abstract Cu-semigroups is often
necessary.

Knowing that the functor $\Cu$ preserves a number of constructions 
is unfortunately not very useful without understanding how 
to actually construct these objects in \textbf{Cu}. 
We will only focus on direct limits in this section:

\begin{thm}\label{thm:Culimits}
The category $\CatCu$ has direct limits, and the functor
$\Cu$ satisfies 
\[\Cu(\varinjlim A_n)\cong \varinjlim \Cu(A_n).\]
\end{thm}
\begin{proof}
We only briefly describe how to show that $\CatCu$ has 
direct limits; a more conceptual approach is given in
\autoref{thm:completion} and the comments after it. Let 
\[(\phi_n\colon S_n\to S_{n+1})_{n\in\N}\] 
be an inductive sequence in $\CatCu$.
Denote by $W$ the direct limit of this sequence in the category of positively ordered monoids. (For example, take the direct limit as
semigroups, and define an order by declaring that two sequences
compare if they eventually compare.)
This will in general not be a Cu-semigroup, since increasing 
sequences may not necessarily have suprema. 
The correct object to consider is a
certain ``completion'' of $W$; 
see \autoref{rem:Scott} for a more formal 
interpretation of what completion
means, and also \ref{thm:completion} and the comments after it. Intuitively speaking, we want to add the suprema of all increasing sequences, similarly to how one adds the limits of all Cauchy 
sequences when completing a metric space. 
Since we also want (O2) to be satisfied, we will 
take sequences in $W$ which are $\ll$-increasing.

We define an order on the space of $\ll$-increasing sequences in 
$W$ by 
setting $(x_n)_{n\in\N}\precsim (y_n)_{n\in\N}$ if 
for every $n\in\N$ there is $m\in\N$ with 
$x_n\leq y_m$. Write $\sim$ for the symmetrization of this 
order: $(x_n)_{n\in\N}\sim (y_n)_{n\in\N}$ if 
$(x_n)_{n\in\N}\precsim (y_n)_{n\in\N}\precsim (x_n)_{n\in\N}$. 
The space $S$ of equivalence classes
of increasing sequences in $W$ is a Cu-semigroup, and one can show
that $S=\varinjlim (S_n,\phi_n)$.
\end{proof}

In retrospect, the construction of direct limits in \textbf{Cu}
is very similar to the construction of direct limits in
\textbf{C$^*$}. Indeed, one first considers the algebraic 
direct limit, and then suitably completes 
the resulting direct limit to get the desired object.
The completion procedure described in the proof of 
\autoref{thm:Culimits} is actually a functor from a
suitable category \textbf{W} to \textbf{Cu}, and 
this will be explored in detail in Section~12; see 
\autoref{thm:curefl} and \autoref{thm:completion}.

We will use the above to compute the Cuntz semigroup
of the CAR-algebra:

\begin{eg}\label{eg:CuCAR}
We denote by $M_{2^\infty}$ the UHF-algebra of type $2^\infty$,
which is the direct limit of $M_{2^n}$ with connecting maps 
of the form $a\mapsto \begin{psmallmatrix}a & 0\\0 & a\end{psmallmatrix}$. By \autoref{eg:CuC}, we have 
$\Cu(M_{2^n})\cong \overline{\mathbb{N}}$, and via the isomorphism
given in that example (the rank), the connecting maps are easily
seen to be multiplication by 2. The algebraic direct limit is
$W=\mathbb{N}\big[\frac{1}{2}\big]\cup \{\infty\}$, with the order
inherited from $[0,\infty]$. Note that $w\ll w$ for every 
$w\in W\setminus\{\infty\}$.
It is easy to see that $W$ is not a Cu-semigroup, since not every increasing sequence has a supremum.\footnote{Take, for example, any increasing
sequence of dyadic numbers converging to a non-dyadic number.} (Note, however, that the supremum exists in $[0,\infty]$.) 

Let $(a_n)_{n\in\N}$ be an $\ll$-increasing sequence in $W$. 
Regarding 
it as an increasing sequence in $[0,\I]$, it has a limit $x\in [0,1]$. 
We claim that the $\sim$-class of $(x_n)_{n\in\N}$ only depends
on $x\in [0,\I]$ and whether $(x_n)_{n\in\N}$ is constant 
or strictly increasing (both interpreted as eventual behaviors).
It is not difficult to see that any two strictly increasing sequences in $W$ 
whose limits in $[0,\I]$ agree are automatically equivalent, and 
similarly for constant sequences. Finally, one also checks that
a constant sequence
cannot be equivalent to a strictly increasing sequence. Since
equivalent sequences must have the same limit in $[0,\I]$, the 
claim follows.

Given a $\ll$-increasing sequence $(a_n)_{n\in\N}$
with limit $x\in [0,\I]$,
we abbreviate its $\sim$-equivalence class as 
\begin{align*}
[(a_n)_{n\in\N}]=\begin{cases}
                 c_x, & \mbox{ if } (a_n)_{n\in\N} \mbox{ is constant (in which case }x\in W\setminus\{\infty\}),\\
                 s_x, & \mbox{ if } (a_n)_{n\in\N} \mbox{ is strictly increasing (in which case }x\in (0,\I]).
                 \end{cases}
\end{align*}

It follows that 
\[\Cu(M_{2^\infty})=\{c_x\colon x\in W\setminus\{\infty\}\}\sqcup \{s_x\colon x\in (0,\I]\}\cong 
\N\big[\tfrac{1}{2}\big]\sqcup (0,\infty].\]
Addition and order work as one would expect on each component.
For mixed terms, we have:
\bi
\item $s_x\leq c_y$ if and only if $x\leq y$;
\item $c_x\leq s_y$ if and only if $x< y$; 
\item $c_x+s_y=s_{x+y}$.
\ei
In particular, $c_x\leq s_x$ but $s_x\nleq c_x$ (otherwise
they would be equal).
\end{eg}

\begin{rem}\label{rem:CuCAR}
The example above shows some phenomena that can be seen in the 
Cuntz semigroups of more general C$^*$-algebras:
\be[{\rm (i)}]\item 
The compact elements of $\Cu(M_{2^\infty})$ are $\mathbb{N}\big[\frac{1}{2}\big]\cong \V(M_{2^{\infty}})$; see \autoref{cor:VACuA} and the comments after it, and see 
also \autoref{prop:CpctSr1} below. The fact that $\V(A)$ can be identified with the compact elements in $\Cu(A)$ is a general fact about stably finite 
C$^*$-algebras, due to Brown-Ciuperca; see \cite{BroCiu_isomorphism_2009}. 
\item The component $(0,\I]$ of $\Cu(M_{2^\infty})$ 
corresponds to the values of positive 
elements, not equivalent to projections, on the unique trace.
\ee
\end{rem}

More generally, the Cuntz semigroup of a simple, stably finite,
$\mathcal{Z}$-stable \ca\ $A$ can be computed in terms of $\V(A)$
and $\mathrm{T}(A)$ (or rather, $\QT(A)$) in a similar fashion; see \autoref{thm:CuAtimesZ}. 

The first part of \autoref{rem:CuCAR} admits a somewhat more direct
proof for C$^*$-algebras of stable rank one, which is the 
case that we will need in these notes. Recall from \cite{BlaCun_structure_1982} that an element $z$ in a C$^*$-algebra $A$ is called a \emph{scaling element} provided $zz^*\neq z^*z$ and $(z^*z)(zz^*)=zz^*$. It was shown in \cite[Theorem 4.1]{BlaCun_structure_1982} and the arguments after it that if a C$^*$-algebra $A$ contains a scaling element, then $M_n(A)$ contains an infinite projection for some $n\geq 1$.

\begin{prop}\label{prop:CpctSr1}
Let $A$ be a \ca\ of stable rank one, and let $x\in
\Cu(A)$ be a compact element. Then there exists a
projection $p\in A\otimes\K$ such that $[p]=x$. 

In particular, the subsemigroup of $\Cu(A)$ consisting
of compact elements is order-isomorphic to $\V(A)$.
\end{prop}
\begin{proof}
By part~(iii) of \autoref{prop:sr1}, we may assume that $A$ is stable. Write $x=[a]$ for some $a\in A_+$. 
For each $\varepsilon>0$, define
\[
g_\varepsilon(t)=\begin{cases}
		0, & \text{if } t\leq \tfrac{\ep}{2} \\
        \text{linear}, &\text{if } t\in [\tfrac{\ep}{2},\varepsilon]\\
        1, &\text{if } t\geq\varepsilon.
       \end{cases}
\]
Note that $g_{\frac{\ep}{2}}g_\ep=g_\ep$ and 
that $ g_\varepsilon(a)\sim (a-\tfrac{\ep}{2})_+\leq a$ by \autoref{prop:CtzCompCX}. Since $x=\sup\limits_{\ep>0} [g_\varepsilon(a)]$ and $x$ is compact, there is $\varepsilon>0$ such that 
\[a\sim g_{\varepsilon}(a)\sim g_{\varepsilon'}(a)\] 
for all $\varepsilon'<\varepsilon$. 
Since $A$ has stable rank one, we can use \autoref{prop:CtzCompsr1} to find $z\in A$ such that $g_{\frac{\ep}{2}}(a)=z^*z$ and $zz^*\in A_{g_\varepsilon(a)}$. Thus $(z^*z)(zz^*)
=g_{\frac{\ep}{2}}(a)zz^*=zz^*$. 

Assume that $zz^*=z^*z$. Then $g_{\frac{\ep}{2}}(a)\in  A_{g_\varepsilon(a)}$. Let $\delta>0$ be such that $\frac{\varepsilon}{4}(1+\delta)<\frac{\varepsilon}{2}$. Then there is $w\in A$ such that $\Vert g_{\frac{\ep}{2}}(a)-g_\varepsilon(a)wg_\varepsilon(a)\Vert <\tfrac{\delta}{4}$. It follows that $\Vert g_{\frac{\ep}{2}}(a)^2-g_{\frac{\ep}{2}}(a)\Vert<\tfrac{\delta}{2}$. A standard application
of functional calculus allows us to find a projection $p\in A_{g_\varepsilon(a)}$ with $\Vert p-g_{\frac{\ep}{2}}(a)\Vert< \delta$; see, for example \cite[Lemma 5.1.6]{Weg_ktheory_1993}. 

By \autoref{lma:CutDownDistance}, we have $(g_{\frac{\ep}{2}}(a)-\delta)_+\precsim p$. By our choice of $\delta$, and since $(g_{\frac{\ep}{2}}(a)-\delta)_+\sim g_{\frac{\ep}{2(1+\delta)}}(a)$, we obtain that 
\[
g_{\frac{\ep}{2(1+\delta)}}(a)\precsim p\precsim g_\varepsilon(a)\precsim g_{\frac{\ep}{2(1+\delta)}}(a),
\]
and thus $a\sim p$.

Assume now that $zz^*\neq z^*z$, so that $z$ is a scaling element. As mentioned before this proposition above, this implies that $M_n(A)$ contains an infinite projection, in contradiction with the fact that $A$ has stable rank one.

The last statement follows from \autoref{lma:CtzCompPjns}.
\end{proof}

\section{Ideals and quotients}
The goal of this and the following sections is to show how some
very important information about $A$ is completely 
encoded in $\Cu(A)$. In this section, we will show how to recover
the ideal lattice of $A$ from its Cuntz semigroup;
as a byproduct, we will see that the Cuntz semigroup of every ideal and every quotient of $A$ can be read off of $\Cu(A)$. The 
upcoming section deals with (quasi)traces on $A$.

\begin{df}\label{df:IdealCu}
Let $S$ be a Cu-semigroup. An \emph{ideal} in $S$ is a 
submonoid $I\subseteq S$ which is closed under suprema of 
increasing sequences and is hereditary, in the sense
that $a\leq b$ and $b\in I$ imply $a\in I$. 
We denote by $\Lat(S)$ the lattice of ideals of $S$, which is
ordered by inclusion.
We say that $S$ is 
\emph{simple} if $\Lat(S)$ consists only of $\{0\}$ and $S$.
\end{df}

Next, we will show that ideals in $A$ naturally induce ideals
in $\Cu(A)$.
For Cu-morphisms which are 
order-embeddings, the ordered structure of the domain and 
the induced structure in the codomain agree.

\begin{lma}\label{lma:IdealAIdealCu}
Let $A$ be a \ca\ and let $J$ be an ideal in $A$. Denote by
$\iota\colon J\to A$ the canonical inclusion. Then 
\[\Cu(\iota)\colon \Cu(J) \to \Cu(A)\]
is an order-embedding, and its image $\{[x]\in \Cu(A)\colon x\in (J\otimes\K)_+\}$ is an ideal in $\Cu(A)$.
\end{lma}
\begin{proof} Without loss of generality, assume that $A$ and $J$ are stable. Let $x,y\in J$ satisfy $x\precsim y$ in $A$. 
Given $\ep>0$ choose $r\in A$ such that 
$ryr^*\approx_{\frac{\ep}{2}} x$.
Find $e\in J_+$ such that $eye\approx_{\frac{\ep}{2\|r\|^2}}y$. Then $re\in J$ and 
\[(re)y(re)^*\approx_{\frac{\ep}{2}}ryr^*\approx_{\frac{\ep}{2}}x,\]
showing that $x\precsim y$ in $J$.
It remains to show that the image of $\Cu(\iota)$ is an ideal, 
for which it suffices to show that it is hereditary. 
Let $a\in A_+$ and $b\in J_+$ satisfy
$a\precsim b$ in $A$, and choose a sequence $(r_n)_{n\in\N}$
in $A$ satisfying $a=\lim\limits_{n\to\I} r_nbr_n^*$. Since $J$
is an ideal, it follows that $r_nbr_n^*\in J$ and thus $a\in J$.
\end{proof}

In view of the lemma above, whenever $J$ is an ideal in $A$, 
we will identify $\Cu(J)$ with an ideal in $\Cu(A)$.
We isolate the following observations for future use:

\begin{rem}\label{rem:x+y}
Let $A$ be a C$^*$-algebra and let $x,y\in A$. Using
that $(x-y)^*(x-y)\geq 0$, one gets
\[(x+y)^*(x+y)\leq 2x^*x+2y^*y.\]
Thus $x^*y+y^*x\leq x^*x+y^*y$ and also 
$[(x+y)^*(x+y)]\leq [x^*x]+[y^*y]$ in $\Cu(A)$.
\end{rem}

\begin{rem}\label{rem:ideala*a}
Let $A$ be a \ca\ and let $J$ be an ideal in $A$.
For $a\in A$, one can use the polar decomposition to 
show that $a\in J$ if and only if $a^*a\in J$.
\end{rem}

The following is the main result of this section; see \cite{CiuRobSan_ideals_2010}.

\begin{thm}\label{thm:CuIdeals}
Let $A$ be a \ca. Let $\Phi\colon \Lat(A)\to \Lat(\Cu(A))$ be 
given by $\Phi(J)=\Cu(J)$. Then $\Phi$ is an order-isomorphism.
Its inverse $\Psi\colon \Lat(\Cu(A))\to \Lat(A)$ 
is given by
\[\Psi(I)=\{x\in A\colon [x^*x]\in I\}.\]
\end{thm}
\begin{proof}
We begin by showing that $\Psi$ is well-defined, namely that
$\Psi(I)$ is an ideal in $A$. Given $x,y\in \Psi(I)$, by
\autoref{rem:x+y} we have
\[[(x+y)^*(x+y)]\leq [x^*x]+[y^*y]\in I.\]
Since $I$ is hereditary, it follows that $[(x+y)^*(x+y)]\in I$
and thus $x+y\in \Psi(I)$.
That $\Psi(I)$ is closed under scalar multiplication is clear. To show 
that it is a left and right ideal, let $x\in \Psi(I)$ and let $a\in A$.
Since
\[(ax)^*(ax)=x^*a^*ax\leq \|a\|^2 x^*x,\]
it follows from \autoref{prop:InHerSubalg} that $(ax)^*(ax)\precsim x^*x$.
Also, $(xa)^*(xa)=a^*x^*xa\precsim x^*x$. Since $I$ is hereditary,
it follows as before that $ax, xa\in \Psi(I)$.

It remains to show that $\Psi(I)$ is closed. Let $(x_n)_{n\in\N}$ be a
sequence in $\Psi(I)$ converging to $a\in A$. Then $(x_n^*x_n)_{n\in\N}$ 
converges to $a^*a$. Given $\ep>0$, find $n\in\N$ such that
$\|x_n^*x_n-a^*a\|<\ep$. By \autoref{prop:CutDownWayBelow},
we have $[(a^*a-\ep)_+]\leq [x_n^*x_n]\in I$. Since $I$ is
hereditary, we deduce that $[(a^*a-\ep)_+]\in I$.
Since $[a^*a]=\sup\limits_{m>0}[(a^*a-\frac{1}{m})_+]$ and $I$
is closed under suprema, we conclude that $[a^*a]\in I$ and
hence $a\in \Psi(I)$. 

It is also clear that both $\Phi$ and $\Psi$ are order-preserving.
To show that they are mutual inverses, let $I\in \Lat(\Cu(A))$.
Given $x\in A_+$, we have $[x]\in \Phi(\Psi(I))$ if and only if
$x\in \Psi(I)$, if and only if $[x^*x]=[x]\in I$. Thus 
$I=\Phi(\Psi(I))$. For the converse, given $J\in \Lat(A)$ 
we want to show that $J=\Psi(\Phi(J))$. 
For $a\in A$, we have $a\in \Psi(\Phi(J))$ if and only if
$[a^*a]\in \Phi(J)$, if and only if $[a^*a]\in \Cu(J)$.
Using the definition of Cuntz equivalence, it is clear that the 
above is equivalent to $a^*a\in J$, which is equivalent to
$a\in J$ by \autoref{rem:ideala*a}.
Thus $J=\Psi(\Phi(J))$ and the proof is complete.
\end{proof}

\begin{nota}\label{nota:infa}
Let $S$ be a Cu-semigroup. Given $a\in S$, write $\infty_a$ for the 
supremum of $(na)_{n\in\N}$.\end{nota}

It is easy to see that 
the ideal generated by $a$ is precisely 
$\{x\in S\colon x\leq \infty_a\}$.

\begin{cor}\label{cor:SimpleCuSmgp} 
Let $S$ be a Cu-semigroup. Then $S$ is simple if and only if 
$\infty_a=\infty_b$ for all $a,b\in S\setminus\{0\}$. 
Equivalently, for all nonzero $a,b\in S$, we have 
$a\leq \infty_b$.
\end{cor}

In view of the above corollary, there is a unique infinity
in every simple Cu-semigroup, which we will denote simply by $\infty$.

In the remainder of this section, we explain how to detect the 
Cuntz semigroups of all quotients of a given C$^*$-algebra in its 
own Cuntz semigroup.

\begin{df} 
Let $S$ be a Cu-semigroup and let $I$ be an ideal in it. 
Given $a,b\in S$, we set $a\leq_I b$ if there exists $c\in I$
such that $a\leq b+c$. We set $a\sim_I b$ if $a\leq_I b$ and 
also $b\leq_I a$.
\end{df}

It is easy to see that $\sim_I$ is an equivalence relation;
we write $S/I$ for the associated quotient. 
For $a\in S$, we write $a_I$ for its equivalence class.
We define an addition on $S/I$ by setting $a_I+b_I=(a+b)_I$, and 
an order by declaring $a_I\leq b_I$ if $a\leq_I b$. It can be 
shown that this gives $S/I$ the structure of a Cu-semigroup. The following is proved in \cite{CiuRobSan_ideals_2010}.

\begin{thm}\label{thm:quotients}
Let $A$ be a \ca, let $J$ be an ideal in it, and let 
$\pi\colon A\to A/J$ denote the quotient map. Given $a,b\in\Cu(A)$, we have 
\[\Cu(\pi)(a)\leq \Cu(\pi)(b) \mbox{ in } \Cu(A/J) \ \ \mbox{ if and only if } \ \ a\leq_{\Cu(J)} b \mbox{ in } \Cu(A).\]
In particular, $\Cu(\pi)\colon \Cu(A)\to \Cu(A/J)$ induces an order-isomorphism
\[\Cu(A)/\Cu(J)\cong \Cu(A/J).\]
\end{thm}
\begin{proof} 
Let $a,b\in \Cu(A)$ satisfy $a\leq_{\Cu(J)} b$, and choose
$c\in \Cu(J)$ with $a\leq b+c$. Since $\Cu(\pi)(c)=0$, we get 
\[\Cu(\pi)(a)\leq\Cu(\pi)(b+c)=\Cu(\pi)(b),\]
as desired. The converse is somewhat more involved, and we omit it.
\end{proof}

The results of this section should be compared to similar 
statements in K-theory: while $\Cu(A)$ encodes the lattice
of ideals of $A$, as well as the Cuntz semigroups of all ideals, 
in general K-theory does not contain any of this information.
(There is one notable exception: the ordered $\KK_0$-group of an
AF-algebra encodes the ideal structure and the $\KK_0$-groups of 
all ideals and quotients. This helps explain why the classification
of AF-algebras works also in the non-simple case, without a larger
invariant.)

\section{Functionals, quasitraces, and Cuntz's theorem}

In this section, we will show that there is a natural bijective
correspondence between (quasi)traces on $A$ and functionals on $\Cu(A)$; see \autoref{thm:QTfunctionals}. We will 
then present Cuntz's theorem, stating that a simple, stably
finite C$^*$-algebra always admits a quasitrace; see \autoref{thm:Cuntz}.

The motivating observation for this section comes from the 
Riesz theorem:

\begin{thm}\label{thm:Riesz}(Riesz).
Let $X$ be a compact Hausdorff space. Then there is a natural 
bijection between the set of all Radon Borel probability measures on $X$ and all tracial states on $C(X)$.
Given a Radon probability measure $\mu$ on $X$, the 
corresponding tracial state $\tau_\mu\colon C(X)\to \C$ is given by
\[\tau_\mu(f)=\int_X f(x)\ d\mu(x) = \int_0^\I \mu(\{x\in X\colon f(x)> t\}) \ dt\]
for all $f\in C(X)$.
\end{thm}

We want to express the trace $\tau_\mu$ in a different manner.
Note that 
\[\{x\in X\colon f(x)> t\}=\supp_{\mathrm{o}}((f-t)_+),\]
which only depends on the Cuntz class of $(f-t)_+$. By writing 
$\tau_\mu$ as
\[\tau_\mu(f)=\int_0^\I \mu\big(\supp_{\mathrm{o}}((f-t)_+)\big) \ dt,\]
we have expressed $\tau_\mu(f)$ entirely in terms of 
Cuntz classes in $\Cu(C(X))$. 

Hence, if we are given a ``nice'' map $\lambda\colon \Cu(A)\to [0,\infty]$,
we may attempt to define a trace on $A$ via 
\[\tau_\lambda(a)=\int_0^\I \lambda\big(\supp_{\mathrm{o}}((a-t)_+)\big) \ dt\]
for all $a\in A$.
This is the essential idea behind the correspondence between traces
on $A$ and functionals on $\Cu(A)$. There is, however, one problem:
even for nice maps $\lambda$, it is not at all clear whether 
the map $\tau_\lambda$ defined above is additive. We are therefore
led to consider non-linear maps, which brings us to the definition of
a \emph{quasitrace}.

\begin{df}\label{df:quasitraces}
Let $A$ be a \uca. A \emph{1-quasitracial state} (or just 
\emph{1-quasitrace}) is a function $\tau\colon A\to \C$ such that
\bi\item[(i)] $\tau(a+ib)=\tau(a)+i\tau(b)$ for all $a,b\in A_{\mathrm{sa}}$;
\item[(ii)] $\tau$ is linear on \emph{commutative subalgebras} of $A$;
\item[(iii)] $\tau(a^*a)=\tau(aa^*)\geq 0$ for all $A$;
\item[(iv)] $\tau(1)=1$.\ei
We say that $\tau$ is a \emph{quasitracial state} (or just a 
\emph{quasitrace}) if it
extends to a 1-quasitrace on $M_n(A)$ for all $n\in\N$.\footnote{It was shown by Blackadar and Handelman that, in order for a $1$-quasitrace to be a quasitrace, it is enough for it to extend to $M_2(A)$; see \cite[Proposition II.4.1]{BlaHan_dimension_1982}.} We denote
by $\QT(A)$ the space of all quasitraces on $A$.
\end{df}

There exist 1-quasitraces that are not quasitraces; such examples
were first constructed by Kirchberg. 

\begin{rem} 
Note that 
a quasitrace is a trace if and only if it is additive
on all positive elements of $A$, not just on commuting ones. 
\end{rem}

It is a major open problem in C$^*$-algebra theory whether every
quasitrace is automatically a trace. By a very deep result of
Haagerup \cite{Haa_quasitraces_2014}, 
this is the case whenever the C$^*$-algebra
in question is \emph{exact}. In particular, this is always the case
for nuclear C$^*$-algebras.

We now define the \emph{dimension function} (or functional) 
associated to a quasitrace.

\begin{df}\label{df:DimFunct}
Let $A$ be a \ca. Given $\tau\in \QT(A)$, we define its associated
\emph{dimension function} $d_\tau\colon \Cu(A)\to [0,\I]$ by
\[d_\tau([a])=\lim_{n\to\I}\tau(a^{\frac{1}{n}})\]
for all $a\in M_\infty(A)_+$, and extended to 
$(A\otimes\K)_+$ by taking suprema.\footnote{Explicitly, 
for a general element
$a\in (A\otimes\K)_+$, one defines $d_\tau([a])$ to be the supremum
in $[0,\I]$
over $\ep>0$ of $d_\tau([(a-\ep)_+])$, which is well defined since $(a-\ep)_+$ belongs to $M_\infty(A)_+$.}
\end{df}

\begin{rem}
The sequence $(a^{\frac{1}{n}})_{n\in\N}$ converges in $A^{\ast\ast}$ in the weak-$\ast$ topology to the \emph{support
projection} $p_a$ of $a$; this is the smallest projection 
which acts as a unit on $a$. Since quasitraces are weak-$\ast$
continuous on $A^{\ast\ast}$, we have 
\[d_\tau(a)=
\lim\limits_{n\to\I}\tau(a^{\frac{1}{n}})=\tau(p_a).\]
\end{rem}

Implicit in \autoref{df:DimFunct} 
is the fact that $d_\tau([a])$
only depends on $[a]$. This is indeed the case. In fact, more is true:

\begin{prop}\label{prop:dtauFunctional}
Let $A$ be a \uca. Given $\tau\in \QT(A)$, the map $d_\tau$ is 
well-defined on $\Cu(A)$. Moreover, we have:
\be[{\rm (i)}]\item $d_\tau(0)=0$;
\item $d_\tau([1])=1$;
\item $d_\tau(s+t)=d_\tau(s)+d_\tau(t)$ for all $s,t\in \Cu(A)$;
\item $d_\tau(s)\leq d_\tau(t)$ whenever $s\leq t$ in $\Cu(A)$;
\item $d_\tau$ preserves suprema of increasing sequences.
\ee
\end{prop}

The above proposition is not hard to prove, but we will omit the argument. Instead, we will show a particular case, which
connects back to the motivation we gave after \autoref{thm:Riesz}.

\begin{prop}\label{prop:dtauonComm}
Let $X$ be a compact Hausdorff space, let $\mu$ be a Radon Borel
probability measure, and let $\tau_\mu\colon C(X)\to \C$ be the 
trace it induces. Given $a\in C(X)_+$, we have
\[d_{\tau_\mu}([a])=\mu(\supp_{\mathrm{o}}(a)).\]
\end{prop}
\begin{proof}
Without loss of generality, we may assume that $\|a\|\leq 1$.
Note that $(a^{\frac{1}{n}})_{n\in\N}$ is an increasing sequence
which converges pointwise to the indicator
function of $\supp_{\mathrm{o}}(a)$.
Applying the dominated convergence theorem at the third step, we get
\[d_\tau([a])\!=\!\lim_{n\to\I} \tau(a^{\frac{1}{n}})=\lim_{n\to\I}\int_X a^{\frac{1}{n}}(x)\ d\mu(x)\! =\! \int_X 
\lim_{n\to\I} a^{\frac{1}{n}}(x)\ d\mu(x)\! =\!\mu(\supp_{\mathrm{o}}(a)),\]
as desired.
\end{proof}

\autoref{prop:dtauFunctional} shows that the map $d_\tau$ is a normalized 
functional on $\Cu(A)$ in the sense of the 
following definition.

\begin{df}\label{df:functionals}
Let $S$ be a Cu-semigroup. A \emph{functional} on $S$ is a 
map $f\colon S\to [0,\I]$ 
which preserves the zero element, 
addition, order, and suprema of increasing sequences. We denote 
by $F(S)$ the set of all functionals on $S$.

If $e\in S$ is a distinguished compact element in $S$, a functional
$\lambda\in F(S)$ is said to be \emph{normalized (at $e$)} if 
$\lambda(e)=1$. We write $F_e(S)$ for the set of all normalized
functionals on $S$.
\end{df}

In this section, we will mostly focus on normalized functionals,
but we will come back to non-normalized ones in \autoref{sec:functionals}.

We write $d\colon \QT(A)\to F_{[1]}(\Cu(A))$ for the map 
given by $d(\tau)=d_\tau$ for $\tau\in \QT(A)$; see \autoref{prop:dtauFunctional}. Both $\QT(A)$ and $F_{[1]}(\Cu(A))$
are convex sets, and it is easy to see that $d$ is an affine 
map.\footnote{The sets $\QT(A)$ and $F_{[1]}(\Cu(A))$ also admit natural
topologies with respect to which they are compact and 
Hausdorff: for QT$(A)$
the topology is induced by pointwise convergence, while the
topology on $F_{[1]}(\Cu(A))$ is described at the beginning of 
\autoref{sec:functionals}. One can 
check that $d$ is continuous with respect to these topologies, and hence a homeomorphism. This is worked out in detail in
\cite{EllRobSan_cone_2011}.}

\begin{thm}\label{thm:QTfunctionals}
Let $A$ be a \uca. Then $d\colon \QT(A)\to F_{[1]}(\Cu(A))$ is 
an affine bijection.
\end{thm}
\begin{proof}
Given $\lambda\in F_{[1]}(\Cu(A))$, define $\tau_\lambda\in \QT(A)$ by 
\[\tau_\lambda(a)=\int_0^\I \lambda\big([(a-t)_+]\big) dt\]
for all $a\in A_+$, and extended $\mathbb{T}$-linearly to $A$ 
using real and imaginary parts, and positive and negative parts.
We claim that $\tau_\lambda$ is a 
quasitracial state on $A$. Condition (i) in \autoref{df:quasitraces}
is satisfied by construction, while (iv) is clear. 
For (ii), let $B$ be a commutative unital C$^*$-subalgebra, 
and let $X$ denote the 
maximal ideal space of $B$. Note that the 
restriction of $\tau_\lambda$ to $B$ induces an 
assignment $\mu_0\colon \mathcal{O}(X)\to [0,1]$,
defined on the open subsets $\mathcal{O}(X)$ of $X$, which is
$\sigma$-additive on disjoint sets.\footnote{In order to define 
$\mu_0$, given an open set $U$ find a continuous function $f\in C(X)$
whose open support is exactly $U$. Then set $\mu_0(U)=\lambda([f])$.}
By Caratheodory's Theorem, $\mu_0$ extends to a Radon Borel probability
measure $\mu$ on $X$. It follows that $(\tau_\lambda)|_B=\tau_\mu$, and thus 
$\tau_\lambda$ is additive on $B$.

For (iii), let $a\in A$. Using \autoref{lma:CutDwnSymm} at 
the second step, we have
\[\tau_\lambda (a^*a)=\int_0^\I \lambda([a^*a-t]_+) dt
 =\int_0^\I \lambda([aa^*-t]_+)dt=\tau_\lambda(aa^*)\geq 0,
\]
as desired.

Finally, in order to show that the two assignments $\tau\mapsto d_\tau$
and $\lambda\mapsto \tau_\lambda$ are mutual inverses, it suffices to
check this on commutative subalgebras, which 
follows from \autoref{prop:dtauonComm}.
\end{proof}

\begin{rem}
Note that the fact that $\tau_\lambda$ is additive on commutative
subalgebras is ultimately a consequence of Caratheodory's Theorem.
Thus, the question of whether all quasitraces are traces can be
reinterpreted as asking whether there is a noncommutative version
of Caratheodory's Theorem for C$^*$-algebras.
\end{rem}

We have stated and proved \autoref{thm:QTfunctionals} only for 
quasitracial \emph{states} on $A$ 
and \emph{normalized} functionals on $\Cu(A)$,
but the bijective correspondence also extends to 
lower-semicontinuous $[0,\I]$-valued quasitraces on $A$ and 
functionals on $\Cu(A)$, with the same formulas. 




Having established the correspondence between quasitraces on $A$
and functionals on $\Cu(A)$, we are ready to prove Cuntz's theorem \cite{Cun_dimension_1978}:
a simple, stably finite \ca\ always admits a quasitrace.
This was historically the first use of the Cuntz semigroup, and
the reason why it receives its name. (It should be mentioned
that Cuntz used a slightly different semigroup, with positive
elements taken from $M_\infty(A)$ instead of $A\otimes\K$, and
in fact he considered its Grothendieck enveloping group.)

We will use, without proof, that if $\tau\in \QT(A)$, then 
\[\ker(\tau):=\{x\in A\colon \tau(x^*x)=0\}\]
is an ideal in $A$. In particular, if $A$ is simple and  
$a\in A_+$ satisfies $\tau(a)=0$, then $a=0$. This follows, 
for example, from part~(3) of Lemma~3.5 in~\cite{Haa_quasitraces_2014}.

\begin{thm}\label{thm:Cuntz} (Cuntz).
 Let $A$ be a simple, unital \ca. Then $A$ is stably finite
 if and only if it admits a quasitracial state.
\end{thm}
\begin{proof}
Suppose that $A$ admits a quasitracial state $\tau$.
Without loss of generality, it suffices to show that $A$ is finite
(otherwise consider the quasitracial state 
$\tau\otimes\mathrm{tr}_n$ on $M_n(A)$ instead of $\tau$). 
Let $v\in A$ be an isometry. Then 
$\tau(vv^*)=\tau(v^*v)=\tau(1_A)=1_A$, and hence 
$\tau(1_A-vv^*)=0$. Thus $1_A-vv^*$ is a projection in the kernel
of $\tau$, which is an ideal in $A$. By simplicity, we must have
$vv^*=1_A$, so that $v$ is a unitary. 

We turn to the converse, so assume that $A$ is stably finite.
\vspace{.2cm}

\textbf{Claim:} \emph{for $n,m\in\N$, if $n[1_A]\leq m[1_A]$
then $n\leq m$.}
Suppose that $n[1_A]=m[1_A]$ for some $n< m\in\N$. Then
$
1_{M_m(A)}\precsim
1_{M_n(A)}$.
By \autoref{lma:CtzCompPjns}, we deduce that
\[1_{M_m(A)}\precsim_{\mathrm{MvN}}
1_{M_n(A)},\]
so there is an isometry $v\in M_m(A)$ such that 
$vv^*=\begin{psmallmatrix}1_{M_n(A)} & 0\\0 & 0_{m-n}\end{psmallmatrix}$. This contradicts stable finiteness of $A$
and proves the claim.
%
\vspace{.2cm}

Set $S_0=\{n[1_A]\colon n\in\N\}\subseteq\Cu(A)$. 
Define an additive map $\lambda_0\colon S_0\to [0,\infty]$ by $\lambda_0(n[1_A])=n$ for all $n\in\N$. Note that $\lambda_0$ is 
order-preserving by the previous claim.
Since $[0,\infty]$ is 
an injective object in the category of positively ordered monoids, we can extend it to a map $\widetilde{\lambda}\colon \Cu(A)\to [0,\infty]$ which preserves the zero, addition and order, but not necessarily suprema of increasing sequences. We fix this by taking its ``regularization''
$\lambda\colon \Cu(A)\to [0,\infty]$ given by
\[\lambda(a)=\sup \{\widetilde{\lambda}(a')\colon a'\ll a\}\]
for all $a\in \Cu(A)$.
One can check that $\lambda$ preserves suprema of increasing
sequences, so it is a functional.
Since $[1_A]\ll [1_A]$, we have $\lambda([1_A])=\widetilde{\lambda}([1_A])=1$,
so $\lambda$ is normalized.
By \autoref{thm:QTfunctionals}, $\lambda$ induces a 
quasitracial state on $A$, as desired.
\end{proof}

\section{A construction of the Jiang-Su algebra \texorpdfstring{$\mathcal{Z}$}{Z} using the Cuntz semigroup}
\label{sec:JS}

In this section, we use the Cuntz semigroup to 
define the Jiang-Su algebra $\mathcal{Z}$ via
generalized dimension drop algebras, 
and we compute its Cuntz semigroup. 

Recall that a function
$f\colon X\to S$ from a topological space $X$ into an $\Cu$-semigroup $S$ is said to be \emph{lower semicontinuous} if
for every $s\in S$, the set 
\[f^{-1}(\{t\in S\colon s\ll t\})\] 
is 
open in $X$. We write $\mathrm{Lsc}(X,S)$ for the set of all 
lower-semicontinuous functions $X\to S$. In some relevant cases, $\mathrm{Lsc}(X,S)$ is a $\Cu$-semigroup; see \cite{AntPerSan_pullbacks_2011}.

\begin{df}\label{df:GenDimDropAlg}
We define the \emph{generalized dimension drop algebra} (of type 2,3)
by 
\[
\mathcal{Z}_{2^\infty,3^\infty}=\left\{a\in C\big([0,1], M_{2^\infty}\otimes M_{3^\infty}\big)\colon \begin{aligned}
&a(0)\in M_{2^\infty}\otimes 1_{M_{3^\infty}},\\
& a(1)\in 1_{ M_{2^\infty}}\otimes M_{3^\infty}
\end{aligned}\right\}
\]
\end{df}

The Cuntz semigroup of $\mathcal{Z}_{2^\infty,3^\infty}$ can be
computed using \cite[Corollary~3.5]{AntPerSan_pullbacks_2011}:
\[
 \Cu(\mathcal{Z}_{2^\infty,3^\infty})\cong \left\{f\in \mathrm{Lsc}\big([0,1], \N[\tfrac{1}{6}]\cup (0,\I]\big)\colon 
 \begin{aligned}
&f(0)\in \N[\tfrac{1}{2}]\cup (0,\I]\\
& f(1)\in \N[\tfrac{1}{3}]\cup (0,\I]
\end{aligned}\right\},
\]
with pointwise order and addition (see, essentially,  \cite[Example~4.3]{AntPerSan_pullbacks_2011}).
Using these identifications, we 
define a Cu-morphism 
\[\phi\colon \Cu(\mathcal{Z}_{2^\infty,3^\infty})
\to \Cu(\mathcal{Z}_{2^\infty,3^\infty}),\] 
by setting
\[\phi(f)=\int_{[0,1]}f(t)\ dt\]
for $f\in \Cu(\mathcal{Z}_{2^\infty,3^\infty})$. This 
integral has to be interpreted appropriately: if $f$ is compact,
then $\phi(f)$ is the constant function with the corresponding 
compact element as its value, and 
otherwise it is the constant sequence with the corresponding 
non-compact element with that value.
The main feature of this map is that it is \emph{trace-collapsing}:
this means that if $\lambda,\lambda'\colon
\Cu(\mathcal{Z}_{2^\infty,3^\infty})\to [0,\I]$ are 
normalized functionals, then $\lambda\circ\phi=
\lambda'\circ\phi$. 

By a deep result of Robert \cite[Theorem~1.0.1]{Rob_classification_2012},
there exists a unital homomorphism 
\[\Phi\colon \mathcal{Z}_{2^\infty,3^\infty}\to \mathcal{Z}_{2^\infty,3^\infty} \ \ \mbox{ with } \ \ \Cu(\Phi)=\phi.\]

\begin{df}\label{df:JiangSu}
The \emph{Jiang-Su algebra} $\mathcal{Z}$ is defined to be the 
stationary direct limit
\[\mathcal{Z}=\varinjlim (\mathcal{Z}_{2^\infty,3^\infty},\Phi).\]
\end{df}

This definition of $\mathcal{Z}$ is not the original one by 
Jiang and Su \cite{JiaSu_simple_1999}, and this presentation of 
$\mathcal{Z}$ was essentially obtained by R{\o}rdam and Winter 
\cite{RorWin_algebra_2010}.

The type of the generalized dimension drop algebra is actually
irrelevant, and the only crucial ingredient is that 2 and 3 are coprime. Indeed, coprimeness 
implies that $\N[\frac{1}{2}] \cap \N[\frac{1}{3}]=\N$ and 
guarantees that the direct limit does not have any projections;
see the proof of \autoref{thm:CuZ}.

The tools developed in the previous sections allow us to 
compute the Cuntz semigroup of $\mathcal{Z}$. The computation of this semigroup was first carried out in \cite[Theorem 3.1]{PerTom_recast_2007}.

\begin{thm}\label{thm:CuZ}
The Cuntz semigroup of $\mathcal{Z}$ can be described as follows.
As sets, we have
\[\Cu(\mathcal{Z})=\N\sqcup (0,\I].\]
Addition and order are the expected ones on each component. 
For $n\in\N$, let $c_n\in\N$ be the corresponding
compact element, and for $x\in (0,\I]$ let $s_x$ be the 
corresponding non-compact element in $(0,\I]$. 
For $n\in\N$ and $x\in (0,\infty]$, 
we have:
\bi
\item $ s_x\leq c_n$ if and only if $x\leq n$;
\item $c_n\leq s_x$ if and only if $n< x$; 
\item $ c_n+s_x=s_{n+x}$.
\ei
\end{thm}
\begin{proof}
This is similar to \autoref{eg:CuCAR}.
We first want to understand the algebraic direct limit of the 
Cuntz semigroups. One can check that
the compact elements in $\Cu(\mathcal{Z}_{2^\infty,3^\infty})$
are the constant functions 
\[c\colon [0,1]\to 
 \N[\tfrac{1}{6}]\cup (0,\I]\]
whose constant value is compact (and thus in $\N[\tfrac{1}{6}]$). 
The endpoint conditions on $c$
give $c(0)\in \N[\frac{1}{2}]$ and $c(1)\in \N[\frac{1}{3}]$.
Since $c(0)=c(1)$ and $\N[\frac{1}{2}] \cap \N[\frac{1}{3}]=\N$, 
we deduce that the function $c$ must have constant value in $\N$. 

It follows that the direct limit of $(\Cu(\mathcal{Z}_{2^\infty,3^\infty}),\phi)$ in the category of partially ordered monoids is
$\N\sqcup (0,\I]$, with order and addition matching the ones in the 
statement. In principle one would need to complete, as was done 
in \autoref{eg:CuCAR}, but one can check that this is already a 
Cu-semigroup.
\end{proof}

\begin{rem}\label{rem:DividUnitZ}
From the description above, it follows that for every $n\in\N$
there exists $z\in (\mathcal{Z}\otimes\K)_+$ with 
\[n[z]\leq [1_{\mathcal{Z}}]\leq (n+1)[z].\]
One can take $z$ to be, for example, 
any representative of $s_{\frac{1}{n}}$.
\end{rem}

\begin{cor}
The Jiang-Su algebra $\mathcal{Z}$ is simple, nuclear, 
has a unique trace, and has no projections other 
than $0$ and $1$.
\end{cor}
\begin{proof}
It is easy to see that $\Cu(\mathcal{Z})$ is a simple Cu-semigroup
(for example, since $\infty_s=\infty_t$ for all 
nonzero $s,t\in \Cu(\mathcal{Z})$). Hence $\mathcal{Z}$ is simple 
by \autoref{cor:SimpleCuSmgp}. 
Moreover, it has a unique
normalized functional (given by $\lambda(c_n)=n$ for $n\in\N$
and $\lambda(s_x)=x$ for all $x\in (0,\I]$). Thus $\mathcal{Z}$ 
has a unique quasitracial state by \autoref{thm:QTfunctionals}, 
and since $\mathcal{Z}$ is exact (being nuclear), it follows that
$\QT(\mathcal{Z})=\mathrm{T}(\mathcal{Z})$, as desired. Finally, a projection in $\mathcal{Z}$ other than $0$ and $1$ would yield 
a compact element in $\Cu(\mathcal{Z})$ strictly between $0$
and $1$, which does not exist. 
\end{proof}

It was an old question of Kaplansky whether all simple \ca s
must contain nontrivial projections, and the first counterexample
was $C^*_\lambda(\mathbb{F}_2)$ (the other proposed candidate at 
the time, the irrational rotation algebra $A_\theta$, turned out to have many projections). 
The question then became whether
\emph{nuclear} simple \ca s must contain nontrivial projections, and 
Blackadar showed that the answer to this question is also negative.
The algebra $\mathcal{Z}$ constructed by 
Jiang and Su in \cite{JiaSu_simple_1999} 
is a further example of a simple, nuclear C$^*$-algebra without projections, and it attracted a great deal of attention
due to its prominent role in the classification programme. For example, since the Elliott invariant cannot
distinguish between $A$ and $A\otimes\mathcal{Z}$ (see 
\autoref{rem:ElltensorZ}), only $\mathcal{Z}$-stable \ca s can be 
expected to be classified using exclusively the invariant Ell.

\section{\texorpdfstring{$\mathcal{Z}$}{Z}-stability and strict comparison; the 
Toms-Winter conjecture}

The goal of this section is to prove a theorem of R{\o}rdam 
\cite{Ror_stable_2004},
relating two notions that do not in principle seem to be
related: $\mathcal{Z}$-stability on the one hand, and almost unperforation of the Cuntz semigroup 
on the other hand. We will also show that
almost unperforation in the Cuntz semigroup is equivalent to 
strict comparison, and make connections to the Toms-Winter
conjecture.

\begin{df}\label{df:Zstab}
Let $A$ be a \ca. We say that $A$ is \emph{$\mathcal{Z}$-stable}
if $A\otimes\mathcal{Z}\cong A$.
\end{df}

For most practical purposes, just knowing that an isomorphism
exists will not be of much help. As it turns out, $\mathcal{Z}$
has some remarkable properties that allow one to show that,
whenever an isomorphism $A\otimes\mathcal{Z}\cong A$ exists, 
then a \emph{nice} isomorphism exists, in the sense of the 
following result.

\begin{thm}\label{thm:Zssa}(Jiang-Su \cite{JiaSu_simple_1999}).
Let $A$ be a separable \ca. 
If $A\otimes\mathcal{Z}\cong A$, then 
there exists an isomorphism 
$\varphi\colon A\to A\otimes\mathcal{Z}$ which is 
approximately unitarily equivalent to the map $a\mapsto a\otimes 1_{\mathcal{Z}}$. In particular, $\varphi$ satisfies
$[\varphi(a)]=[a\otimes 1_{\mathcal{Z}}]$ for all $a\in (A\otimes\K)_+$.
\end{thm}

We will begin by relating $\mathcal{Z}$-stability to 
the following notion:

\begin{df}
Let $S$ be a Cu-semigroup. We say that $S$ is \emph{almost
unperforated} if whenever $s,t\in S$ and $n\in\N$ satisfy
$(n+1)s\leq nt$, then $s\leq t$.
\end{df}

The following is the main result of this section.
For $a\in A_+$ and $n\in\N$, recall that $n[a]=[a\otimes 1_n]$
in $\Cu(A)$. The following is 

\begin{thm}\label{thm:RordamStrComp} (R{\o}rdam \cite{Ror_stable_2004}). 
Let $A$ be a simple, separable, $\mathcal{Z}$-stable unital C$^*$-algebra. Then 
$\Cu(A)$ is almost unperforated.
\end{thm}
\begin{proof} 
Let $n\in\N$ and $a,b\in (A\otimes\K)_+$ satisfy
\[\tag{7.1}(n+1)[a]\leq n[b] \ \mbox{ in } \Cu(A).\]
Use \autoref{rem:DividUnitZ} to find $z\in (\mathcal{Z}\otimes\K)_+$ such that 
\[\tag{7.2} n[z]\leq [1_{\mathcal{Z}}]\leq (n+1)[z]  \ \mbox{ in } \Cu(\mathcal{Z})\]
Working in (the stabilization of) $A\otimes\mathcal{Z}$, we get
\[a\otimes 1_{\mathcal{Z}} \stackrel{(7.2)}{\precsim} a\otimes (z\otimes 1_{n+1})
 \sim (a\otimes 1_{n+1})\otimes z \stackrel{(7.1)}{\precsim}
 (b \otimes 1_n)\otimes z\stackrel{(7.2)}{\precsim} b\otimes 1_{\mathcal{Z}}.
\]
Use \autoref{thm:Zssa} to choose an isomorphism 
$\varphi\colon A\to A\otimes\mathcal{Z}$ satisfying 
$[\varphi(a)]=[a\otimes 1_{\mathcal{Z}}]$ for all $a\in (A\otimes\K)_+$.
Applying $\varphi$ to the subequivalence above, we deduce
that $[\varphi(a)]\leq [\varphi(b)]$. Since $\varphi$ is an 
isomorphism, this finishes the proof.\end{proof}

We put \autoref{thm:RordamStrComp} into context, by relating it
to the notion of \emph{strict comparison}.

\begin{df}
Let $A$ be a simple unital C$^*$-algebra. We say that $A$ has \emph{strict comparison
(of positive elements by quasitraces)} if whenever $a,b\in (A\otimes\K)_+$ are nonzero and satisfy 
$d_\tau(a)<d_\tau(b)$ for all $\tau\in \QT(A)$, then $a\precsim b$.
\end{df}

Strict comparison is essentially a property of $\Cu(A)$. Using 
\autoref{thm:QTfunctionals}, one can show that $A$ has strict 
comparison if and only if whenever $s,t\in \Cu(A)$ are nonzero
and satisfy $\lambda(s)<\lambda(t)$ for all $\lambda\in F_{[1]}(\Cu(A))$, then $s\leq t$. Less obvious is the fact that strict comparison
is equivalent to almost unperforation; see 
\cite[Proposition~3.2]{Ror_stable_2004}.
(One direction is easy,
namely almost unperforation implies strict comparison, while the 
converse requires an order-semigroup theoretic version of the 
Hahn-Banach theorem, similar to what was used in the 
proof of \autoref{thm:Cuntz}.)

In particular, \autoref{thm:RordamStrComp} asserts that 
$\mathcal{Z}$-stable \ca s have strict comparison. 
Perhaps surprisingly, it is conjectured that the converse is
true in the simple, nuclear setting:

\begin{cnj}\label{cnj:TW}(Toms-Winter regularity conjecture, see \cite{TomWin_villadsen_2009}, and also \cite[Conjecture~5.2]{Win_proceedings_2018}).
Let $A$ be a simple, separable, unital, nuclear \ca. The following are equivalent:
\be[{\rm (i)}]\item $\dim_{\mathrm{nuc}}(A)<\I$;
\item $A$ is $\mathcal{Z}$-stable;
\item $A$ has
strict comparison.
\ee 
\end{cnj}

Nuclear dimension for \ca s is a noncommutative version of
the covering dimension for topological spaces which was introduced by
Winter and Zacharias \cite{WinZac_nuclear_2010}, and we will not define this notion here. 

We state the conjecture in this form for historical reasons, but 
the fact that (i) and (ii) are equivalent is by now a theorem:
that (i) implies (ii) is an impressive
result of Winter \cite{Win_nuclear_2012}, while the implication
from (ii) to (i) has recently been obtained in an equally 
outstanding work by Castillejos, Evington, Tikuisis, White and Winter \cite{CasEviTikWhiWin_nuclear_2019}.
The fact that (ii) implies (iii) is precisely \autoref{thm:RordamStrComp}. The converse implication
remains open, although it is known in some cases,
such as whenever $\partial_eT(A)$ is compact and 
finite-dimensional thanks to the independent 
works of Kirchberg-R{\o}rdam \cite{KirRor_central_2014},
Toms-White-Winter \cite{TomWhiWin_stability_2015}, and 
Sato \cite{Sat_traces_2012}; or
whenever $A$ has stable rank one and locally finite 
nuclear dimension thanks to the work of Thiel 
\cite{Thi_ranks_2020}; 
or whenever $A$ has uniform property $\Gamma$, by the 
work of Castillejos, Evington, Tikuisis and White \cite{CasEviTikWhi_uniform_2022}.

\section{Toms' example and the relation with the Elliott invariant}\label{sec:Toms}

In this section, we present an example, due to Andrew Toms
\cite{Tom_classification_2008}, of two C$^*$-algebras that agree on the Elliott invariant (and more), yet they are distinguished by their Cuntz semigroup. This example shows the importance of the Cuntz semigroup outside the class of $\mathcal{Z}$-stable C$^*$-algebras, as a key ingredient for classification. We also relate the Cuntz semigroup with the Elliott invariant.


\begin{df}
Let $A$ be a unital C$^*$-algebra. The \emph{Elliott invariant} of $A$, denoted by $\Ell(A)$, consists of the $4$-tuple 
\[
\mathrm{Ell}(A)=\big((\KK_0(A),[1_A]), \KK_1(A), \mathrm{T}(A), r_A\big),
\]
 where $r_A\colon \KK_0(A)\times \mathrm{T}(A)\to \R$ is the pairing between K-theory and traces, defined as $r_A([p],\tau)=\tau(p)$
 for all projections $[p]\in \KK_0(A)$ and all $\tau\in\mathrm{T}(A)$.
\end{df}

It should be pointed out that Elliott's original formulation also included 
the positive cone $\KK_0(A)^+$ of $\KK_0(A)$ 
as part of the invariant. 
The modification we make here is 
inspired by the most recent approach to classification
\cite{CarGabSchTikWhi_classification_2022}, and
has the following convenient
consequence:

\begin{rem}\label{rem:ElltensorZ}
Let $A$ be a unital C$^*$-algebra. Then $\Ell(A)\cong\Ell(A\otimes\mathcal{Z})$.
\end{rem}

For $\mathcal{Z}$-stable C$^*$-algebras,
$\KK_0(A)^+$ can be recovered from the remaining
parts of $\Ell(A)$ as follows:
\[\KK_0(A)^+=\{x\in \KK_0(A)\colon r_A(x,\tau)>0 \mbox{ for all }
 \tau\in \mathrm{T}(A)\}\cup\{0\}.
\]
Thus, for $\mathcal{Z}$-stable C$^*$-algebras, 
there is no loss
of information when dropping $\KK_0(A)^+$ from the invariant. For the sake of this discussion, denote by 
$\widetilde{\mathrm{Ell}}(A)$ the invariant obtained from Ell$(A)$ by
adding the positive cone $\KK_0(A)^+$ of $\KK_0(A)$ 
as part of the invariant. While 
$\Ell(A)\cong\Ell(A\otimes\mathcal{Z})$ for \emph{every}
unital \ca\ $A$, this is no longer true if one 
considers $\widetilde{\mathrm{Ell}}$ instead, 
as $\mathcal{Z}$-stable \ca s have weakly unperforated
$\KK_0$-groups;
see \cite{GonJiaSu_obstructions_2000}.


The Elliott conjecture originally predicted that 
$\widetilde{\mathrm{Ell}}$ would be a complete invariant for the class of simple, separable, unital, nuclear C$^*$-algebras. The counterexamples to this conjecture were obtained by R\o rdam first, and later by Toms, which is the one we present below. These examples 
showed that $\mathcal{Z}$-stability is not automatic
for simple, separable, nuclear C$^*$-algebras, and this strongly suggested the need to add this condition as 
an assumption in the conjecture.
After many years of work by many researchers, which
culminated in \cite{KirPhi_embed_2000, Phi_classification_2000, GonLinNiu_class1_2020, GonLinNiu_class2_2020, EllGonLinNiu_class_2020, TikWhiWin_class_2017}, the following impressive classification theorem was obtained; see also \cite[Theorem~D]{Win_proceedings_2018}. Note that,
since $\mathcal{Z}$-stability is assumed, 
$\widetilde{\mathrm{Ell}}$ can be replaced by
$\mathrm{Ell}$ in the following statement:

\begin{thm}\label{thm:classification}
Let $A$ and $B$ be simple, separable, unital, nuclear, $\mathcal{Z}$-stable C$^*$-algebras satisfying the UCT. Then $A\cong B$ if and 
only if Ell$(A)\cong$ Ell$(B)$.
\end{thm}

Here, the UCT stands for the so-called Universal Coefficient Theorem and, in the theorem above, its assumption is potentially vacuous. This is one of the most important open problems in the area.

Let us turn the attention now to the construction offered by Toms. This was in turn based on previous work of Villadsen; see \cite{Vil_simple_1998, Vil_stable_1999}.
Let $(k_i)_{i\in\N}$ and $(n_i)_{i\in\N}$ be sequences of natural numbers, to be specified later. For each $i\in\N$, set $N_i=\prod_{j\leq i} n_j$ and
\[
A_i=M_{k_i}\otimes C\left( [0,1]^{6N_i}\right).
\]
Identify $[0,1]^{6N_i}$ with $\big([0,1]^{6N_{i-1}}\big)^{n_i}$ and for each $i$ and $l$ such that $1\leq l\leq n_i$,
let
\[
\pi_l^{(i)}\colon [0,1]^{6N_i}\to [0,1]^{6N_{i-1}}
\]
be the coordinate projection, given by $\pi_l^{(i)}(x_1,\dots,x_{n_i})=x_l$ for all $(x_1,\dots,x_{n_i})\in [0,1]^{6N_{i}}$.
For ease of notation, write $X_i=[0,1]^{6N_i}$. For each $i\in\N$, choose a dense sequence $(z_l^{(i)})$ in $X_i$, and choose points $x_1^{(i)},\dots, x_i^{(i)}\in X_i$ by setting $x_i^{(i)}=z_i^{(i)} $ and, if $1\leq j\leq i-1$, choose $x_j^{(i)}$ such that $\pi_1^{(j)}\pi_1^{(j+1)}\dots\pi_1^{i-2}\pi_1^{i-1}(x_j^{(i)})=z^{(i)}_{i+1-j}$.

Let us define $\phi_{i-1}\colon A_{i-1}\to A_i$ as follows. 
Given $f\in A_{i-1}$ and $x\in [0,1]^{6N_i}$, set
\[
\phi_{i-1}(f)(x)=\mathrm{diag}\Big(f\big(\pi_1^{(i)}(x)\big),\dots,f\big(\pi_{n_i}^{(i)}(x)\big),f\big(x_1^{(i-1)}\big),\dots,f\big(x_{i-1}^{(i-1)}\big)\Big).
\]
We now choose $n_i$ in such a way that $n_i$ is much larger than $i$ as $i\to\infty$ and such that for each $r\in\N$, there is $i_0$ with $r| (n_{i_0}+i_0)$. Set $n_1=1$, $k_1=4$ and $k_{i+1}=k_i(i+6N_i)$, and let $A=\varinjlim (A_i,\phi_i)$.

\begin{prop} 
\label{prop:Asimple} The C$^*$-algebra $A$ just constructed is simple,
separable, unital, nuclear, satisfies the UCT, has real rank one and stable rank one, and $A\otimes\mathcal{Z}$ 
is isomorphic to an AI algebra. 
\end{prop}
\begin{proof}
The choice of the points $x_j^{(i)}$, for $1\leq j\leq i$ and for each $i$, ensure that $A$ is simple, as the arguments in \cite{Vil_simple_1998} show. is $A$ is separable, unital, nuclear, and that it satisfies the UCT, are clear by construction.

Since $X_i$ is contractible for all $i\in\N$, we have 
\[(\KK_0(A_i), [1_{A_i}],\KK_1(A_i))\cong(\Z, k_i,\{0\}).
\]
This, coupled with the fact that $\KK$-theory is continuous and the choice of $(n_i)_{i\in\N}$, ensures that $(\KK_0(A),[1_A],\KK_1(A))\cong (\Q, 1,\{0\})$.
It follows that there is a simple
AI-algebra $B$ whose Elliott invariant is that of $A$. Moreover,
both $A\otimes\mathcal{Z}$ and $B$ satisfy the assumptions of 
\autoref{thm:classification}, so they are isomorphic.

By the results in \cite{Vil_simple_1998}, the real rank and the stable rank of $A$ are both equal to one, and thus the same is true for $B$.
\end{proof}

\begin{prop}
\label{prop:Anotaunp} The Cuntz semigroup of the C$^*$-algebra $A$ in \autoref{prop:Asimple} is not almost unperforated.
\end{prop}
\begin{proof}
It is enough to find positive elements $x,y\in A_1$ such that for all $i\in\N$ and some $\delta>0$, we have $11[\phi_{1,i}(x)]\leq 10[\phi_{1,i}(y)]$ in $\Cu(A_i)$, but $\Vert r\phi_{1,i}(y)r^*-\phi_{1,i}(x)\Vert>\delta$ for all $r\in A_i$.

To give the idea, we do this for $i=1$, that is, we show that $\Cu(A_1)$ is not almost unperforated, and partly reproduce the argument in \cite[Proof of Theorem 1.1]{Tom_classification_2008}. Notice that $A_1=C([0,1]^3\times [0,1]^3, M_4)$.
Set 
\[S=\big\{x\in [0,1]^3\colon \tfrac{1}{8}<\mathrm{dist}\big(x,\big(\tfrac{1}{2},\tfrac{1}{2},\tfrac{1}{2}\big)\big)<\tfrac{3}{8}\big\},\] so that $M_4(C_0(S\times S))$ is a hereditary subalgebra of $A_1$.

Let $\xi$ be a line bundle over $S^2$ with nonzero Euler class, and use $\theta_1$ to denote the trivial line bundle. Since $\xi\times\xi$ does not have zero Euler class, $\theta_1$ is not a sub-bundle of $\xi\times\xi$ over $S^2\times S^2$ (see \cite[Lemma 1]{Vil_simple_1998}). Considering $\xi\times\xi$ and $\theta_1$ as projections in $M_4(C_0(S^2\times S^2))$, we have $\Vert x(\xi\times \xi)x^*-\theta_1\Vert\geq \tfrac{1}{2}$ for all $x\in M_4(C_0(S^2\times S^2))$. Stability properties of vector bundles yield, on the other hand, that $11[\theta_1]\leq 10[\xi\times \xi]$.
Set 
\[S'=\big\{x\in S\colon\mathrm{dist}\big(x,\big(\tfrac{1}{2},\tfrac{1}{2},\tfrac{1}{2}\big)\big)<\tfrac{1}{4}\big\}\subseteq S\] and
let $f\in A_1$ be a positive scalar function supported on $S\times S$ which equals one on $S'\times S'$. Let $\rho$ denote the projection of $\overline S$ onto $S'$. One has, by restricting from $S\times S$ to $S'\times S'$, that 
\[
\Vert xf(\rho^*(\xi)\times\rho^*(\xi))x^*-f\theta_1\Vert\geq \tfrac{1}{2},
\]
for any $x\in A_1$, where $\rho^*$ is the pullback of $\xi$ via $\rho$, and also
\[
11[f\theta_1]\leq 10[f(\rho^*(\xi)\times \rho^*(\xi))]
\]
in $\Cu(A_1)$. This shows that $\Cu(A_1)$ is not almost unperforated.
\end{proof}

\begin{thm} 
\label{cor:ABnotiso} The C$^*$-algebras $A$ and $B=A\otimes\mathcal{Z}$ from \autoref{prop:Asimple} are not isomorphic, yet they have the same stable and real rank, and satisfy $\Ell(A)\cong\Ell(B)$.\footnote{In fact, they even satisfy $\widetilde{\Ell}(A)\cong\widetilde{\Ell}(B)$, which is a stronger statement since $A$ is not $\mathcal{Z}$-stable.}
\end{thm}
\begin{proof} 
We know from \autoref{prop:Anotaunp} that $\Cu(A)$ is not almost unperforated. On the other hand, $\Cu(B)$ is almost unperforated as $B$ is $\mathcal{Z}$-stable,
by \autoref{thm:RordamStrComp}. Hence, $A$ and $B$ cannot be isomorphic.

The argument in \autoref{prop:Asimple} shows that $A$ and $B$ have stable rank and real rank one, and they have the same Elliott invariant by \autoref{rem:ElltensorZ}.
\end{proof}

We now proceed to discuss the connection of the Cuntz semigroup
with the Elliott invariant for classifiable \ca s.
The first connection is given by the following
computation of the Cuntz semigroup of a $\mathcal{Z}$-stable 
C$^*$-algebra, obtained in \cite{BroPerTom_cuntz_2008, BroTom_three_2007}. For each element $x=[p]\in \V(A)$, we denote by $\widehat{x}\colon \QT(A)\to \R_{++}$ the continuous function defined by $\widehat{x}(\tau)=\tau(p)$ for $\tau\in$ QT$(A)$. We also denote by $\mathrm{LAff}(\QT(A))_{++}$ the semigroup of lower semicontinuous, affine functions defined on $\QT(A)$ with values on $(0,\infty]$.

\begin{thm}\label{thm:CuAtimesZ}
Let $A$ be a simple, separable, unital, stably finite 
$\mathcal{Z}$-stable
C$^*$-algebra.
Then 
\[\Cu(A)\cong \underbrace{\mathrm{V}(A)}_{\mathrm{compacts}} \sqcup \ \
\mathrm{LAff}(\QT(A))_{++},\]
with addition and order defined as follows:
\begin{enumerate}[{\rm (i)}]
\item The addition in $\V(A)$ is the usual addition and in $\mathrm{LAff}(\QT(A))_{++}$ is given by pointwise addition of functions. If $x\in\V(A)$ and $f\in\mathrm{LAff}(\QT(A))_{++}$, then $x+f=\widehat{x}+f$.
\end{enumerate}
For $x\in \V(A)$ and $f\in \mathrm{LAff}(\QT(A))_{++}$, we have
\begin{itemize}
\item[(ii)]  $x\leq f$ if $\widehat{x}(\tau)< f(\tau)$ for every $\tau\in \QT(A)$.
\item[(iii)] $f\leq x$ if $f(\tau)\leq \widehat{\tau}(x)$ for every $\tau\in \QT(A)$.
\end{itemize}
\end{thm}

\begin{rem}
For stably finite, nuclear $\mathcal{Z}$-stable C$^*$-algebras, the above allows one
to recover $\KK_0(A)$ as the Grothendieck group of $\V(A)=\Cu(A)_c$,
as well as $\QT(A)=F(\Cu(A))$. The pairing $r_A$ can be recovered then by evaluation of a functional on a compact element. In particular,
the pair $(\Cu(A),\mathrm{K}_1(A))$ is equivalent to Ell$(A)$.
\end{rem}

In contrast to \autoref{thm:CuAtimesZ}, the Cuntz semigroup
records essentially \emph{no} information in the purely infinite
case. Indeed, if $A$ is simple and purely inifnite, then 
$\Cu(A)\cong \{0,\infty\}$ regardless of the K$_0$-group of $A$.
This is easy to see: given $a,b\in A_+$ nonzero, find $r\in A$
such that $rbr^*=a$, which implies $a\precsim b$. It follows that 
all nonzero positive elements in (matrices over) $A$ are 
Cuntz-equivalent.

As we see from \autoref{thm:CuAtimesZ}, the Cuntz 
semigroup of $A$ does not in general contain any information
about its K$_1$-group. One can, however, 
recover K$_1(A)$ using the Cuntz semigroup of a 
different algebra, namely $A\otimes C(\T)$.\footnote{Thinking
of the Cuntz semigroup as a refined version of K$_0$, the 
idea behind this is that $\KK_0(A\otimes C(\T))\cong \KK_0(A)\oplus \KK_1(A)$.}
We will need the following construction.

Let $S$ be a $\Cu$-semigroup. Assume that the subset
$S_{\mathrm{nc}}$ of non-compact elements is an absorbing
subsemigroup, in the sense that $S_{nc}+S\subseteq S_{\mathrm{nc}}$.
(This is always the case if $S=\Cu(A)$ for a simple, separable,
unital, stably finite $\mathcal{Z}$-stable \ca\ by \autoref{thm:CuAtimesZ}.)
Denote by $S_{\mathrm{c}}$ the subsemigroup of compact elements and
$S_\mathrm{c}^*=S_\mathrm{c}\setminus\{0\}$. Let $G$ be an abelian
group and consider the semigroup
\[
S_G=(\{0\}\sqcup (G\times S_{\mathrm{c}}^*))\sqcup
S_{\mathrm{nc}}\,,
\]
with natural operations in both components, and $(g,x)+y=x+y$
whenever $x\in S_{\mathrm{c}}^*$, $y\in S_{\mathrm{nc}}$, and $g\in
G$. This semigroup can be ordered as follows:
\begin{enumerate}[{\rm (i)}]
\item For $x,y\in S_{\mathrm{c}}^*$, and $g,h\in G$, we have $(g,x)\leq (h,y)$ if and only if $x=y$ and $g=h$, or else
$x<y$.
\item For $x\in S_{\mathrm{c}}^*$, $y\in S_{\mathrm{nc}}$, $g\in
G$, $(g,x)$ is comparable with $y$ if $x$ is comparable with $y$.
\end{enumerate}
The proof of the following is left as an exercise.

\begin{lma}
\label{lem:tontada} Let $S$ be an object of \textbf{Cu} such that
$S_{\mathrm{nc}}$ is an absorbing subsemigroup. If $G$ is an abelian
group, then $S_G$ is also an object of \textbf{Cu}.
\end{lma}

For a  C$^*$-algebra $A$, let us denote $\Cu_\mathbb{T}(A)=\Cu(A\otimes C(\mathbb{T}))$. 

\begin{thm}
\label{thm:zstable}(\cite[Theorem 3.8]{AntDadPerSan_recovering_2014}). Let $A$ be a separable, finite,
$\mathcal{Z}$-stable C$^*$-algebra. Then, there is an
order-isomorphism
\[
\Cu_{\mathbb{T}}(A)\cong (\{0\}\sqcup (\mathrm{K}_1(A))\times \mathrm{V}(A)^*
)\sqcup\Lsc_{\mathrm{nc}}(\T,\Cu(A))\,,
\]
where $\mathrm{V}(A)^*=\mathrm{V}(A)\setminus\{0\}$, and where the right hand side is ordered as above.
\end{thm}

Let us interpret this in terms of the theory of classification. To this end, let us write $\mathrm{\textbf{Ell}}$ to denote the
category whose objects are $4$-tuples of the form
\[
((G_0,u),G_1,X,r)\,,
\]
where $G_0$ is a (countable) abelian group with distinguished element
$u$, $G_1$ is a (countable) abelian
group, $X$ is a (metrizable) Choquet simplex, $r\colon X\to
\mathrm{St}(G_0,u)$ is an affine map, where $\mathrm{St}(G_0,u)$
denotes the state space of $(G_0,u)$, such that, if we set 
\[G_0^+=\{g\in G_0\colon r(x)(g)>0 \mbox{ for all } x\in X\}\cup \{0\},\]
then $(G_0, G_0^+,u)$ is a simple, partially ordered group with
order unit $u$.

The morphisms in $\mathrm{\textbf{Ell}}$ between $((G_0,u),G_1,X,r)$ and
$((H_0,v),H_1,Y,s)$ are given by triples $(\theta_0,\theta_1, \gamma)$, where $\theta_0\colon G_0 \to H_0$ is a
morphism of groups with $\theta_0(u)=v$, the map
$\theta_1\colon G_1\to H_1$ is a morphism
of groups, and $\gamma\colon Y\to X$ is an affine and
continuous map such that $r\circ\gamma=\theta_0^*\circ s$, where
$\theta_0^*\colon \mathrm{St}(H_0,v)\to \mathrm{St}(G_0,u)$ is the
naturally induced map at the level of states.

Let $\mathbf{C^*_{\mathcal{Z}}}$ denote the class of simple, separable, unital, nuclear, finite $\mathcal{Z}$-stable
C$^*$-algebras. The Elliott invariant then naturally defines a
functor
\[
\mathrm{Ell}\colon\mathbf{\mathbf{C^*_{\mathcal{Z}}}}\to \mathrm{\textbf{Ell}}.
\]
We now define
\[
\mathrm{F}\colon \mathrm{\textbf{Ell}}\to \CatCu
\]
as follows. If $\mathcal{E}=((G_0,u),G_1,X,r)$ is an object of
$\mathrm{\textbf{Ell}}$ and $G_0^+$ is defined as above, then notice first that $G_0^+\sqcup \mathrm{LAff}(X)_{++}$ is an object 
of $\CatCu$ (see,
for example, \cite[Lemma 6.3]{AntBosPer_completion_2011}).
Set $G_0^{++}=G_0^+\setminus\{0\}$ and 
\[
\mathrm{F}(\mathcal{E})=(\{0\}\sqcup (G_1\times G_0^{++}))\sqcup\Lsc_{\mathrm{nc}}(\T,
G_0^+\sqcup \mathrm{LAff}(X)_{++})\,,
\]
with addition given by
$(g+f)(x)=r(x)(g)+f(x)$, for all $g\in G$, $f\in\mathrm{LAff}(X)$ and
$x\in X$. 
It follows from \autoref{lem:tontada} that $\mathrm{F}(\mathcal{E})$ is also an object of $\CatCu$. 

\begin{lma}\label{lma:FunctorF}
$\mathrm{F}$ is a functor, and its corestriction $\mathrm{F}\colon \mathrm{\textbf{Ell}}\to \mathrm{F}(\mathrm{\textbf{Ell}})$ is full, faithful and dense.
\end{lma}
\begin{proof}
We only check that, if
\[
(\theta_0,\theta_1,\gamma)\colon((G_0,u),G_1,X,r)\to((H_0,v),H_1,Y,s)
\]
is a morphism in $\mathrm{\textbf{Ell}}$ and $f\colon\T\to
G_0^+\sqcup\mathrm{LAff} (X)_{++}$ is non-compact, then
$(\theta_0\sqcup\gamma^*)\circ f\colon\T\to
H_0^+\sqcup\mathrm{LAff}(Y)_{++}$ is also non-compact. Here
\[
\theta_0\sqcup\gamma^*\colon G_0^+\sqcup\mathrm{LAff}(X)_{++}\to
H_0^+\sqcup\mathrm{LAff}(Y)_{++}
\]
is defined as $\theta_0$ on $G_0^+$ and $\gamma^*$ on
$\mathrm{LAff}(X)_{++}$.

By way of contradiction, if $(\theta_0\sqcup\gamma^*)\circ f$ is compact, then there is $h\in
H_0^+$ such that $\theta_0(f(\T))=\{h\}$ and $f(\T)\subseteq G_0^+$.
As $f$ is non-compact and lower semicontinuous, there are $s,t\in
\T$ with $f(t)<f(s)$, whence $f(s)-f(t)\in G_0^{++}$ is an
order-unit. Thus, there exists $n\in\N$ with $f(s)\leq
n(f(s)-f(t))$. After applying $\theta_0$, we obtain that $h\leq 0$,
hence $h=0$. But this is not possible since, as $f$ is not
constant, it takes some non-zero value $a$, which will be an
order-unit with $\theta_0(a)=0$, contradicting that $\theta_0(u)=v$.
\end{proof}

As a consequence, we get the following.

\begin{thm}[{\cite{Tik_cuntz_2011}}] 
\label{thm:equivalence}
Upon restriction to the class of unital, 
simple, separable and finite $\mathcal
Z$-stable algebras, $\mathrm{Ell}$ is a classifying
functor if, and only if, so is $\Cu_{\T}$.
\end{thm}
\begin{proof}
Since F$\colon \mathrm{\textbf{Ell}}\to \mathrm{F}(\mathrm{\textbf{Ell}})$ is a full, faitful
and dense functor by \autoref{lma:FunctorF}, 
it is a general fact in category theory
that it yields an equivalence of
categories, so that there exists another
functor $\mathrm{G}\colon \mathrm{F}(\mathrm{\textbf{Ell}})\to 
\mathrm{\textbf{Ell}}$ such that $\mathrm{F}\circ \mathrm{G}$
and $\mathrm{G}\circ \mathrm{F}$ are naturally equivalent to the respective
identities.
This, together with \autoref{thm:zstable} and \cite[Corollary 5.7]{BroPerTom_cuntz_2008} implies that 
there are natural equivalences of functors
\[
\mathrm{F}\circ\mathrm{Ell}\simeq\Cu_{\T}\,\text{ and }\,\mathrm{Ell}\simeq
\mathrm{G}\circ\Cu_{\T},
\]
which implies the result.
\end{proof}

\section{Additional axioms and properties}\label{sec:Axioms}

In this section, we introduce two important new axioms and a property for Cu-semigroups. As it turns out, the axioms are satisfied by the Cuntz semigroup of any C$^*$-algebra, whereas the property, that measures cancellation in the Cuntz semigroup, holds for C$^*$-algebras of stable rank one.

\begin{df}
Let $S$ be a $\Cu$-semigroup. We say that $S$ has \emph{weak cancellation} provided $x+z\ll y+z$ in $S$ implies $x\ll y$, for all $x,y,z\in S$.
\end{df}

Not every Cuntz semigroup is weakly cancellative. If, for example, $A$ is a purely infinite simple C$^*$-algebra, then $\Cu(A)=\{0,\infty\}$, where necessarily $\infty$ is compact. Hence $\infty+\infty\ll 0+\infty$, but it is not true that $\infty\ll 0$. It is easy to show that weak cancellation passes to quotients, hence another example where weak cancellation fails in the non purely infinite setting may be obtained by taking a stably finite C$^*$-algebra $A$ with a purely infinite simple quotient, for example the cone over $\mathcal{O}_2$. 
An example with finite stable rank is the Toeplitz algebra $\mathcal{T}$, since it contains a non-unitary isometry $s$, so that $p:=ss^*$ is a projection satisfying $p\oplus (1-p)=p+(1-p)=1\sim p$ with $p\neq 1$.

We record below an equivalent formulation for weak cancellation that is repeatedly used in the literature. The proof is left as an exercise; see \cite[Lemma 2.5]{AntPerRobThi_stable_2022} for the argument.

\begin{lma}
Let $S$ be a $\Cu$-semigroup. The following conditions are equivalent:
\begin{enumerate}[{\rm (i)}]
\item $S$ has weak cancellation.
\item If $x,y,z\in  S$ safisfy $x+z\ll y+z$, then $x\leq y$.
\item If $x,y,z,z'\in S$ satisfy $x+z\leq y+z'$ and $z'\ll z$, then $x\leq y$. 
\end{enumerate}
\end{lma}

The following theorem is due to R\o rdam and Winter \cite{RorWin_algebra_2010}. 

\begin{thm}
\label{wc} Let $A$ be a C$^*$-algebra of stable rank one. Then $\Cu(A)$ has weak cancellation.
\end{thm}
\begin{proof}
By part~(iii) of \autoref{prop:sr1}, we may assume that $A$ is stable. Let $x$, $y$, $z$ in $\Cu(A)$ satisfy $x+z\ll y+z$. Assume first that $z$ is compact, that is, $z\ll z$. Choose elements $a,b\in A_+$
with $x=[a]$, $y=[b]$, and use \autoref{prop:CpctSr1} 
to choose a projection $p\in A$ with $z=[p]$. 
We may assume that
$a,b$ are both orthogonal to $p$. Let $\varepsilon>0$. 
Using \autoref{thm:StableUnitImpl},
choose a unitary $u$ (in the unitization of $A$) such that
$u((a-\varepsilon)_++p)u^*\in A_{b+p}$. 
Then $p$ and $upu^*$ are projections in $A_{b+p}$ whose Cuntz classes in $A$ agree.
Using that $A_{b+p}$ is a hereditary subalgebra of $A$, one 
readily checks that their Cuntz classes in $A_{b+p}$ 
also agree (this is almost identical to the 
proof of \autoref{lma:IdealAIdealCu}). 
Note that $A_{b+p}$ has stable rank by part~(ii)
of \autoref{prop:sr1}, and hence is stably finite by
part~(i) of the same lemma. By \autoref{lma:CtzCompPjns}
and the comments after it, we deduce that $p\sim_{\mathrm{MvN}} q$ in $A_{b+p}$. Using again that 
$A_{b+p}$ has stable rank one, part~(iv) of 
\autoref{prop:sr1} gives us a unitary $v$ in the unitization of $A_{b+p}$ such that $vpv^*=upu^*$.

Observe now that $vu(a-\varepsilon)_+u^*v^*$ belongs to $A_{b+p}$ and is orthogonal to $p$. Thus 
\[vu(a-\varepsilon)_+u^*v^*\in (1-p)A_{p+b}(1-p)=A_b.\] 
This shows that $(a-\varepsilon)_+\precsim b$ and since $\varepsilon$ is arbitrary, we conclude that $a\precsim b$, that is, $x\leq y$.

For the general case, choose $z'\ll z$ such that $x+z\leq  y+z'$. Choose representatives as before $x=[a]$, $y=[b]$, $z=[c]$. We may assume that $z'=[(c-\varepsilon)_+]$ for some $\varepsilon>0$, and so
\[\tag{10.1}
a\oplus c\precsim b\oplus (c-\varepsilon)_+
\] 
Define
\[h_\varepsilon(t)=\begin{cases}
		1, & \text{if } t=0 \\
        \text{linear}, &\text{if } t\in [0,\varepsilon]\\
        0, &\text{if } t\geq\varepsilon.
       \end{cases}
\]
Then
$\ep\leq c+h_\varepsilon(c)\leq 1$. Using part~(ii) of~\autoref{cor:FuncCalc} at the second step and using 
\autoref{lma:orth} at the third, we get
\[\tag{10.2}(c-\ep)_++h_\varepsilon(c)\leq c+h_\varepsilon(c)\sim 1\precsim c\oplus h_\ep(c).\]
On the other hand, $(c-\ep)_+\perp h_\varepsilon(c)$ and
thus \autoref{lma:orth} implies that
\[\tag{10.3}(c-\ep)_+ \oplus h_\ep(c)\sim (c-\ep)_+ + h_\ep(c)
\]
Combining these things, we get 
\begin{align*}
a\oplus 1&\stackrel{\mathrm{(10.2)}}{\precsim} a\oplus c\oplus h_\varepsilon(c)\\
&\stackrel{\mathrm{(10.1)}}{\precsim}  b\oplus (c-\varepsilon)_+\oplus h_\varepsilon(c)\\ &\stackrel{\mathrm{(10.3)}}{\precsim} b\oplus (c-\ep)_++h_\varepsilon(c)\\
&\stackrel{\mathrm{(10.2)}}{\precsim} b\oplus 1,
\end{align*}
and therefore $a\precsim b$ by the first part of the proof.
\end{proof}

\begin{rem}
\label{rmk:wc}
Notice that if $S$ is a $\Cu$-semigroup with weak cancellation, then $S$ has cancellation of compact elements. In particular, if $A$ is a C$^*$-algebra of stable rank one, then its Cuntz semigroup $\Cu(A)$ has cancellation of compact elements (in fact, a close inspection of the proof of \autoref{wc} reveals one proves this first before concluding weak cancellation).
\end{rem}

An abelian monoid $S$ is said to be \emph{algebraically ordered} provided $x\leq y$ in $S$ exactly when there is $z\in S$ with $x+z=y$. The order in a Cu-semigroup is in general not algebraic
(consider, for example, the Cuntz semigroup of $\mathcal{Z}$; see
\autoref{thm:CuZ}),
and the axiom below measures how far this is from being the case.

\begin{df}
\label{pgr:O5}
Let $S$ be a $\Cu$-semigroup. We say that $S$ satisfies (O5) (or has \emph{almost algebraic order}) if, whenever $x+z\leq y$ and $x'\ll x$, $z'\ll z$, there is $w\in S$ such that 
\[x'+w\leq y\leq x+w \ \ \mbox{ and } \ \ z'\leq w.\]
\end{df}

The formulation above is taken from \cite{AntPerThi_tensor_2018}, and has the advantage that it behaves well with respect to inductive limits of $\Cu$-semigroups; see \cite[Theorem 4.5]{AntPerThi_tensor_2018}.

This axiom appears in a different form in other papers, namely by taking $z=0$. In fact, this was the original formulation in \cite{RorWin_algebra_2010}, which therefore reads: whenever $x\leq y$ and $x'\ll x$, there is $w\in S$ such that $x'+w\leq y\leq x+w$. Let us temporarily refer to this version as axiom $\widetilde{\mathrm{(O5)}}$.

\begin{lma}
Let $S$ be a $\Cu$-semigroup. Then (O5) implies $\widetilde{\mathrm{(O5)}}$ and, in the presence of weak cancellation, the converse holds.
\end{lma}
\begin{proof}
The first part of the statement is trivial. Assume, for the converse, that  $S$ is a $\Cu$-semigroup that has weak cancellation and satisfies $\widetilde{\mathrm{(O5)}}$. We are to show that $S$ satisfies (O5). To this end, let $x',x,z',z,y\in S$ satisfy $x+z\leq y$ and $x'\ll x$, $z'\ll z$. First choose $x'\ll x''\ll x$, $z'\ll z''\ll z$, and $y'\ll y$ such that $x''+z''\ll y'$. 
Since $x'\ll x''\leq y$, by $\widetilde{\mathrm{(O5)}}$ there is $w\in S$ such that $x'+w\leq y\leq x''+w$. Using weak cancellation on the inequality
\[
x''+z''\ll y'\ll y\leq x''+w,
\] 
we obtain $z''\ll w$. Thus $z'\leq w$. On the other hand $x'+w\leq y\leq x''+w\leq x+w$, as desired.
\end{proof}

\begin{rem}
\label{rem:compl}
Suppose that $S$ satisfies (O5), and let $x\in S_c$ and $y\in S$
satisfy $x\leq y$. 
Applying (O5) to $x\ll x\leq y$, we find $w\in S$ with $x+w=y$. 
This shows that compact elements can always be complemented and thus the subsemigroup $S_c$ of compact elements is algebraically ordered.
\end{rem}

The reason why (O5) is significant is recorded below:

\begin{prop}
\label{O5}
Let $A$ be a C$^*$-algebra. Then $\Cu(A)$ satisfies (O5).
\end{prop}
\begin{proof}
We prove the simpler version $\widetilde{\mathrm{(O5)}}$, mostly
following the original proof in \cite{RorWin_algebra_2010}. (For the general argument, see \cite[Proposition 4.6]{AntPerThi_tensor_2018}.) Thus assume that $x'\ll x\leq y$ in $\Cu(A)$, and set $x=[a]$, $y=[b]$. For simplicity we assume that $x'=[(a-2\varepsilon)_+]$ for some $\varepsilon>0$. 

By \autoref{thm:Rordam}, there is $v\in A$ such that $(a-\varepsilon)_+=v^*v$ and $vv^*\in A_b$. Let $h_\varepsilon$ be defined as in the proof of \autoref{wc}, so that $h_\varepsilon(vv^*)\perp (vv^*-\varepsilon)_+$
and $\ep\leq vv^*+h_\ep(vv^*)$. The latter inequality 
implies that $\ep b\leq b^{\frac{1}{2}}(vv^*+h_\ep(vv^*))b^{\frac{1}{2}}$
and thus
\[\tag{10.4} b\precsim b^{\frac{1}{2}}(vv^*+h_\ep(vv^*))b^{\frac{1}{2}}.\]
Let $c=h_\varepsilon(vv^*)bh_\varepsilon(vv^*)$ and set $w=[c]\in \Cu(A)$. Note
that 
\[\tag{10.5} c\sim b^{\frac{1}{2}}h_\ep(vv^*)^2 b^{\frac{1}{2}}\]
by part~(iv) of \autoref{cor:FuncCalc}.
Since both $(vv^*-\varepsilon)_+$ and $c$ belong to $A_b$,
we have 
\[\tag{10.6}(vv^*-\varepsilon)_++c\precsim b\]
by \autoref{prop:InHerSubalg}.
Using \autoref{lma:CutDwnSymm} at the second step, and using 
\autoref{lma:orth} at the third step, we get
\[
(a-2\varepsilon)_+\oplus c= (v^*v-\varepsilon)_+\oplus c\sim (vv^*-\varepsilon)_+\oplus c\sim (vv^*-\varepsilon)_++c
\stackrel{\mathrm{(10.6)}}{\precsim} b,
\]
so that $x'+w\leq y$. 

It remains to show that $b\precsim a\oplus c$, which
gives $y\leq x+w$. 
Using \autoref{lma:orth} at the fourth step, and using $v^*v=(a-\ep)_+$ and part~(iv) of \autoref{cor:FuncCalc} at the fifth step, we get
\begin{align*}
b&\stackrel{\mathrm{(10.4)}}{\precsim} b^{\frac{1}{2}}(vv^*+h_\varepsilon(vv^*)^2)b^{\frac{1}{2}}\\
& \ \ = \ \ b^{\frac{1}{2}}vv^*b^{\frac{1}{2}}+b^{\frac{1}{2}}h_\varepsilon(vv^*)^2b^{\frac{1}{2}}\\
&\stackrel{\mathrm{(10.5)}}{\precsim} vv^*+c \precsim vv^*\oplus c\\
& \ \ \sim \ \ (a-\varepsilon)_+\oplus c\precsim a\oplus c.\qedhere
\end{align*}
\end{proof}

An abelian monoid $S$ is said to have \emph{Riesz decomposition} if whenever $x\leq y+z$ in $S$, there are elements $x_1, x_2\in S$ with $x=x_1+x_2$ and $x_1\leq y$, $x_2\leq z$. Again, Cu-semigroups do not 
in general have Riesz decomposition (and $\Cu(\mathcal{Z})$ is an
example).
The following axiom measures a certain degree of Riesz decomposition in a $\Cu$-semigroup. 

\begin{df}
\label{pgr:O6}
Let $S$ be a $\Cu$-semigroup. We say that $S$ satisfies (O6), or that
it has \emph{almost Riesz decomposition}, if whenever $x',x,y,z\in S$ satisfy $x'\ll x\leq y+z$, then there are $z,t\in S$ with 
\[x'\leq z+t, \ \ \ \ z\leq x, y, \ \ \mbox{ and } \ \ t\leq x, z.\]
\end{df}

Again, (O6) is satisfied by the Cuntz semigroup of every C$^*$-algebra; see \cite[Proposition 5.1.1]{Rob_cone_2013}.

\begin{prop}
\label{O6}
Let $A$ be a C$^*$-algebra. Then $\Cu(A)$ satisfies (O6).
\end{prop}
\begin{proof}
We may assume that $A$ is stable.  Suppose that $x\leq y+z$ in $\Cu(A)$, and $x'\ll x$. Choose representatives $a,b,c\in A$ such that $x=[a]$, $y=[b]$, and $z=[c]$. We may assume that $b\perp c$ and for simplicity we also assume that  $x'=[(a-\varepsilon)_+]$ for some $\varepsilon>0$.

By \autoref{thm:Rordam}, there are $v\in A$ and $\delta>0$ such that $(a-\varepsilon)_+=v^*v$ and $vv^*\in A_{(b+c-\delta)_+}$. Define a continuous function by
\[g_\delta(t)=\begin{cases}
		0, & \text{if } t=0 \\
        \text{linear}, &\text{if } t\in [0,\delta]\\
        1, &\text{if } t>\delta.
       \end{cases}
\]
Then $g_\delta(b+c)vv^*=vv^*$. Using at the first step 
that 
$g_\delta(b+c)=g_\delta(b)+g_\delta(c)$, which holds because 
$b\perp c$, and using 
\autoref{rem:x+y} at the second step,
we have
\begin{align*}
g_\delta(b+c)&vv^*g_\delta(b+c) \\ 
&=g_\delta(b)vv^*g_\delta(b)+g_\delta(b)vv^*g_\delta(c)+g_\delta(c)vv^*g_\delta(b)+g_\delta(c)vv^*g_\delta(c)\\
&\leq 2(g_\delta(b)vv^*g_\delta(b)+g_\delta(c)vv^*g_\delta(c)).
\end{align*}

Using this at the third step and that $b\perp c$ again at the fourth step, we obtain
\begin{align*}
[(a-\varepsilon)_+]=[vv^*]= & [g_\delta(b+c)vv^*g_\delta(b+c)]\\
\leq & [g_\delta(b)vv^*g_\delta(b)+ g_\delta(c)vv^*g_\delta(c)]\\
= & [g_\delta(b)vv^*g_\delta(b)]+ [g_\delta(c)vv^*g_\delta(c)].
\end{align*}
The proof is finished by setting $z:=[g_\delta(b)vv^*g_\delta(b)]=[v^*g_\delta(b)^2v]\leq [(a-\varepsilon)_+], [b]$ and $t:=[g_\delta(c)vv^*g_\delta(c)]\leq [(a-\varepsilon)_+], [c]$.
\end{proof}

\section{\texorpdfstring{C$^*$}{C*}-algebras with stable rank one}
\label{sec:sr1}

In this section we explore some properties satisfied by the Cuntz semigroups of C$^*$-algebra with stable rank. One obtains in particular solutions to three open problems on the Cuntz semigroup that, at the same time, reflect on the structure of such algebras. In these notes, we will discuss two of the three problems mentioned (see Sections \ref{sec:BH} and \ref{sec:functionals}). Large portions of this section are taken from \cite{AntPerRobThi_stable_2022}.

We will need the following variation of the axiom (O6), introduced
by Thiel in \cite{Thi_ranks_2020}:

\begin{df}
\label{pgr:O6+}
Let $S$ be a $\Cu$-semigroup. We say that $S$ satisfies (O6+) if whenever $x,y,z\in S$ satisfy $x\leq y+z$ and $u',u,v',v\in S$ are such that $u'\ll u\leq x,y$, and $v'\ll v\leq x,z$, then there are elements $s,t\in S$ such that
\[
x\leq s+t, \ \ \ \  u'\ll s\leq x,y, \ \ \mbox{ and } \ \ v'\ll t\leq x,z.
\]
\end{df}

It was shown in the proof of \cite[Lemma 6.3]{Thi_ranks_2020} that (O6+) has a simpler equivalent formulation. We  reproduce the argument below for convenience, since we shall use this later.

\begin{lma}
Let $S$ be a $\Cu$-semigroup. Then $S$ satisfies (O6+) if and only if the following (one-sided) property holds: whenever $x,y,z\in S$ satisfy $x\leq y+z$ and $u',u\in S$ are such that $u'\ll u\leq x,y$, then there is $s\in S$ such that $x\leq s+y$ and $u'\ll s\leq x,y$. 
\end{lma}
\begin{proof}
It is clear that (O6+) implies the property in the statement.
Conversely, assume
the above property and let $x,y,z,u',u,v',v\in S$ satisfy $x\leq y+z$, $u'\ll u\leq x$, and $v'\ll v\leq x,z$. Applying the property in the statement to $x\leq y+z$ and $u'\ll u\leq x,y$, we obtain $s\in S$ such that $x\leq s+y$ and $u'\ll s\leq x,y$. A second application of said property to $x\leq s+y$ and $v'\ll v\leq x,z$ yields an element $t\in S$ with $x\leq s+t$ and $v'\ll t\leq x,z$, as desired.
\end{proof}

The proof of the following is rather involved, hence we omit the details.

\begin{thm}[{\cite[Theorem 6.4]{Thi_ranks_2020}}]
\label{thm:O6+}
Let $A$ be a C$^*$-algebra of stable rank one. Then $\Cu(A)$ satisfies (O6+).
\end{thm}

Next, we observe that the assumption of stable rank one is necessary
in \autoref{thm:O6+}. The following is \cite[Example 6.7]{Thi_ranks_2020}

\begin{eg}
The Cuntz semigroup of the C$^*$-algebra $C(S^2)$ does not satisfy (O6+), and
it is well-known that $\mathrm{sr}(C(S^2))=2$.
\end{eg}
\begin{proof}
Denote by $\Lsc (S^2,\NN)$ the monoid of lower semicontinuous functions from $S^2$ with values in $\NN$, equipped with pointwise order and addition. For any open subset $U\subseteq S^2$, we use $\mathbf{1}_U$ to denote the characteristic function of $U$, which is an element of $\Lsc(S^2,\NN)$. Note that the non-compact elements of Lsc$(S^2,\overline{\N})$ are given by 
\[\Lsc(S^2,\NN)_{\mathrm{nc}}=\Lsc(S^2,\NN)\setminus\{n\mathbf{1}_{S^2}\colon n\geq 1\}.\] 
Using \cite[Theorem 1.2]{Rob_cone_2013} it is possible to show that
\[
\Cu(C(S^2))\cong (\N_{>0}\times\Z)\sqcup\Lsc(S^2,\NN)_{\mathrm{nc}}.
\]
Addition and order may be described as follows. Elements in each one of the components of the disjoint union are added as usual and ordered also as usual. If $(n,m)\in \N_{>0}\times\Z$ and $f\in \Lsc(S^2,\NN)_{\mathrm{nc}}$, then $(n,m)+f=n\mathbf{1}_{S^2}+f$.

Next, $(n,m)\leq f$ if and only if $n\mathbf{1}_{S^2}\leq f$, and $f\leq (n,m)$ if and only if $f\leq n\mathbf{1}_{S^2}$.

Choose open subsets of $U,V\subseteq S^2$ with $\overline{U}\subseteq V$. Then $\mathbf{1}_{U}\ll\mathbf{1}_{V}$, and they are non-constant functions. We have 
\[(1,0)\leq (1,1)+n\mathbf{1}_{U} \ \ \mbox{ and } \ \ \mathbf{1}_{U}\ll\mathbf{1}_{V}\leq (1,0), (1,1).\] 
To reach a contradiction, assume that $\Cu(C(S^2))$ satisfies (O6+). Then there is $s\in\Cu(C(S^2))$ such that $(1,0)\leq s+\mathbf{1}_{U}$ and $\mathbf{1}_{U}\leq s\leq (1,0), (1,1)$.

Since $s\leq (1,0),(1,1)$ we see that $s\in \Lsc(S^2,\NN)_{\mathrm{nc}}$ and thus there is $x\in S^2$ such that $s(x)=0$. But since $\mathbf{1}_{U}\leq s$, we see that $x\notin U$, which implies that $(s+\mathbf{1}_{U})(x)=0$. Therefore $(1,0)$ is not dominated by $s+\mathbf{1}_{U}$, which is impossible.
\end{proof}

\begin{rem}
It is natural to ask whether (O6+) follows from (O6) and weak cancellation (the latter, as we have shown in \autoref{wc}, holds for C$^*$-algebras of stable rank one). That is however not the case. Again, $C(S^2)$ is a counterexample, since by the computation above one can show its Cuntz semigroup is weakly cancellative. 
\end{rem}

Our next goal is to show that Cuntz semigroups of C$^*$-algebras of stable rank one admit infima which are compatible with addition,
in the sense of the definition below. 
This will have important consequences later on.

\begin{df}
A partially ordered set $S$ is an \emph{inf-semilattice} provided the order-theoretic infimum $x\wedge y$ exists for any elements $x,y\in S$. We say that an ordered semigroup is \emph{inf-semilattice ordered} if it is an inf-semilattice and, for any $x,y,z\in S$, we have
\[
(x+z)\wedge (y+z)=(x\wedge y)+z.
\]
\end{df}

\begin{lma}
\label{byO6plus}
Let $S$ be a $\Cu$ semigroup which is an inf-semilattice and satisfies (O6+). If $x, y,z\in \Cu(A)$ satisfy $x\leq y+z$, then
$x\leq (x\wedge y)+(x\wedge z)$.
\end{lma}
\begin{proof}
Applying axiom (O6+) to $x\leq y+z$ with $u=u'=v=v'=0$, we obtain $x\leq s+t$ for some $s\leq x,y$ and $t\leq x,z$. The existence of infima proves the lemma.
\end{proof}

The lemma below shows that axiom (O6+) is automatic for inf-semilattice \emph{ordered} semigroups.

\begin{lma}
An inf-semilattice ordered $\Cu$-semigroup $S$ satisfies (O6+). 
\end{lma}
\begin{proof}
Assume that $x,y,z\in S$ satisfy $x\leq y+z$ and $u',u,v',v\in S$ are such that $u'\ll u\leq x,y$, and $v'\ll v\leq x,z$. Let $s=x\wedge y$ and $t=x\wedge z$. Then clearly $u'\ll u\leq s$ and $v'\ll v\leq t$.
Using the inf-semillatice condition repeatedly, we get
\[
(x+x)\wedge (x+z)\wedge (x+y)\wedge (y+z)=(x+x\wedge y)\wedge(z+x\wedge y)=(x\wedge y)+(x\wedge z).
\]
Using this at the second step, we conclude that
\[
x\leq (2x)\wedge (x+z)\wedge (x+y)\wedge (y+z)=(x\wedge y)+(x\wedge z)=s+t. \qedhere
\]
\end{proof}

As we outline in \autoref{Cstarinf-semilattice}, the Cuntz semigroup of any C$^*$-algebra of stable rank one is inf-semilattice ordered. Given the argument above, one would get a different proof that (O6+) is verified for the Cuntz semigorup of C$^*$-algebras of stable rank one. However, (O6+) is used to show the inf-semilattice ordered property, so that raises the question of whether one can show that C$^*$-algebras of stable rank one have inf-semilattice ordered Cuntz semigroups without using (O6+).

In preparation for \autoref{prop:key} below, we need to briefly discuss a different picture of the Cuntz semigroup, as developed in \cite{CowEllIva_Cuntz_2008}.
\begin{df}
A \emph{(right) Hilbert module} over a C$^*$-algebra $A$ is a (right) $A$-module $X$ together with a map $\langle \cdot,\cdot \rangle\colon X\times X\to A$ which is $\C$-linear on the second entry and satisfies
\begin{enumerate}[{\rm (i)}]
\item $\langle x, ya\rangle=\langle x, y\rangle a$
\item $\langle x,y\rangle=\langle y,x\rangle^*$
\item $\langle x,x\rangle\geq 0$ and equals zero precisely when $x=0$,
\end{enumerate} 
for all $x,y\in X$, $a\in A$, and such that $X$ is complete with respect to the norm $\Vert x\Vert=\langle x,x\rangle^{\frac{1}{2}}$.
\end{df}
The standard example of an $A$-Hilbert module is the C$^*$-algebra
$A$ itself, with structure given by $\langle x,y\rangle=x^*y$. 
More generally, if $a\in A_+$, then $\overline{aA}$ is naturally a
Hilbert $A$-module.
We denote by $\ell^2(A)$ the so-called \emph{standard Hilbert $A$-module}, which is given as
\[
\ell^2(A)=\Big\{(x_n)_{n\in\N}\in\prod_{n\in\N} A\colon \sum_{n\in\N} x_n^*x_n \text{ converges in the norm of } A\Big\},
\]
with inner product defined as $\langle x,y\rangle=\sum_{n\in\N}x_n^*y_n$ for $x,y\in \bigoplus_{n\in\N} A$.

\begin{df}
A Hilbert $A$-module $X$ is said to be \emph{countably generated} if there is a countable set $\{x_n\colon n\in\N\}\subseteq X$ such that  $X=\overline{\sum_{n=1}^\infty x_nA}$.
\end{df}

A fundamental result in the theory of Hilbert modules is Kasparov's absorption theorem; 
see \cite[Theorem 1.4.2]{ManTroi_Hilbert_2005}.

\begin{thm}[Kasparov's theorem]
Let $X$ be a countably generated Hilbert $A$-module. Then $X\oplus\ell^2(A)\cong \ell^2(A)$.
\end{thm}

In \cite{CowEllIva_Cuntz_2008}, a notion of subequivalence between countably generated Hilbert $A$-modules, weaker than isomorphism, was introduced. By antisymmetrizing this subequivalence, a partially ordered semigroup was built, and it  was shown that this semigroup is isomorphic to $\Cu(A)$. Under this isomorphism, a positive element $a\in (A\otimes\K)_+$ corresponds to $\overline{a(A\otimes\K)}$.

Further, in the stable rank one case, it was proved that equivalence of $A$-modules is the same as isomorphism, and a Hilbert $A$-module $X$ is  subequivalent to a Hilbert $A$-module $Y$ precisely when $X\cong X'\subseteq Y$; see also \cite[Section 4]{AraPerTom_survey_2011}  for further details.

One of the key results in this and coming sections is the following result, proved in \cite[Proposition 2.8]{AntPerRobThi_stable_2022}:

\begin{prop}
\label{prop:key} Let $A$ be a stable C$^*$-algebra of stable rank one, and let $a\in A_+$. Then there are a C$^*$-algebra $B$, also of stable rank one, that contains $A$ as closed, two-sided ideal, and a projection $p_a\in B$ such that the following property holds:
\[
\text{For }x\in \Cu(A), \text{ we have }x\leq [a]\text{ in }\Cu(A) \text{ if and only if }x\leq [p_a]\text{ in }\Cu(B).
\]
\end{prop}
\begin{proof}
We sketch the construction. Consider the Hilbert module $H=\overline{aA}$, which by Kasparov's theorem is isomorphic to a direct summand of $\ell^2(A)$, that is, $H\oplus H'\cong\ell^2(A)$ for some Hilbert module $H'$. Also, since $A$ is stable,
we have $\ell^2(A)\cong A$ as Hilbert modules.

Let $M(A)$ be the multiplier algebra of $A$, identified with the algebra of adjointable operators on $A$, and let $p_a\in M(A)$ be the projection that corresponds, under the previous identification, to the projection onto $\overline{aA}$, so that $\overline{aA}\cong p_aA$. Set $B=C^*(p_a,A)\subseteq M(A)$. By construction then, $A$ is a closed, two-sided ideal of $B$ and, since $B/A\cong\C$, we see that $B$ also has stable rank one.

To prove the stated property, let $b\in A_+$ and set $x=[b]\in \Cu(A)$. If $b\precsim p_a$ in $B$ and $\ep>0$ is given, then by
\autoref{lma:CutDownDistance} there is an element $d\in B$ such that $(b-\varepsilon)_+=d^*p_ad$. Setting $v=p_ad\in p_aB$, we have $(b-\varepsilon)_+=v^*v$.

Observe also that $v^*v\in A$, hence $v\in A$ by \autoref{rem:ideala*a}, since $A$ is an ideal of $B$. Thus $v\in p_aB\cap A=p_aA\cong \overline{aA}$. Altogether, we have that
\[
\overline{(b-\varepsilon)_+A}=\overline{v^*vA}\cong \overline{vv^*A}\subseteq p_aA\cong\overline{aA}
\]
as Hilbert $A$-modules. 
This implies, by using the Hilbert module picture of $\Cu(A)$, that $(b-\varepsilon)_+\precsim a$ and, since $\varepsilon>0$ is arbitrary, we obtain $b\precsim a$.

For the converse, it suffices to show that $[a]\leq [p_a]$ in $\Cu(B)$. This follows from $\overline{aB}=\overline{aA}\cong p_aA\subseteq p_aB$, as Hilbert $B$-modules.
\end{proof}

%

\begin{df}
We say that an ordered semigroup $S$ satisfies the \emph{Riesz interpolation property} if given $x,y,z,t\in S$ with $x,y\leq z,t$, there exists $w\in S$ such that $x,y\leq w\leq z,t$.

A partially ordered group is called an \emph{interpolation group} if it satisfies the Riesz interpolation property.
\end{df}

If $S$ is algebraically ordered and cancellative, it is well known that the Riesz interpolation property is equivalent to the Riesz decomposition property, and also to the Riesz refinement property; see \cite[Proposition 2.1]{Goo_poag_1986}

It was shown in \cite{Per_structure_1997} that, for C$^*$-algebras of real rank zero and stable rank one, $\Cu(A)$ enjoys the three Riesz properties. This is largely due to the fact that the Murray-von Neumann semigroup is cancellative and has Riesz decomposition; see \cite{BlaHan_dimension_1982} and \cite{Zha_riesz_1991}. However, this is no longer true if one drops the real rank zero assumption,
even if one assumes stable rank one. An example of this is given by the Jiang-Su algebra $\mathcal{Z}$, as was observed in \cite[Remark 6.9]{Thi_ranks_2020}. Indeed, we have
seen in \autoref{thm:CuZ} that 
\[\Cu(\mathcal{Z})\cong \underbrace{\{c_n\colon n\in\N\}}_{\mathrm{compacts}}\sqcup \{s_t\colon t\in (0,\I]\}=\N\sqcup (0,\I].\]
Let $x=c_1\in\N$, and let $y=z=s_{\frac{2}{3}}\in (0,\infty]$. We clearly have that $x=c_1\leq s_{\frac{4}{3}}=y+z$. Since $c_1$ is compact and there are no other compact elements below $c_1$ except $c_0$ and $c_1$, we see that if $x=u+v$, we must have $u=c_0$ and $v=c_1$, or reversed. Clearly $u,v\not\leq s_{\frac{2}{3}}$. 

However, as we shall prove below, Riesz interpolation persists if one drops the assumption of real rank zero and keeps the stable rank one condition.

In the lemma below, we use axiom $\widetilde{\mathrm{(O5)}}$, since the $\Cu$-semigroup is assumed to be weakly cancellative; see \autoref{pgr:O5}. We also use the one-sided version of (O6+); see \autoref{pgr:O6+}. 

\begin{lma}
\label{lma:ht}
Let $S$ be a weakly cancellative $\Cu$-semigroup satisfying (O5) and (O6+), and let $e,x\in S$. If $e$ is compact, the set
\[
\big\{ z\in S : z\leq e,x \big\}
\]
is upward directed.	
\end{lma}	
\begin{proof}
We first notice that the set $\{ z\in S : z\leq e,x \}$ is order-hereditary and closed under suprema of increasing sequences. To 
prove the lemma, it suffices to show that the set
\[
\big\{ z'\in S : \text{there is }z\in S\text{ such that }z'\ll z\leq e,x \big\}
\]
is upward directed. (We leave the proof of this claim as an exercise.)
Take elements $z_1,z_2,z_1', z_2' \in S$ that satisfy
\[
z_1'\ll z_1\leq e,x, \ \ \text{ and } \ \ z_2'\ll z_2\leq e,x.
\]
Apply (O5) to $z_1'\ll z_1\leq e$ to find $w\in S$ such that $z_1'+w\leq e\leq z_1+w$. Since $z_1\leq x$, we obtain $e\leq x+w$. We now apply (O6+) to this inequality together with $z_2'\ll z_2\leq e,x$. Thus, we find $y\in S$ such that $e\leq y+w$ and $z_2'\ll y\leq e,x$. Putting together the left hand side of the inequality coming from (O5),  the one coming from (O6+), and using that $e$ is compact, we get
\[
z_1'+w\leq e\ll e\leq y+w.
\]
Weak cancellation in $S$ implies $z_1'\ll y$. Hence, $z_1',z_2'\ll y\leq e,x$. Now choose $y'\in S$ with $z_1',z_2'\ll y'\ll y$, and check that $y '$ has the desired properties.
\end{proof}

\begin{thm}
\label{thm:CuARiesz}
Let $A$ be a C$^*$-algebra of stable rank one. Then $\Cu(A)$ has the Riesz interpolation property.
\end{thm}	
\begin{proof}
Let $x,y\in\Cu(A)$. We must show that the set $\{ z\in \Cu(A) : z\leq x,y \}$ is upward directed. If $x$ is compact, this follows from \autoref{lma:ht}. We now use \autoref{prop:key} to reduce the general case to this case.

By part~(iii) of \autoref{prop:sr1}, we may assume that $A$ is stable. Choose $a\in A_+$ such that $x=[a]$.
Applying \autoref{prop:key} to $A$ and $a$, we obtain a C$^*$-algebra $B$ with stable rank one that contains $A$ as a closed, two-sided ideal, and a projection $p_a\in B$ such that $z\in\Cu(A)$ satisfies $z\leq x$ if and only if $z\leq[p_a]$. Since $[p_a]$ is compact in $\Cu(B)$, and since $B$ has stable rank one, it follows from \autoref{lma:ht} that the set $\{ z\in\Cu(B) :  z\leq [p_a],y \}$ is upward directed. The inclusion $A\subseteq B$ identifies $\Cu(A)$ with an 
ideal in $\Cu(B)$ by \autoref{lma:IdealAIdealCu}. 
The result follows once we show that
\[
\big\{ z\in \Cu(A) : z\leq x,y \big\}
= \big\{ z\in\Cu(B) :  z\leq [p_a],y \big\}.
\]

One inclusion follows using that $x\leq [p_a]$. For the converse inclusion, take $z\in \Cu(B)$ such that $z\leq [p_a],y$.
Since $\Cu(A)$ is an ideal of $\Cu(B)$ and $y\in\Cu(A)$, we have $z\in\Cu(A)$.
Since also $z\leq [p_a]$, \autoref{prop:key} implies that $z\leq x$.
\end{proof}	

We close this section showing that \autoref{thm:CuARiesz} can be used to prove that the Cuntz semigroup of a separable C$^*$-algebra with stable rank one is inf-semilattice ordered. The existence of infima follows easily from the Riesz interpolation property.

\begin{df}
A $\Cu$-semigroup $S$ is said to be \emph{countably based} provided there is a countable set $\mathcal{B}\subseteq S$ such that for
all $s,t\in S$ satisfying $s\ll t$, there is $u\in \mathcal{B}$ such that $s\ll u\ll t$.
\end{df}
It is known that the Cuntz semigroup $\Cu(A)$ of a separable C$^*$-algebra $A$ is always countably based. For example, one can take a countable dense subset $F$ of $A$ and then consider the set $\mathcal{B}=\big\{[(a-\tfrac{1}{n})_+]\colon a\in F, n\in \N\big\}$.

\begin{thm}
\label{Cstarinf-semilattice}
Let $A$ be a separable C$^*$-algebra of stable rank one. Then $\Cu(A)$ is an inf-semilattice ordered semigroup.
\end{thm}	
\begin{proof} (Outline)
Without loss of generality, we may assume that $A$ is stable. 

Since $A$ is separable, as observed above $\Cu(A)$ is countably based.  It follows from this that every upward directed set has a supremum. Given $x,y\in \Cu(A)$, this applies in particular to the set $\{ z\in \Cu(A) : z\leq x,y\}$, which is upward directed since 
$\Cu(A)$ has the Riesz interpolation property by \autoref{thm:CuARiesz}.
Notice that the supremum of the set $\{ z\in \Cu(A) : z\leq x,y\}$ is precisely $x\wedge y$.
Thus, $\Cu(A)$ is an inf-semilattice.

In order to prove the distributivity of $\wedge$ over addition, note that we always have $x\wedge y+z\leq x+z, y+z$. Thus we need to show that
\[
(x+z)\wedge (y+z)\leq ( x\wedge y)+z,
\]
for all $x,y,z\in\Cu(A)$.
To indicate the flavour of the argument, we give a proof in the case that both $x$ and $z$ are compact elements. The general case is obtained through successive generalizations.

Let $w=(x+z)\wedge (y+z)$. Choose $w'\in \Cu(A)$ such that $w'\ll w$. Applying (O5) (or rather, $\widetilde{\mathrm{(O5)}}$) to the inequality $w'\ll w\leq x+z$, we find $v\in\Cu(A)$ such that $w'+v\leq x+z\leq w+v$. We get $x+z\leq y+z+v$. As $A$ has stable rank one, $\Cu(A)$ has cancellation of compact elements (see \autoref{rmk:wc}), and since $z$ is compact by assumption, we obtain $x\leq y+v$.
By \autoref{byO6plus}, $x\leq (x\wedge y)+v$. Adding $z$ on both sides we get $x+z\leq (x\wedge y)+v+z$. Since both $x$ and $z$ are compact, so is $x+z$, and thus
\[
w'+v\leq x+z\ll x+z\leq (x\wedge y)+z+v\,.
\]
Now weak cancellation implies $w'\leq (x\wedge y)+z$ and, since $w'$ is arbitrary satisfying $w'\ll w$, the inequality $(x+z)\wedge (y+z)\leq ( x\wedge y)+z$ holds.
\end{proof}

\section{The classical Cuntz semigroup, its relation to \texorpdfstring{$\Cu(A)$}{Cu(A)}, and the Blackadar-Handelman conjecture}
\label{sec:BH}

In the previous section we studied the basic structural properties of the Cuntz semigroup of any separable  C$^*$-algebra with stable rank one. We will obtain now an easy consequence of \autoref{thm:CuARiesz}, which allows us to solve, in the stable rank one case, a conjecture due to Blackadar and Handelman on the structure of dimension functions.

\begin{df}[\cite{Cun_dimension_1978}]
Let $A$ be a unital C$^*$-algebra. Set $M_\infty(A)=\bigcup_{n\in\N}  M_n(A)$, identifying each $M_n(A)$ inside $M_{n+1}(A)$ as the upper left corner. Then $M_\infty(A)$ is a local C$^*$-algebra, and the 
relation of Cuntz (sub)equivalence from \autoref{df:CuSubeq} can be 
restricted to it. Define
\[
\W(A)=M_\infty(A)_+/\!\!\sim,
\]
the so-called \emph{classical Cuntz semigroup} of $A$. 

A \emph{dimension function} on $A$ is a map $d\colon M_\infty(A)_+\to [0,\infty)$ satisfying the following properties:
\begin{itemize}
 \item $d(a\oplus b)=d(a)+d(b)$ for all $a,b\in M_\infty(A)_+$;
 \item $d(a)\leq d(b)$ whenever $a\precsim b$;
 \item $d(1_A)=1$.
\end{itemize}
Denote the set of dimension functions of $A$ by $\mathrm{DF}(A)$. If we also denote by $\KK_0^*(A)$ the Grothendieck group of $\W(A)$, it is not hard to verify that there is a bijection between $\mathrm{DF}(A)$ and $\mathrm{St}(\KK_0^*(A),[1_A])$, the state space of the group $\KK_0^*(A)$. Indeed, given a dimension function $d$ on $A$, one sets $s_d([a]-[b])=d(a)-d(b)$, which defines a state on $\KK_0^*(A)$.
\end{df}

The following first appeared in \cite{BlaHan_dimension_1982}:

\begin{cnj}\label{cnj:BH} (Blackadar-Handelman).
Let $A$ be a unital \ca. Then $\mathrm{DF}(A)$ is a Choquet simplex. 
\end{cnj}

The above conjecture was verified for C$^*$-algebras with real rank zero and stable rank one (\cite{Per_structure_1997}), for certain C$^*$-algebras of stable rank $2$ (\cite{AntBosPerPet_geometric_2014}),  and for C$^*$-algebras with finite radius of comparison and finitely many extremal quasitraces (\cite{DeS_conject_2016}). It was also asked in \cite[Problem 3.13]{AntBosPerPet_geometric_2014} for which unital C$^*$-algebras does it hold that $\KK_0^*(A)$ is an interpolation group.

Note that we also have 
\[\W(A)=\{x\in \Cu(A)\colon x=[a] \text{ for some } a\in M_\infty(A)_+\}.\] In the case that $A$ has stable rank one, we show below that $\W(A)$ is a hereditary subset of $\Cu(A)$; see \cite[Lemma 3.4]{AntBosPer_completion_2011}. This means that, if $x\leq y$ in $\Cu(A)$ and $y\in \W(A)$, then $x\in \W(A)$. 

\begin{lma}
\label{lma:WAhereditary}
Let $A$ be a C$^*$-algebra of stable rank one. Then  $\W(A)$ is hereditary in $\Cu(A)$. In particular,
\[\W(A)=\{x\in\Cu(A)\colon x\leq n[a]\text{ for some }a\in A_+,n\in\N\}.\]
\end{lma}
\begin{proof}
Let $a\in (A\otimes\mathcal{K})_+$, $b\in M_{\infty}(A)_+$, and assume that $a\precsim b$.  We need to show that there is $c\in M_{\infty}(A)_+$ such that $c\sim a$. 

Since $A\otimes\mathcal{K}$ is the completion of $M_{\infty}(A)$ and $a\in (A\otimes\mathcal{K})_{+}$, there exists a sequence $(a_{n})_{n\in\N}$ in $M_{\infty}(A)_{+}$ with $a=\lim\limits_{n\to\I} a_{n}$ and $\|a-a_{n}\|\leq \tfrac{1}{n}$. 
By \autoref{lma:CutDownDistance}, for each $n\in\N$
there exists $d_n\in A\otimes\mathcal{K}$ such that $(a-\tfrac{1}{n})_{+}=d_{n}a_{n}d_{n}^{*}$ and then  
\[
(a-\tfrac{1}{n})_{+}=d_{n}a_{n}d_{n}^{*}\sim a^{\frac{1}{2}}_{n}d^{*}_{n}d_{n}a^{\frac{1}{2}}_{n}\,.
\]
Put $b_n:=a_{n}^{\frac{1}{2}}d^{*}_{n}d_{n}a_{n}^{\frac{1}{2}}\in M_{\infty}(A)_{+}$. We have $[ a]=\sup\limits_{n\in\N} [ b_n]$ in $\Cu(A)$, and $([ b_n])_{n\in\N}$ is $\ll$-increasing in $\Cu(A)$.

Now, the sequence $([ b_n])_{n\in\N}$ is bounded above in $\W(A)$ by $[ b]$. Therefore, it also has a supremum $[ c]$ in $\W(A)$, by \cite[Lemma 4.3]{BroPerTom_cuntz_2008}. In fact, the arguments in \cite{BroPerTom_cuntz_2008} show that for each $n$ there exist $m$ and $\delta_n>0$ with $(c-\tfrac{1}{n})_{+}\precsim (b_{m}-\delta_{n})_{+}$, and such that $(\delta_n)_{n\in\N}$ strictly decreases to zero. Therefore
\[
(c-\tfrac{1}{n})_{+}\precsim (b_{m}-\delta_{n})_{+}\precsim b_{m}\precsim a
\]
in $A\otimes\mathcal{K}$, and thus $c\precsim a$.

On the other hand, since also $b_{n}\precsim c$ for all $n$, and $[ a]=\sup\limits_{n\in\N}[b_n]$ in $\Cu(A)$, we see that $a\precsim c$. Thus $c\sim a$, as desired.
\end{proof}

\begin{lma}
\label{lma:interpolation}
Let $S$ be a positively ordered semigroup that has the Riesz interpolation property. Then its Grothendieck group $G(S)$ is an interpolation group.
\end{lma}
\begin{proof}
Let $a_1,a_2,b_1,b_2$ be elements in $G(S)$ such that $a_i\leq b_j$ for all $i,j=1,2$. There exist elements $z,x_i,y_j \in S$ such that
$a_i=[x_i]-[z]$ and $b_j=[y_j]-[z]$. (If $a_i=[x_i]-[v_i]$ and $b_i=[y_i]-[w_i]$, then one may take $z=v_1+v_2+w_1+w_2$.) Therefore, by adding $[z]$ to the inequality we get 
$[x_i]\leq [y_j]$ for all $i,j=1,2$.
Thus there exists $t\in S$ such that
\[x_i+t \leq y_j+t\]
for all $i,j=1,2$. By assumption, there exists $x\in S$ interpolating the above
inequality. Consider the element $e=[x]-[z+t]\in G(S)$.
Then it is easy to check that $e$ satisfies $a_i\leq e\leq b_j$ for all $i,j=1,2$.
\end{proof}

\begin{thm}
\label{thm:BH}
Let $A$ be a C$^*$-algebra of stable rank one. Then $\KK_0^*(A)$ is an interpolation group and thus $\mathrm{DF}(A)$ is a Choquet simplex.
In particular, \autoref{cnj:BH} holds if $A$ has stable rank one.
\end{thm}
\begin{proof}
We know from \autoref{lma:WAhereditary} that $\W(A)$ is hereditary. We use this to show that $\W(A)$ has the Riesz interpolation property. By \autoref{thm:CuARiesz}, this is the case for $\Cu(A)$. Let $x,y,z,t\in \W(A)$ be such that $x,y\leq z,t$. Then this also holds in $\Cu(A)$ and thus there is $w\in\Cu(A)$ such that $x,y\leq w\leq z,t$. Since $\W(A)$ is hereditary, we have $w\in \W(A)$.

By \autoref{lma:interpolation}, $\KK_0^*(A)$ is an interpolation group, and using \cite[Theorem 10.17]{Goo_poag_1986}, we obtain that its state space, that is, $\mathrm{DF}(A)$, is a Choquet simplex.
\end{proof}

One of the reasons for introducing $\Cu(A)$ was the need of a continuous invariant. Regarding $\W$ as a functor from C$^*$-algebras to the category of positively ordered semigroups, it is clear that $\W$ is not continuous. (This already fails for $\K=\varinjlim M_n$.)

As it turns out, $\Cu(A)$ can be regarded as the completion of $W(A)$, just as $A\otimes\K$ is the completion of $M_\infty(A)$. In order to outline the exact relationship between these two semigroups, we need some additional concepts. The following is inspired by
\cite[Definition~I-1.11, p.57]{GieHofKeiLawMisSco_continuous_2003}.

\begin{df}
\label{pgr:aux}
Let $(X,\leq)$ be a partially ordered set.
A binary relation $\prec$ on $X$ is called an \emph{auxiliary relation} if the following properties are satisfied:
\begin{itemize}
\item[{\rm (i)}]
If $x\prec y$ then $x\leq y$, for all $x,y\in X$.
\item[{\rm (ii)}]
If $w\leq x\prec y\leq z$ then $w\prec z$, for all $w,x,y,z\in X$.
\end{itemize}
If, further, $X$ is a monoid, then an auxiliary relation $\prec$ is said to be \emph{additive} if it is compatible with addition and $0\prec x$ for every $x\in X$.
\end{df}

Observe that an auxiliary relation as defined above is transitive. To see this, if $x\prec y$ and $y\prec z$, then $x\leq y$ by condition (i) and applying condition (ii) to $x\leq y\prec z\leq z$, we obtain $x\prec z$.
In the case of a $\Cu$-semigroup $S$, the compact containment relation $\ll$ is an example of an auxiliary relation on $S$. If $A$ is a C$^*$-algebra, then $\W(A)$ may be equipped with the following auxiliary relation: $[a]\prec [b]$ if and only if $[a]\leq [(b-\varepsilon)_+]$ for some $\varepsilon>0$; see \cite[Proposition 2.2.5]{AntPerThi_tensor_2018}.
Equivalently by \autoref{rem:CutDownWayBelow}, $[a]\prec [b]$
in $\W(A)$ if and only if $[a]\ll [b]$ in $\Cu(A)$. 

\begin{df}
A $\mathrm{W}$-semigroup is a positively ordered semigroup $S$ together with an auxiliary relation $\prec$ such that the following axioms hold:
 \begin{itemize}
 \item[(W1)] For each $a\in S$, the set $a^{\prec}=\{b\in S\colon b\prec a\}$ has a $\prec$-increasing countable cofinal subset (with respect to $\prec$).
 \item[(W3)] $\prec$ is additive.
 \item[(W4)] If $a \prec b+c$ in $S$ then there are $b'\prec b$ and $c'\prec c$ such that $a\prec b'+c'$.
 \end{itemize}
 
A positively ordered semigroup morphism $f\colon S\to T$ between two $\W$-semigroups is a $\W$-morphism provided it preserves $\prec$ and that is also continuous, in the sense that if $b\prec f(a)$ in $T$, then there is $a'\prec a$ in $S$ such that $b\leq f(a')$. We will denote by $mathbf{W}$ the category whose objects are the $\W$-semigroups and whose morphisms are the $\W$-morphisms. The set of $\W$ morphisms between $\W$-semigroups $S$ and $T$ will be denoted by $\mathbf{W}(S,T)$. 
\end{df}

We remark that the terminology has evolved so that initially a $\mathrm{W}$-semigroup was also required to satisfy (W2): for each $a\in S$, we have $a=\sup a^\prec$, but this is not relevant for the theory. It is not even relevant, for many purposes, to require that a $\W$-semigroup is positively ordered (and only that it is equipped with a transitive relation $\prec$ satisfying the axioms above). 

Again, if $A$ is a C$^*$-algebra, it was shown in \cite[Proposition 2.2.5]{AntPerThi_tensor_2018} that $\W(A)$ is a $\W$-semigroup with the auxiliary relation defined above. 
Moreover, if $\varphi\colon A\to B$ is a homomorphism of \ca s, 
then the restriction $W(\varphi)$ of $\Cu(A)$ to $W(A)$ is a 
W-morphism $W(\varphi)\colon W(A)\to W(B)$.

It is easy to verify that $\CatCu$ is a full subcategory of \textbf{W} and that \textbf{W} has limits; 
see \cite[Theorem 2.2.9]{AntPerThi_tensor_2018}. It turns out it is also a reflexive category, which follows from the theorem below; see \cite{AntPerThi_tensor_2018}.

\begin{thm}
\label{thm:curefl}
Given a $\W$-semigroup $(S,\prec)$, there are a $\Cu$-semigroup $\gamma(S)$ and a $\W$-morphism $\alpha\colon S\to \gamma(S)$ such that:
\begin{enumerate}[{\rm (i)}]
\item $a'\prec a$ in $S$ whenever $\alpha(a')\ll \alpha(a)$.
\item If $b'\ll b$ in $\gamma(S)$, then there is $a\in S$ such that $b'\ll\alpha(a)\ll b$.
\end{enumerate}
\end{thm}
\begin{proof}(Outline)
We just show how to construct $\gamma(S)$. One considers the set $S_\prec$ of $\prec$-increasing sequences in $S$. Any two such sequences are added pointwise, and one declares $(a_n)_{n\in\N}\precsim (b_n)_{n\in\N}$ if for every $k\in\N$ there is $n\in\N$ such that $a_k\prec b_n$. This defines a translation invariant preorder that yields an equivalence relation by setting $(a_n)_{n\in\N}\sim (b_n)_{n\in\N}$ if and only if $(a_n)_{n\in\N}\precsim (b_n)_{n\in\N}$ and $(b_n)\precsim (a_n)_{n\in\N}$. We then define $\gamma(S)$ to be $S_\prec/\!\!\sim$. Addition is induced by addition of sequences and the order is induced by $\precsim$. It is possible to prove that $[(a_n)_{n\in\N}]\ll [(b_n)_{n\in\N}]$ precisely if there is $k\in\N$ such that $a_n\prec b_k$ for all $n\in\N$.

In order to define $\alpha\colon S\to\gamma(S)$, let $a\in S$ and apply (W1) to find $(a_n)_{n\in\N}\in S^\prec$ which is cofinal in $a^\prec$. Then set $\alpha(a)=[(a_n)_{n\in\N}]$. We omit the details.
\end{proof}

The construction just outlined defines a functor $\gamma\colon \mathrm{\textbf{W}}\to\CatCu$ which is a reflector for the inclusion. Applied to C$^*$-algebras, this yields:

\begin{thm}[{\cite[Theorem 3.2.8]{AntPerThi_tensor_2018}}]
\label{thm:completion}
The compositions $\gamma\circ\W$ and $\Cu$ are naturally isomorphic as functors from \textbf{C$^*$} to $\CatCu$. In particular, if $A$ is a C$^*$-algebra, then $\Cu(A)$ is naturally isomorphic to $\gamma(\W(A))$.
\end{thm}

The result above is extremely useful when constructing
objects in the category $\CatCu$ as certain ``completions''
of objects in \textbf{W}. We already saw an
example of this in \autoref{thm:Culimits}, when we 
constructed inductive limits in $\CatCu$. Indeed, what
we did there was to consider the inductive limit of 
Cu-semigroups in the category \textbf{W}, and then apply
the functor $\gamma$. The same strategy can be used to 
construct other objects, such as (infinite) direct sums.
On the other hand, many other constructions (such as
products) will require a different
treatment, since these constructions do not obviously
exist in \textbf{W} either. This is done in \autoref{sec:Structure}, where we consider an even larger
category \textbf{Q} and a natural functor 
$\tau\colon \mathrm{\textbf{Q}}\to \CatCu$; see 
\autoref{prp:CuCoreflQ}.

\section{Functionals and the realization of ranks}
\label{sec:functionals}

In this section, we formulate the problem of realization of ranks and sketch the solution for C$^*$-algebras of stable rank one. The inf-semilattice ordered structure of the Cuntz semigroup for such algebras is a key element for the solution.
Another ingredient that is needed in this setting is an
additional axiom for Cu-semigroups, called \emph{Edward's 
condition}; see \cite{AntPerThi_edwards_2021}. Since we will
omit the proofs where this axiom is needed, we will also not 
discuss this condition here. 

Recall from \autoref{df:functionals} that a \emph{functional}
on a Cu-semigroup $S$ is an additive function $\lambda\colon S
\to [0,\I]$ satisfying $\lambda(0)=0$, that preserves order and 
suprema of increasing sequences.
We equip the set $F(S)$ of functionals on $S$ with operations of addition and scalar multiplication by nonzero, positive real numbers, defined pointwise. Moreover, $F(S)$ is a topological cone with 
respect to the topology whose subbase is given by the 
collection of all sets
\[V_{x,r}=\{\lambda\in F(S)\colon f(x)>r\} \ \ \mbox{ and } 
 \ \ W_{x,r}=\{\lambda\in F(S)\colon f(x')<r \mbox{ for some } x'\ll x\},
\]
for $x\in S$ and $r\in (0,\infty)$; see \cite{Kei_cuntz_2017} and also 
\cite{EllRobSan_cone_2011}. With respect to this topology, 
given $\lambda\in F(S)$ and a net $(\lambda_i)_{i\in I}$ in $F(S)$, we have $\lambda_i\to \lambda$ if and only if
\[
\limsup \lambda_i(x')\leq \lambda(x)\leq \liminf \lambda_i(x)\hbox{ for all }x',x\in S \text{ such that }x'\ll x.
\]
It was shown in \cite[Theorem~3.17]{Kei_cuntz_2017} (see 
also \cite[Theorem~4.8]{EllRobSan_cone_2011}) that, with this topology, $F(S)$ is a compact Hausdorff ordered
topological cone.

By \autoref{thm:QTfunctionals} and the comments after it,
for any C$^*$-algebra $A$ there is a natural bijection between $F(\Cu(A))$ and the set 
of $[0,\infty]$-valued, lower semicontinuous $2$-quasitraces on $A$.


A significant difference when considering 
\emph{non-normalized} functionals on Cu-semigroups, is that 
these naturally arise from the ideal structure of the semigroup, 
as follows:

\begin{lma}
\label{prp:extFctl}
Let $S$ be a $\Cu$-semigroup, let $I\subseteq S$ be an ideal, and let $\lambda\colon I\to[0,\infty]$ be a functional.
Define $\tilde{\lambda}\colon S\to[0,\infty]$ by
\[
\tilde{\lambda}(x)=
\begin{cases}
\lambda(x), &\text{ if } x\in I; \\
\infty, &\text{ otherwise}.
\end{cases}
\]
Then $\tilde{\lambda}$ is a functional on $S$.
\end{lma}
\begin{proof}
Let us show that $\tilde{\lambda}$ is order-preserving. If $x\leq y$ in $S$ and $y\notin I$, then $\tilde{\lambda}(y)=\infty$, and clearly $\tilde{\lambda}(x)\leq\tilde{\lambda}(y)$. If $y\in I$, then $x\in I$ as well, since $I$ is an ideal of $S$, and thus $\tilde{\lambda}(x)=\lambda(x)\leq\lambda(y)=\tilde{\lambda}(y)$.

Next, let $x,y\in S$. Clearly $x+y\in I$ if and only if both $x,y\in I$. If $x,y\in I$, then 
\[\tilde{\lambda}(x+y)
=\lambda(x+y)
=\lambda(x)+\lambda(y)
=\tilde{\lambda}(x)+\tilde{\lambda}(y).\]
If either $x\notin I$ or $y\notin I$, then $x+y\notin I$, hence $\tilde{\lambda}(x+y)=\infty=\tilde{\lambda}(x)+\tilde{\lambda}(y)$. That $\tilde{\lambda}$ preserves suprema of increasing sequences follows in a similar manner.
\end{proof}

For a $\Cu$-semigroup $S$ that satisfies (O5), we give below the appropriate notion of dual for the cone $F(S)$.
Denote by $\mathrm{Lsc}(F(S))$ the set of functions $f\colon F(S)\to[0,\infty]$ that are additive, order-preserving, homogeneous (with respect to nonzero, positive scalars), lower semicontinuous, and satisfy $f(0)=0$. 
This set is equipped with pointwise order, addition, and scalar multiplication by nonzero positive scalars. 

\begin{df}\label{df:rank}
Let $S$ be a Cu-semigroup and let $x\in S$.
Given $x\in S$, the \emph{rank} of $x$ is the function $\widehat{x}\colon F(S)\to [0,\infty]$ given by evaluation, namely:
\[
\widehat{x}(\lambda)=\lambda(x)
\]
for all $\lambda\in F(S)$. One can check that
$\widehat{x}$ belongs to $\mathrm{Lsc}(F(S))$. 

The \emph{rank map} $\mathrm{rk}\colon S\to\mathrm{Lsc}(F(S))$ of
$S$ is defined by $\mathrm{rk}(x)=\widehat{x}$ for all $x\in S$. \end{df}

It is easy to check that the rank map preserves addition, order, and suprema of increasing sequences.

The \emph{realification} of $S$, denoted by $S_R$, was introduced in \cite{Rob_cone_2013} as the smallest subsemigroup of $\mathrm{Lsc}(F(S))$ that is closed under suprema of increasing sequences and contains all elements of the form $\tfrac{1}{n}\widehat{x}$ for $x\in S$ and $n\geq 1$. It can be shown that $S_R\cong S\otimes_\Cu [0,\infty]$, thus justifying the 
term ``realification''; see the proof of \autoref{monoidal} 
for the definition of tensor products in $\CatCu$.
Moreover, 
it was proved in \cite[Proposition~3.1.1]{Rob_cone_2013} that $S_R$ is a $\Cu$-semigroup satisfying (O5);
see also \cite[Proposition~7.5.6]{AntPerThi_tensor_2018}.

Given $f,g\in\mathrm{Lsc}(F(S))$, we write $f\lhd g$ if $f\leq(1-\varepsilon)g$ for some $\varepsilon>0$ and if $f$ is continuous at each $\lambda\in F(S)$ satisfying $g(\lambda)<\infty$.
We denote by $L(F(S))$ the subsemigroup of $\mathrm{Lsc}(F(S))$ consisting of those $f\in\mathrm{Lsc}(F(S))$ that can be written as the pointwise supremum of a sequence $(f_n)_{n\in\N}$ in $\mathrm{Lsc}(F(S))$ such that $f_n\lhd f_{n+1}$ for all $n\in\N$.

One has that in fact $S_R=L(F(S))$, as was shown in \cite[Theorem 3.2.1]{Rob_cone_2013}. It was also proved in \cite[Theorem 4.2.2]{Rob_cone_2013} that $L(F(S))$ is inf-semilattice ordered. The semigroup $L(F(S))$ is thought of as the \emph{dual} of $F(S)$,
since $F(L(F(S))=F(S)$, although it is not known whether $L(F(S))=\Lsc(F(S))$.

\begin{pbm}\emph{The problem of realizing functions as ranks.}
\label{probrealizing} 
Let $S$ be a $\Cu$-semigroup satisfying (O5).
The problem of realizing functions on $F(S)$ as ranks of elements in $S$ consists of finding necessary and sufficient conditions for the map $x\mapsto \widehat{x}$ to be a surjection from $S$ to $L(F(S))$.
\end{pbm}

The following notion is crucial to solve the problem of realization of ranks in the stable rank one setting. The motivation for the terminology can be found in \cite{ThiVil_nowscat_2021}.

\begin{df}
An \emph{ideal-quotient} in a C$^*$-algebra $A$ is a quotient of the form $I/J$, where $J\subset I$ are closed, two-sided ideals of $A$.
A C$^*$-algebra is \emph{nowhere scattered} if it has no non-zero elementary\footnote{Recall that a C$^*$-algebra is said to be \emph{elementary} if it is isomorphic to the compact 
operators on some Hilbert space. Equivalently, a \ca\ $A$ is 
elementary if it is simple and there exists a projection $p\in A$
satisfying $pAp\cong \C$.} ideal quotients. 
\end{df}

Nowehere scatteredness can be nicely characterized in terms
of Cuntz semigroups and functionals. To this end we need the lemma below, which is a nice
application of the axioms (O5) and (O6) of independent interest. 
Recall the notation $\infty_s$ from \autoref{nota:infa}.

\begin{lma}
\label{lma:auxiliary}
Let $S$ be a $\Cu$-semigroup satisfying axioms (O5) and (O6), 
let $\lambda\in F(S)$ satisfying 
\begin{enumerate}[{\rm (i)}]
\item $\lambda(S)=\NN$, and 
\item $\lambda(s)=0$ if and only if $s=0$,
\end{enumerate}
and fix $s_0\in S$ with $\lambda(s_0)=1$. 
Then $I=\{s\in S\colon s\leq \infty_{s_0}\}$ is an ideal in $S$, 
and $\lambda$ restricts to an isomorphism 
$I\cong\NN$ of $\Cu$-semigroups.
\end{lma}
\begin{proof}
Let $s\in S$ satisfy $\lambda(s)<\infty$.
We claim that $s$ is compact. To see this, 
write $s=\sup\limits_{n\in\N} t_n$ with $t_n\ll t_{n+1}$. Since $\lambda(s_0)=\sup\limits_{n\in\N}\lambda(t_n)$ is compact in $\overline{\N}$, 
there exists $n_0\in\N$ such that for all $n\geq n_0$ we have $\lambda(t_n)=1$. 
For such $n\geq n_0$, apply (O5) to $t_n\ll t_{n+1}\leq s$ to find $c\in S$ such that $t_n+c\leq s\leq t_{n+1}+c$. 
Applying $\lambda$ yields $\lambda(c)=0$, which implies $c=0$ by (i). This shows that $s$ is compact. 

Fix $m\in\N$ and $t\in S$. We claim that $t\leq ms_0$ if and only if there exists $k\leq m$ with $t=ks_0$. 
One implication is obvious, so we prove the other one. 
The
inequality $t\leq ms_0$ implies that $\lambda(t)\leq m\lambda(s_0)<\infty$, and thus $t$ is compact by the previous claim.
In order to establish the claim, we may clearly assume that $t\neq 0$ and proceed by induction on $m$. 
Suppose that $m=1$. Use \autoref{rem:compl} to find $t'\in S$ with $t+t'=s_0$. Applying $\lambda$ we get $\lambda(t')=0$, which again by (i) implies that $t'=0$,
showing that $t=s_0$. This proves the case $m=1$ of the induction.

Assume now that $t\leq ms_0=(m-1)s_0+s_0$. Apply (O6) to find elements $t_1, t_2\in S$ such that $t_1\leq (m-1)s_0, t$ and $t_2\leq t, s_0$. Notice that $t_1$ and $t_2$ are compact elements as well, since they 
are dominated by the compact element $t$. 
By the induction assumption, there is $k\leq m-1$ such that $t_1=ks_0$. If $t_2=0$, then $ks_0=t_1\leq t\leq t_1+t_2=ks_0$, and hence $t=ks_0$. If $t_2\neq 0$, then $t_2=s_0$ since $t_2\leq s_0$. Now $t_1=ks_0\leq t\leq t_1+t_2=(k+1)s_0$, and we may find elements $c,d\in S$ such that $ks_0+c=t$ and $t+d=(k+1)s_0$. Putting these equalities together we obtain $ks_0+c+d=(k+1)s_0$, and applying $\lambda$ we get $\lambda(c+d)=1$. Thus one of $c$ or $d$ must be zero, and hence $t$ equals either $ks_0$ or $(k+1)s_0$, as desired. This proves the claim.

Let $t\in S\setminus\{0\}$ satisfy $t\leq \infty_{s_0}$.
Find a $\ll$-increasing sequence $(t_n)_{n\in\N}$ in $S$ with  $t=\sup\limits_{n\in\N} t_n$. Fix $n\in\N$. Since 
\[t_n\ll t =\sup\limits_{m\in\N} ms_0,\]
there exists $k\in\N$ with $t_n\leq m_ns_0$. 
By the second claim, there exists $k_n\leq m_n$ such that 
$t_n=k_ns_0$. Thus we must have either $t=ks_0$ for some $k\in\N$, 
or else $t=\infty_{s_0}$.

Set $I=\{s\in S\colon s\leq \infty_{s_0}\}$, and observe that $I$ is
an ideal in $S$. In particular, $I$ is a $\Cu$-semigroup. The 
paragraph above shows that $I=\{ms_0\in S\colon m\in\NN\}$. 
Furthermore, by applying $\lambda$ we see that $ks_0\leq ms_0$ precisely when $k\leq m$. In other words, this shows that $\lambda$ restricts to 
a Cu-semigroup isomorphism $I\cong \NN$.
\end{proof}

\begin{prop}
Let $A$ be a C$^*$-algebra. Then the following are equivalent:
\be[{\rm (i)}]\item $A$ is \emph{not} nowhere scattered,
\item there exists a functional $\lambda \in F(\Cu(A))$
such that $\lambda(\Cu(A))\cong \overline{\N}$.\ee
\end{prop}
\begin{proof}
Assume that (i) holds, 
and find closed, two sided ideals $I,J$ in $A$ 
with $J\subseteq I$ such that $I/J$ is elementary.
Then $\Cu(I/J)\cong\NN$ by \autoref{eg:CuC}.
Moreover, by \autoref{thm:quotients}
the quotient map $\pi\colon I\to I/J$ induces a surjective $\Cu$-morphism $\Cu(\pi)\colon\Cu(I)\to \Cu(I/J)\cong\NN$;
in particular, $\Cu(\pi)$ is a functional on $\Cu(I)$. 
Let $\lambda\colon \Cu(A)\to [0,\infty]$ be the functional
obtained by extending
$\Cu(\pi)$ to $\Cu(A)$ as 
in \autoref{prp:extFctl}.
It is then clear from the definition of $\lambda$ that $\lambda(\Cu(A))=\NN$.

Conversely, write $S=\Cu(A)$ and let $\lambda\colon S\to [0,\I]$ be a functional satisfying $\lambda(S)=\NN$. Consider the ideal $K=\lambda^{-1}(0)$. It is easy to check that $\lambda$ induces a functional $\overline{\lambda}\colon S/K\to [0,\I]$ such that $\overline{\lambda}(S/K)=\NN$. We may now apply \autoref{lma:auxiliary} to obtain an ideal $\overline{L}$ of $S/K$ which is isomorphic to $\NN$. Thus, there is an ideal $L$ of $S$ that contains $K$ such that $L/K\cong\NN$. This implies that $A$ has an ideal-quotient whose Cuntz semigroup is isomorphic to $\NN$, and thus this ideal-quotient
is elementary by \cite[Lemma~8.2]{ThiVil_nowscat_2021}.\footnote{Indeed, if a \ca\ $B$ satisfies $\Cu(B)\cong \overline{\N}$, it is not hard to show that $B$ must be elementary. The basic idea is as follows: since $B$ is stably finite, for example by \autoref{thm:Cuntz}, there is a 
projection $p\in B$ which represents the compact element $1\in\overline{\N}$. One can show that $pBp\cong \C$. 
Since $B$ is moreover simple, for example by \autoref{thm:CuIdeals},
it follows that $B$ is elementary.}
Therefore $A$ is not nowhere scattered.
\end{proof}

The key to solve the problem posed in \autoref{probrealizing} in the stable rank one case relies on the following:

\begin{df}[The map $\alpha$]
\label{alpha}
Let $A$ be a C$^*$-algebra with stable rank one. Define
\[
\alpha\colon L(F(\Cu(A))\to \Cu(A)
\]
by setting $\alpha(f)=\sup\{x\in\Cu(A) \colon \widehat{x}\ll f\}$.
\end{df}

It is not at all obvious that the set $\{x\in\Cu(A) \colon \widehat{x}\ll f\}$ has a supremum -- this follows from 
the fact that $A$ is assumed to have stable rank one, and 
is shown in \cite[Theorem 7.2, Proposition 7.3]{AntPerRobThi_stable_2022}. 

\begin{lma}
The map $\alpha$ from \autoref{alpha} preserves order, suprema of increasing sequences, and infima of pairs of elements.
\end{lma}
\begin{proof}
The core of the argument consists of proving that in fact $\alpha(f)=\sup I_f$, where
\[
I_f=\{x\in\Cu(A)\colon \widehat{y}\ll f\text{ for all }y\ll x\}.
\]
We will not prove this, and isntead we will only explain how to 
to prove the lemma using it.
To see that $\alpha$ preserves the order, let $f\leq g$ in $L(F(\Cu(A)))$.
Then $I_f\subseteq I_g$, and thus $\alpha(f)\leq \alpha(g)$.

To check that $\alpha$ preserves suprema, let $(f_n)_{n\in\N}$ be an increasing sequence of elements in $L(F(\Cu(A)))$, and set $f =\sup\limits_{n\in\N} f_n$.
Since, as observed, $\alpha$ is order-preserving, the sequence $(\alpha(f_n))_{n\in\N}$ is increasing in $\Cu(A)$ and thus it has a supremum $x =\sup\limits_{n\in\N} \alpha(f_n)$.
Since $\alpha(f_n)\leq \alpha(f)$ for all ${n\in\N}$, we have $x\leq \alpha(f)$.
Now let $z\in\Cu(A)$ satisfy $\widehat{z}\ll f$.
Using that, by definition, $\alpha(f)$ is the supremum of all such $z$, it suffices to show that $z\leq x$.
As $\widehat{z}\ll f$, we have $\widehat{z}\ll f_n$ for some $n\in\N$, and then $z\in I_{f_n}$. Therefore $z\leq \alpha(f_n)\leq x$.

Finally, to prove that $\alpha$ preserves infima, let $f,g\in L(F(\Cu(A)))$. 
Since $\alpha$ is order-preserving, we get $\alpha(f\wedge g)\leq \alpha(f)\wedge \alpha(g)$. To show the converse inequality, let $0<\varepsilon<1$ and suppose that $z\leq \alpha((1-\varepsilon)f)\wedge \alpha((1-\varepsilon)g)$.
Then 
\[
\widehat{z}\leq (1-\varepsilon)f\wedge (1-\varepsilon)g=(1-\varepsilon)(f\wedge g),
\] 
whence $z\leq \alpha(f\wedge g)$. This implies $\alpha((1-\varepsilon)f)\wedge \alpha((1-\varepsilon)g)\leq\alpha(f\wedge g)$. Finally, let $\varepsilon\to 0$ and use that $\alpha$ preserves suprema of increasing sequences to obtain $\alpha(f)\wedge\alpha(g)\leq \alpha(f\wedge g)$.
\end{proof}

Recall the definition of $\infty_a$ from \autoref{nota:infa}.

\begin{thm}
\label{thm:realization}
Let $A$ be a separable, nowhere scattered C$^*$-algebra of stable rank one. Then, for all $f\in L(F(\Cu(A)))$ we have
\[
f=\widehat{\alpha(f)}.
\] 
\end{thm}
\begin{proof} (Outline)
The set $\{x\in\Cu(A)\colon \widehat{x}\leq \infty_f\}$ is an ideal of $\Cu(A)$, and thus has the form $\Cu(I)$ for a closed two-sided ideal $I$ of $A$. Note that $I$ is automatically separable, nowhere scattered, and has stable rank one.

Using that $L(F(\Cu(A)))=\Cu(A)_R$, one can choose a sequence 
$(x_n)_{n\in\N}$ in $\Cu(A)$ and $(k_n)_{n\in\N}$ in $\N$ such that $f=\sup\limits_{n\in\N}\frac{\widehat{x_n}}{k_n}$. Notice that $x_n\in I$ for all $n\in\N$. Considering $\frac{\widehat{x_n}}{k_n}\in L(F(\Cu(I)))$, denote by $f_0$ its supremum, so that $f(\lambda)=f_0(\lambda_{|\Cu(I)})$.

One now checks that $f_0$ is full in $L(F(\Cu(I)))$. By letting $\alpha_I\colon L(F(\Cu(I)))\to\Cu(I)$ be the map as in \autoref{alpha} it is possible to show, with considerable effort, that $f_0=\widehat{\alpha_I(f_0)}$; see \cite[Theorem 7.10]{AntPerRobThi_stable_2022}.

Next, we claim that $\alpha(f)=\alpha(f_0)$. To see this, let
\[
L= \big\{ x\in\Cu(A) : \widehat{x}\leq (1-\varepsilon)f\text{ in } L(F(\Cu(A))) \text{ for some }\varepsilon>0 \big\}.
\]
Using the definition of $\alpha$, we have that $\alpha(f)$ is the supremum of $L$ in $\Cu(A)$.
If $x\in\Cu(A)$ and $\varepsilon>0$ satisfy $\widehat{x}\leq(1-\varepsilon)f$ in $L(F(\Cu(A)))$, then $x$ belongs to $\Cu(I)$ and hence $\widehat{x}\leq(1-\varepsilon)f_0$ in $L(F(\Cu(I)))$.
This implies that $L\subseteq\Cu(I)$ and so $\alpha(f_0)$ is the supremum of $L$ in $\Cu(I)$.
Using that $\Cu(I)\subseteq\Cu(A)$ is hereditary, we see that the supremum of $L$ in $\Cu(I)$ and in $\Cu(A)$ agree. Therefore $\alpha(f)=\alpha(f_0)$, as was claimed.

Finally, if $\lambda\in F(\Cu(A))$, we have
\[
\widehat{\alpha(f)}(\lambda)
= \lambda(\alpha(f))
= \lambda|_{\Cu(I)}(\alpha(f_0))
= f_0(\lambda|_{\Cu(I)})
= f(\lambda),
\]
and thus $\widehat{\alpha(f)}=f$ in $L(F(\Cu(A)))$.
\end{proof}

\section{Structure of the category \textbf{Cu}}
\label{sec:Structure}

This section will be devoted to discussing the following result, 
which combines results from a number of papers; see \cite{AntPerThi_tensor_2018, AntPerThi_abstract_2020,AntPerThi_ultraproducts_2020}.

\begin{thm}
\label{thm:cu-structure}
The category $\CatCu$ of abstract Cuntz semigroups is a closed, symmetric, monoidal, bicomplete category.
\end{thm}

Below we shall present the constructions that are used in order to prove the above theorem; see \autoref{monoidal}, \autoref{thm:closed}, and \autoref{thm:Cucomplete}. We will also give the relevant definitions of the concepts that appear in its statement.

\begin{df}
\label{pgr:CatQtau}
A \emph{$\mathcal{Q}$-semigroup} is a positively ordered monoid satisfying axioms (O1) and (O4) from \autoref{df:CatCu}, together with an additive auxiliary relation $\prec$ as in \autoref{pgr:aux}.
A \emph{$\mathcal{Q}$-morphism} is a morphism of positively ordered monoids that preserves the auxiliary relation and suprema of increasing sequences.
Given $\mathcal{Q}$-semigroups $S$ and $T$, we denote by $\mathrm{\textbf{Q}}(S,T)$ the set of all $\mathcal{Q}$-morphisms from $S$ to $T$.
\end{df}

Given a $\Cu$-semigroup $S$, then $S$ together with the relation $\ll$ is a $\mathcal{Q}$-semigroup.
Moreover, given two $\Cu$-semigroups $S$ and $T$, a map $\varphi\colon S\to T$ is a $\Cu$-morphism if and only if it is a $\mathcal{Q}$-morphism when considered as a map from $(S,\ll)$ to $(T,\ll)$.
We obtain a functor $\iota\colon\CatCu\to\mathrm{\textbf{Q}}$ that sends a $\Cu$-semigroup $S$ to the $\mathcal{Q}$-semigroup $(S,\ll)$, and embeds \textbf{Cu} as a full subcategory of $\mathrm{\textbf{Q}}$.

\begin{df}[The $\tau$-construction]
Let $S=(S,\prec)$ be a $\mathcal{Q}$-semigroup.
A \emph{path} in $S$ is an order-preserving map $f\colon (-\infty,0]\to S$ such that $f(t)=\sup\limits_{t'<t}f(t')$ for all $t\in(-\infty,0]$, and such that $f(t')\prec f(t)$ whenever $t'<t$.
We denote the set of paths in $S$ by $\Paths(S)$.

Pointwise addition, together with the constant zero path, give $\Paths(S)$ the structure of a commutative monoid.
Given $f,g\in \Paths(S)$, we write $f\precsim g$ if, for every $t<0$, there is $t'<0$ such that $f(t)\prec g(t')$.
Set $f\sim g$ if $f\precsim g$ and $g\precsim f$.
It follows from \cite[Lemma~3.4]{AntPerThi_abstract_2020} that the relation $\precsim$ is reflexive, transitive, and compatible with addition of paths.
We set
\[
\tau(S) := \Paths(S)/\!\sim.
\]
Given $f\in \Paths(S)$, its equivalence class in $\tau(S)$ is denoted by $[f]$.
Equip $\tau(S)$ with an addition and order by setting $[f]+[g]:=[f+g]$ and $[f]\leq [g]$ if $f\precsim g$.
\end{df}

The construction just outlined will be referred to as the \emph{$\tau$-construction}, and it has a number of features that we now describe.

\begin{thm}(\cite[Theorem~3.15]{AntPerThi_abstract_2020}).
\label{thm:tau}
Retaining the notation in \autoref{pgr:CatQtau}, if $S$ is a $\mathcal{Q}$-semigroup, then $\tau(S)$ is a $\Cu$-semigroup.
\end{thm}
\begin{proof} (Outline)
Let $([f_n])_{n\in\N}$ be an increasing sequence of paths in $S$. An inductive process allows one to construct a strictly increasing sequence $(t_m)_{m\in\N}$ in $(-\infty, 0]$ and a path $f$ in $S$ satisfying
\begin{enumerate}[{\rm (i)}]
\item $\sup\limits_{m\in\N}t_m=0$
\item $f_n(t_m)\prec f_l(t_l)$ whenever $n,m<l$.
\item $f_n(t_n)=f(\frac{-1}{n+1})$ for all $n\geq 1$.
\end{enumerate} 
Then one can verify that $[f]=\sup\limits_{n\in\N} [f_n]$ and thus $\tau(S)$ satisfies (O1).
In order to verify (O2), given $f\in \Paths(S)$ and $\varepsilon>0$, define $f_\varepsilon\colon (-\infty,0]\to S$ by 
\[
f_\varepsilon(t)=\begin{cases}
		f(t), & \text{if } t<-\varepsilon \\
        0, &\text{otherwise}.
       \end{cases}
\]
If follows by construction that $[f]=\sup\limits_{\varepsilon>0}[f_\varepsilon]$. We need to check that $[f_\varepsilon]\ll [f]$.
To this end, let $([g_n])_{n\in\N}$ be an increasing sequence in $\tau(S)$ such that $[f]\leq \sup\limits_{n\in\N}[g_n]$. By the construction outlined at the beginning of the proof, there are a path $g$ in $S$ and an increasing sequence $(t_m)_{m\in\N}$ in $(-\infty,0]$ with supremum  
$0$ such that $[g]=\sup\limits_{n\in\N}[g_n]$ and $g(\frac{-1}{m})=g_m(t_m)$ for all $m\in\N$.

Choose $m_0\geq 1$ such that $-\varepsilon<-\frac{1}{m_0}$. Since $f\precsim g$, there is $r>0$ such that $f(-\frac{1}{m_0})\prec g(r)$. Now choose $m_1$ such that $r<-\frac{1}{m_1+1}$. Then, if $t<-\varepsilon$, we have
\[
f_\varepsilon(t)=f(t)\prec f\big(\tfrac{-1}{m_0}\big)\prec g(r)\prec g\big(\tfrac{-1}{m_1+1}\big)=g_{m_1}(t_{m_1}),
\] 
which shows that $f_\varepsilon\prec g_{m_1}$. This implies that $\tau(S)$ satisfies (O2). Axioms (O3) and (O4) are more routine to verify.
\end{proof}

We define $\varepsilon_S\colon\tau(S)\to S$ by $\varepsilon_S([f])= f(0)$ for $f\in\Paths(S)$.
One can check that $\varepsilon_S$ is a well-defined $\mathcal{Q}$-morphism. Since for a path $f\in\Paths(S)$, we think of $f(0)$ as the endpoint of $f$, we call $\varepsilon_S$ the \emph{endpoint map}.

Given a $\mathcal{Q}$-morphism $\varphi\colon S\to T$, we define $\tau(\varphi)\colon\tau(S)\to\tau(T)$ by $\tau(\varphi)([f])=[\varphi\circ f]$ for $f\in \Paths(S)$.
One can show that $\tau(\varphi)$ is a $\Cu$-morphism.
This defines a covariant functor $\tau\colon\mathrm{\textbf{Q}}\to\CatCu$.
We omit the proof of the following result.

\begin{thm}[{\cite[Theorem~4.12]{AntPerThi_abstract_2020}}]
\label{prp:CuCoreflQ}
The category $\CatCu$ is a full, coreflective subcategory of $\mathrm{\textbf{Q}}$.
The functor $\tau\colon\mathrm{\textbf{Q}}\to\CatCu$ is a right adjoint to the inclusion functor $\iota\colon\CatCu\to\mathrm{\textbf{Q}}$.
Moreover, given a $\mathcal{Q}$-semigroup $S$, the endpoint map $\varepsilon_S\colon\tau(S)\to S$ is a universal $\mathcal{Q}$-morphism, in the sense that 
for every $\Cu$-semigroup $T$, there is a natural bijection
\[
\CatCu\big( T,\tau(S) \big) \cong \mathrm{\textbf{Q}}\big( \iota(T),S \big),
\]
implemented by sending a $\Cu$-morphism $\psi\colon T\to\tau(S)$ to $\varepsilon_S\circ\psi$.
\end{thm}

We now define what is meant by a monoidal category; see also \cite{Mac_Categories_1971}.
\begin{df}
A \emph{monoidal category} is a category $\mathbf{C}$ together with a bifunctor $\otimes\colon\mathbf{C}\times\mathbf{C}\to \mathbf{C}$, a unit object $I$, and natural isomorphisms 
\[
(X\otimes Y)\otimes Z\cong X\otimes (Y\otimes Z)\ \ \ \mbox{ and } \ \ \  I\otimes X\cong X\cong X\otimes I,
\]
whenever $X,Y,Z$ are objects in $\mathbf{C}$. (Other coherence axioms are also required.) We say that $\mathbf{C}$ is moreover \emph{symmetric} if $X\otimes Y\cong Y\otimes X$ for any $X,Y\in\mathbf{C}$.
\end{df}

\begin{thm}
\label{monoidal}
The category $\CatCu$ is monoidal.
\end{thm}
\begin{proof}
To show that $\CatCu$ is monoidal we need to construct the tensor product of any two $\Cu$-semigroups $S$ and $T$. We only sketch here how to proceed and refer the reader to the material in \cite[Chapter 6]{AntPerThi_tensor_2018}.

Given $\Cu$-semigroups $S$, $T$, we first form the algebraic tensor product $S\odot T$ as positively ordered monoids, based on expressions on the free abelian monoid $\N[S^\times\times T^\times]$  so that if $a'\leq a$ in $S$ and $b'\leq b$ in $T$, one has $a'\odot b'\leq a\odot b$. For $f,g\in \N[S^\times\times T^\times]$, we set $f\preceq g$ provided $g=\sum_{j\in J} a_j\odot b_j$ and $f\leq \sum_{j\in J'}a_j'\odot b_j'$, where $J'\subseteq J$ and $a_j'\ll a_j$, $b_j'\ll b_j$ for each $j\in J'$.

With this structure one can show, after considerable effort, that $(S\odot T,\preceq)$ is a $\W$-semigroup and it follows that $\gamma(S\odot T)$ is the tensor product $S\otimes_\Cu T$ of $S$ and $T$ in the category $\CatCu$, where $\gamma$ is as in the proof of \autoref{thm:curefl}. The unit for the tensor product is $\NN$.
\end{proof}

We now focus on the adjoint of the tensor product.

\begin{df}
We say that a monoidal category $\mathbf{C}$ is \emph{closed} if, for each object $Y\in\mathbf{C}$, the functor $-\otimes Y\colon\mathbf{C}\to\mathbf{C}$ has a right adjoint, that we denote by $\ihom{Y,X}$.
\end{df}

In a closed, monoidal category \textbf{C}, we thus have a natural bijection (in $X$ and~$Z$)
\[
\mathbf{C}(X\otimes Y, Z)\cong \mathbf{C}(X, \ihom{Y,Z}).
\]

Unlike in the category of abelian groups, the adjoint of the tensor
product in \textbf{Cu}
does not merely consist of the usual hom-set in the category. The reason for this is that, for $\Cu$-semigroups $S$ and $T$, the set $\CatCu(S,T)$ may be too small, and not even a $\Cu$-semigroup. Indeed, for any $\Cu$-semigroup $S$ we have $\CatCu(\NN,S)\cong S_c$, so
in particular $\CatCu(\NN,\NN)\cong\N$ is not a Cu-semigroup. Instead, one has to consider the set of generalized $\Cu$-semigroups, as in \autoref{df:CatCu}. Notice that, given $\Cu$-semigroups $S$ and $T$, the set $\CatCu[S,T]$ is a positively ordered monoid, when equipped with pointwise order and addition. It also satisfies axioms (O1) and (O4): the supremum of a sequence of generalized $\Cu$-morphisms is just the pointwise supremum.

Define a binary relation on $\CatCu[S,T]$ by setting $f\prec g$ if, whenever $s'\ll s$ in $S$, we have $f(s')\ll g(s)$ in $T$. With this relation, a generalized $\Cu$-morphism $f\colon S\to T$ is a $\Cu$-morphism precisely when $f\prec f$. It is an easy exercise to verify that $\prec$ is an auxiliary relation for the pointwise order in $\CatCu[S,T]$ which makes the latter into a $\mathcal{Q}$-semigroup, although not necessarily a $\Cu$-semigroup.

For example, write $\mathbb{P}=[0,\infty]$, which is known to be the Cuntz semigroup of the Razak algebra. One can show that $(\mathbf{Cu}[\mathbb{P},\mathbb{P}],\prec)$ is isomorphic to $(\mathbb{P},\prec_1)$, where $\prec_1$ is defined as follows: $a\prec_1 b$ if and only if $a\leq\infty$ and $a\leq b$; see \cite[Examples 4.14 and 5.13]{AntPerThi_abstract_2020}. With this structure, $(\mathbb{P},\prec_1)$ is a $\mathcal{Q}$-semigroup, and the $\tau$-construction applied to it yields $[0,\infty)\sqcup (0,\infty]$ where the compact elements are the ones in the first component. Let us denote these elements by $c_a$, for $a\in [0,\infty)$, and the elements in the second component by $s_a$, for $a\in (0, \infty]$. Addition and order are given by:
\begin{itemize}
\item $c_a+c_b=c_{a+b}$, $s_a+s_b=s_{a+b}$, and $c_a+s_b=s_{a+b}$.
\item $s_a\leq c_b$ if and only if $a\leq b$; $c_a\leq s_b$ if and only if $a<b$.
\end{itemize}

\begin{thm}
\label{thm:closed}
The category $\CatCu$ is closed.
\end{thm}
\begin{proof}
We define $\ihom{S,T}=\tau(\CatCu[S,T],\preceq)$, which by \autoref{thm:tau} is a $\Cu$-se\-mi\-group, and refer to this as the \emph{internal-hom functor} for $S$ and $T$.

The internal-hom functor is the adjoint of the tensor product as defined in \autoref{monoidal}. In fact, it was shown in  \cite[Theorem 5.10]{AntPerThi_abstract_2020} that there is a natural isomorphism of positively ordered monoids
\[
\CatCu(S,\ihom{T,P})\cong \CatCu(S\otimes_\Cu T,P).\qedhere
\]
\end{proof}

In these notes we will only discuss completeness (and not cocompleteness), as it is what will be used below in the construction of ultraproducts. We now recall the basic notions.

\begin{df}
Let \textbf{C} be a category, let \textbf{I} be a small category, and let $\mathrm{F}\colon \mathrm{\mathbf{I}}\to\mathrm{\textbf{C}}$ be a functor.
A \emph{cone} to $\mathrm{F}$ is a pair $(L,\varphi)$, where $L$ is an object in $\mathrm{\textbf{C}}$ and $\varphi=(\varphi_i)_{i\in \mathrm{\mathbf{I}}}$ is a collection of morphisms $\varphi_i\colon L\to \mathrm{F}(i)$ in $\mathrm{\textbf{C}}$ such that $\mathrm{F}(f)\circ\varphi_i=\varphi_j$ for any morphism $f\colon i\to j$ in $\mathrm{\mathbf{I}}$. 

A \emph{small limit} of $\mathrm{F}$ is a universal cone $(L,\varphi)$, that is, if $(L',\varphi')$ is another cone, there exists a unique $\mathrm{\textbf{C}}$-morphism $\alpha\colon L'\to L$ such that $\varphi_i'=\varphi_i\circ\alpha$ for every $i\in \mathrm{\mathbf{I}}$. This is summarized in the diagrams below:
\[
\xymatrix@R-20pt{
& \mathrm{F}(i) \ar[dd]^-{\mathrm{F}(f)}
& &  &  & \mathrm{F}(i) \ar[dd]^-{\mathrm{F}(f)} \\
L \ar[ur]^-{\varphi_i} \ar[dr]_-{\varphi_j}
&  &   & L' \ar[r]^\alpha \ar@/^1.3pc/[urr]^-{\varphi_i'} \ar@/_1.3pc/[drr]_-{\varphi_j'}
& L \ar[ur]^-{\varphi_i} \ar[dr]_-{\varphi_j}  &   \\
& \mathrm{F}(j), & & & & \mathrm{F}(j). \\
}
\]

If a limit of $\mathrm{F}$ exists, then it is unique up to natural isomorphism, and we denote it by  $(\mathrm{\textbf{C}}\text{-}\!\varprojlim \mathrm{F},\pi)$ or just by $\mathrm{\textbf{C}}\text{-}\!\varprojlim \mathrm{F}$. 

The category $\mathrm{\textbf{C}}$ is said 
to be \emph{complete} if all functors from small categories into 
it have limits. (Cocompleteness is defined dually.)
\end{df}

The following is proved in \cite[Theorem 3.8]{AntPerThi_ultraproducts_2020}.

\begin{thm}
\label{thm:Cucomplete}
The category $\CatCu$ is complete.
\end{thm}
\begin{proof}
Let F$\colon \mathbf{Cu}\to \mathbf{I}$ be a functor into a small
category \textbf{I}.
The basic strategy consists of completing, via the $\tau$-construction, the limit of the composition 
$\mathbf{Cu}\to \mathbf{I}\hookrightarrow \mathbf{Q}$.
We outline the procedure below.

Denote by $\CatPom$ the category of positively ordered monoids.
Let $(S_i)_{i\in \mathbf{I}}$ be a collection of objects in 
$\CatPom$. The product of this family in $\CatPom$ is given by
\[
\mathrm{\textbf{PoM}}-\prod_{i\in \mathbf{I}}S_i=\{(s_i)_{i\in \mathbf{I}}\colon s_i\in S_i \text{ for all }i\in \mathbf{I}\},
\]
with componentwise addition and order. 
Regarding F as a functor 
$\mathrm{F}\colon \mathrm{\mathbf{I}}\to\CatPom$, set
\begin{align*}
\label{pgr:prodinvlim:eq2}
S:= \Big\{ (s_i)_{i\in \mathrm{\mathbf{I}}}\in\CatPom\text{-}\prod_{i\in \mathrm{\mathbf{I}}} \mathrm{F}(i) \colon  \mathrm{F}(f)(s_i)=s_j \text{ for all } f\colon i\to j \text{ in } \mathrm{\mathbf{I}} \Big\}.
\end{align*}
It is straightforward to verify that $0\in S$ and that $S$ is closed under addition in $\CatPom$-$\prod_{i\in \mathrm{\mathbf{I}}} \mathrm{F}(i)$, hence $S$ is also a positively ordered monoid. 

For each $i\in \mathrm{\mathbf{I}}$, the projection map $\pi_i\colon\CatPom\text{-}\prod_{j\in J} \mathrm{F}(j)\to \mathrm{F}(i)$ restricts to a $\CatPom$-morphism $\pi_i\colon S\to \mathrm{F}(i)$. 
Set $\pi=(\pi_i)_{i\in \mathrm{\mathbf{I}}}$.
It is also straightforward to verify that $(S,\pi)$ 
is the limit of $\mathrm{F}$ in $\CatPom$. If further $\mathrm{F}(i)$ satisfies (O1) and (O4) for each $i\in \mathrm{\mathbf{I}}$, then this is also the case for $S$.

Also, since the range of $\mathrm{F}$ is contained in $\mathrm{\textbf{Q}}$ (in fact, in $\CatCu$), we may define an auxiliary relation $\prec_{\mathrm{pw}}$ on $S$ by stating $(s_i)\prec_\mathrm{pw} (t_i)$ precisely when $s_i\prec t_i$ in $\mathrm{F}(i)$ for each $i\in \mathrm{\mathbf{I}}$.
This construction shows that $(S,\prec_\mathrm{pw})$ is the limit of $\mathrm{F}$ in the category $\mathrm{\textbf{Q}}$. In fact, the 
argument just outlined shows that $\mathrm{\textbf{Q}}$ is complete.

Let $(S,(\pi_i)_{i\in \mathrm{\mathbf{I}}})$ be the limit of $\mathrm{F}$ in $\mathrm{\textbf{Q}}$ as outlined, and let $\tau(S)$ be the $\tau$-completion of $S$.
Set $S_i:=\tau(\mathrm{F}(i))$, which we identify with $\mathrm{F}(i)$, and set 
\[\psi_i=\tau(\pi_i)\colon\tau(S)\to\tau(S_i)\cong S_i.\]
Then $\tau(S)$ together with $(\psi_i)_{i\in \mathrm{\mathbf{I}}}$ is the limit of $\mathrm{F}$ in $\CatCu$, that is:
\[
\CatCu\text{-}\varprojlim \mathrm{F} 
= \tau\left( \mathrm{\textbf{Q}}\text{-}\varprojlim \mathrm{F} \right)
= \tau\left( \CatPom\text{-}\varprojlim \mathrm{F}, \ll_\mathrm{pw} \right).\qedhere
\]
\end{proof}

As an immediate consequence, we have:

\begin{cor}
\label{cor:Cuprod}
The category $\CatCu$ has arbitrary products, arbitrary inverse limits, and finite pullbacks.
In particular, if $(S_i)_{i\in I}$ is a family of $\Cu$-semigroups, then
\[
\mathrm{\textbf{Cu}}\text{-}\prod\limits_{i\in I} S_i
=\tau\left( \mathrm{\textbf{Q}}\text{-}\prod\limits_{i\in I} S_i \right)
=\tau\left( \CatPom\text{-}\prod\limits_{i\in I} S_i, \ll_\mathrm{pw} \right).
\]
\end{cor}

Since we are mostly interested in the category $\CatCu$,
we will from now on denote the product in $\CatCu$ simply by 
$\prod$, instead of
$\mathrm{\textbf{Cu}}\text{-}\prod$. (The notation should
not be confused with the product of C$^*$-algebras.)

\section{Applications to products and ultraproducts of \texorpdfstring{C$^*$}{C*}-algebras}
\label{sec:prod}

In this section we explore the extent to which the Cuntz semigroup functor preserves products and ultraproducts of C$^*$-algebras. We begin with a simple observation, which is a consequence of the universal property of the product.

\begin{rem}\label{rem:UnivPropProd}
Let $(A_j)_{j\in J}$ be a family of C$^*$-algebras, and 
set $A=\prod_{j\in J}A_j$. For each $j\in J$, the natural projection $\pi_j\colon A\to A_j$ induces a $\Cu$-morphism $\tilde{\pi}_i\colon\Cu(A)\to \Cu(A_j)$. 
Then, by the universal property of the product, there is a unique $\Cu$-morphism
\[
\Phi\colon \Cu(A)\to \prod_{j\in J} \Cu(A_j)
\]
such that $\tilde{\pi}_i=\sigma_i\circ\Phi$ for all $j\in J$, where $\sigma_i\colon\prod_{j\in J}\Cu(A_j)\to\Cu(A_i)$ denotes the natural $\Cu$-morphism associated to the product in the category $\CatCu$. 
\end{rem}

\begin{lma}
\label{prp:PhiInjective}
Let $(A_j)_{j\in J}$ be a family of C$^*$-algebras, and 
set $A=\prod_{j\in J}A_j$. Let $\underline{a}=(a_j)_{j\in J}$ and 
$\underline{b}=(b_j)_{j\in J}\in A_+$. Then $\underline{a}\precsim\underline{b}$ in $A$ if, and only if, 
$\Phi([\underline{a}])\leq\Phi([\underline{b}])$ in $\prod_{j\in J}\Cu(A_j)$.
\end{lma}
\begin{proof}
We show that the conditions in the statement are equivalent to: 
\be 
\item[($\ast$)] For every $\varepsilon>0$ there exists $\delta>0$ such that $(a_j-\varepsilon)_+\precsim(b_j-\delta)_+$ in $A_j$ for every $j\in J$.\ee

To see this, set $\varphi_t(\underline{a}) = ([(a_j+t)_+])_j$, and $\varphi_t(\underline{b}) = ([(b_j+t)_+])_j$, for any $t\in(-\infty,0]$, whence $\Phi([\underline{a}]) = [ \varphi_t(\underline{a})_{t\in(-\infty,0]} ]$,  and $\Phi([\underline{b}]) = [ \varphi_t(\underline{b})_{t\in(-\infty,0]} ]$.

As shown in \autoref{thm:Cucomplete}, we have 
\[\prod_{j\in J} \Cu(A_j)=\tau\Big(\CatPom\text{-}\prod_{j\in J}\Cu(A_j),\ll_\mathrm{pw}\Big),\] 
and thus $\Phi([\underline{a}])\leq\Phi([\underline{b}])$ if and only if for every $t<0$ there exists $t'<0$ such that $\varphi_t(\underline{a}) \ll_\mathrm{pw} \varphi_{t'}(\underline{b})$. That is, given $\varepsilon>0$ there exists $\delta>0$ such that $\varphi_{-\varepsilon}(\underline{a}) \ll_\mathrm{pw} \varphi_{-\delta}(\underline{b})$. Thus the second condition in the statement is equivalent to (*).

R{\o}rdam's lemma (\autoref{thm:Rordam}) shows that the first condition of the statement implies $(*)$. For the converse, assume $(*)$ and let $\varepsilon>0$.
Again by R{\o}rdam's lemma, we need to find $\underline{s}$ in $A$ such that $\underline{s}\underline{s}^*=(\underline{a}-\varepsilon)_+$ and $\underline{s}^*\underline{s} A_{\underline{b}}$.
By assumption, for each $j\in J$ there is $\delta>0$ such that $(a_j-\tfrac{\varepsilon}{2})_+\precsim (b_j-\delta)_+$ in $A_j$.
By R{\o}rdam's lemma applied to $(a_j-\tfrac{\varepsilon}{2})_+\precsim (b_j-\delta)_+$ and $\tfrac{\varepsilon}{2}$, we obtain $s_j\in A_j$ such that
\[
(a_j-\varepsilon)_+ 
= \left( (a_j-\tfrac{\varepsilon}{2})_+  -\tfrac{\varepsilon}{2}\right)_+
= s_js_j^*, \text{ and } 
s_j^*s_j \in (A_j)_{(b_j-\delta)_+}.
\]

Since $\| s_j \| \leq \| a_j \|^{\frac{1}{2}}$ for all $j\in J$, we see that $(s_j)_{j\in J}$ is bounded, and thus $\underline{s}:=(s_j)_{j\in J}\in A$ satisfies $\underline{s}\underline{s}^*=(\underline{a}-\epsilon)_+$.
Now let $f_\delta\colon\R\to[0,1]$ be continuous with $f(t)=0$ for $t\leq 0$ and $f(t)=1$ for $t\geq\delta$. Then $f_\delta(b_j)s_j^*s_jf_\delta(b_j)=s_j^*s_j$ for each $j$, and therefore $f_\delta(\underline{b})\underline{s}^*\underline{s}f_\delta(\underline{b})=\underline{s}^*\underline{s}$, which implies that $\underline{s}^*\underline{s}$ belongs to the hereditary algebra generated by $\underline{b}$.
\end{proof}

\begin{prop}
\label{prop:stable}
Let  $(A_j)_{j\in J}$ be a family of \emph{stable} C$^*$-algebras. Then the canonical map $\Phi$ from \autoref{rem:UnivPropProd} is a $\Cu$-isomorphism.
\end{prop}
\begin{proof}
Recall first that a C$^*$-algebra $A$ is said to have property (S) if if for every $a\in A_+$ and every $\varepsilon>0$ there exist $b\in A_+$ and $x\in A$ such that $a=x^*x$, $b=xx^*$ and $\|ab\|\leq\varepsilon$. It is known that any stable C$^*$-algebra has property (S), and the converse holds in the separable case; see \cite{HjeRor_stability_1998}.

Set $A=\prod_{j\in J}A_j$.
We claim that if all $A_j$ are stable, then $A$ has property (S). To show this, let $a=(a_j)_{j\in J}\in A_+$ and $\varepsilon>0$, and use that each $A_j$ has property (S) to find $b_j\in (A_j)_+$ and $x_j\in A_j$ such that 
\[a_j=x_j^*x_j, \ \ b_j=x_jx_j^*, \ \ \mbox{ and } \ \ \|a_jb_j\|\leq\varepsilon.\]
Set $b:=(b_j)_{j\in J}$ and $x:=(x_j)_{j\in J}$. Note that $\|x\|=\sup\limits_i\|x_i\|<\infty$ since $\|x_j\|=\|a_j\|^{\frac{1}{2}} \leq \|a\|^{\frac{1}{2}}$ for each $j$. Similarly, $\|b\|<\infty$. Hence $b,x$ belong to $A$, so $a=x^*x$, $b=xx^*$ and $\|ab\|\leq\varepsilon$.

Since $A$ has property (S), it is possible to compute its Cuntz semigroup by just looking at Cuntz classes of its positive elements (withouth need to go to matrices). The basic idea is that if $a$ is a positive element in the stabilisation of $A$, then one may choose a separable subalgebra of $A$ with property (S) that contains $a$, and such subalgebra will then be stable. This is used to show that Cuntz equivalence on $A$ and on $A\otimes\K$ agree.

We now prove that the natural map $\Phi$ from
\autoref{rem:UnivPropProd} is an isomorphism. We already know from \autoref{prp:PhiInjective} that $\Phi$ is an order-embedding.  Therefore, to see it is surjective it is enough to verify, using that $\Cu(A)$ has suprema of increasing sequences preserved by $\Phi$,  that the image of $\Phi$ is order-dense.

Let $x,y\in \prod_{j\in J}\Cu(A_j)$ satisfy $x\ll y$.
We will find $\underline{b}\in A_+$ such that $x\ll\Phi([\underline{b}])\ll y$.
Choose $\ll_\mathrm{pw}$-increasing paths 
\[(\vect{x}_t)_{t\in(-\infty,0]}, (\vect{y}_t)_{t\in(-\infty,0]}\in \CatPom\text{-}\prod_{j\in J}\Cu(A_j)\] representing $x$ and $y$, respectively. By \cite[Lemma~3.16]{AntPerThi_abstract_2020}, we may choose $t<0$ such that $\vect{x}_0\ll_\mathrm{pw}\vect{y}_t$.

For each $j\in J$, choose $x_{0,j}\in\Cu(A_j)$ such that $\vect{x}_0=(x_{0,j})_{j\in J}$,
and choose $a_j\in (A_j)_+$ such that $\vect{y}_t = ([a_j])_{j\in J}$.
Then $x_{0,j}\ll[a_j]$.
Using functional calculus (see \cite[Lemma 5.8]{AntPerThi_ultraproducts_2020}), one can find a contraction $b_j\in (A_j)_+$ such that $x_{0,j} \ll [(b_j-\frac{1}{2})_+]$, and $[b_j]=[a_j]$. Set $\underline{b}:=(b_j)_{j\in J}$, which is a contraction in $A_+$.
We have $\vect{x}_0
\ll_\mathrm{pw} ([(b_j-\frac{1}{2})_+])_{j\in J}
= \varphi_{-\frac{1}{2}}(\underline{b})$, and 
\[\varphi_0(\underline{b})
= ([b_j])_{j\in J}
= ([a_j])_{j\in J}
= \vect{y}_t
\ll_\mathrm{pw} \vect{y}_{\frac{t}{2}}.\]
Therefore $x =[(\vect{x}_t)_{t\leq 0}]
\ll [(\varphi_t(\underline{b}))_{t\leq 0}]
= \Phi([\underline{b}])
\ll [(\vect{y}_t)_{t\leq 0}] = y.$ This shows that $\underline{b}$ has the desired properties.
\end{proof}

\begin{eg}
The above result does not hold if the algebras are not stable. 
Set $A_j:=\C$ for each $j\in\N$, and set $A:=\prod_{j\in \N} A_j \cong \ell^\infty(\N)$.
We have $\Cu(A_j)\cong\overline{\N}$ for each $j$, and the product $\prod_{j\in \N}\NN$ is defined as the equivalence classes of $\ll_\mathrm{pw}$-increasing paths $(-\infty,0]\to \CatPom\text{-}\prod_{j\in \N}\NN$, so that $\prod_{j\in \N}\NN$ may be identified with equivalence classes of componentwise increasing paths $(-\infty,0)\to \CatPom\text{-}\prod_{j\in \N}\NN$.
In particular, compact elements in $\prod_{j\in \N}\NN$ naturally corresponds to functions $\N\to\N$. 

We claim that the natural order-embedding $\Phi\colon \Cu(A) \to \prod_{j\in \N}\NN$ is not surjective. To see this, note that $\Phi$ maps the Cuntz class of the unit of $A$ to the compact element in $\prod_{j\in \N}\NN$ corresponding to the function $f\colon \N\to\N$ with $f(j)=1$ for all $j$. 
Let $x$ be the compact element in $\prod_{j\in \N}\NN$ that corresponds to the function $g\colon\N\to\N$ with $g(j)=j$ for all $j$.
Since $g\nleq n f$ for every $n\in\N$, we have $x\nleq \infty_{ \Phi([1_A])}$.
In particular, $\Phi([1_A])$ is not full, and $\Cu(A)$ is isomorphic to a proper ideal of $\prod_{j\in \N}\Cu(A_j)$; see \autoref{nota:infa} and the comments after it. 
\end{eg}

In order to capture the Cuntz semigroup of the product of C$^*$-algebras in the non-stable case, we need to keep track of the position of the algebra and to this end the invariant needs to be modified.

\begin{df}
A \emph{scale} for a $\Cu$-semigroup $S$ is a downward hereditary subset $\Sigma\subseteq S$ that is closed under suprema of increasing sequences in $S$, and that generates $S$ as an ideal. The pair $(S,\Sigma)$ will be referred to as a \emph{scaled $\Cu$-semigroup}.

Given scaled $\Cu$-semigroups $(S,\Sigma)$ and $(T,\Theta)$, a \emph{scaled $\Cu$-morphism} is a $\Cu$-morphism $\alpha\colon S\to T$ satisfying $\alpha(\Sigma)\subseteq\Theta$.
We let $\CatCu_\mathrm{sc}$ denote the category of scaled $\Cu$-semigroups and scaled $\Cu$-morphisms.
\end{df}

For a C$^*$-algebra $A$, note that
\[
\Sigma_A := \big\{ x\in\Cu(A) \colon \text{there exists } a\in A_+ \text{ such that } x\leq [a] \big\}
\]
is a scale for $\Cu(A)$.
We call $\Cu_\mathrm{sc}(A):=(\Cu(A),\Sigma_A)$ the \emph{scaled Cuntz semigroup} of $A$.
Given a homomorphism $\varphi\colon A\to B$ of \ca s, the induced $\Cu$-morphism $\Cu(\varphi)\colon\Cu(A)\to\Cu(B)$ maps $\Sigma_A$ into $\Sigma_B$, and thus $\Cu_\mathrm{sc}$ defines a functor from the category \textbf{C$^*$} of C$^*$-algebras to $\CatCu_\mathrm{sc}$. The following is \cite[Theorem 4.6]{AntPerThi_ultraproducts_2020}.

\begin{thm}
\label{pgr:scaledcu}
The category $\CatCu_\mathrm{sc}$ is complete.
\end{thm}
\begin{proof}
Let $\mathrm{\mathbf{I}}$ be a small category, and let $\mathrm{F}\colon \mathrm{\mathbf{I}}\to\CatCu_\mathrm{sc}$ be a functor, written 
$i\mapsto \mathrm{F}(i)=(S_i,\Sigma_i)$.
Considering the underlying functor $\mathrm{\mathbf{I}}\to\CatPom$ given by $i\mapsto S_i$, we let $(S,(\pi_i)_{i\in \mathrm{\mathbf{I}}})$ be the limit of $\mathrm{F}$ in $\CatPom$ with $S$ as in the proof of \autoref{thm:Cucomplete}. We have that
\[
\Sigma_0 := S \cap \prod_{i\in \mathbf{I}} \Sigma_i
= \Big\{ (s_i)_{i\in \mathrm{\mathbf{I}}}\in\prod_{i\in \mathrm{\mathbf{I}}} \Sigma_i \colon  \mathrm{F}(f)(s_i)=s_j \text{ for all } f\colon i\to j \text{ in }\mathrm{\mathbf{I}} \Big\}.
\]
Then $\Sigma_0$ is a downward hereditary subset of $S$ satisfying $\pi_i(\Sigma_0)\subseteq \Sigma_i$ for all $i\in \mathrm{\mathbf{I}}$.

Composing with the forgetful functor (that forgets the scaled structure), we get as limit
$\tau( S, \ll_\mathrm{pw} )$ together with maps 
\[
\psi_i:=\tau(\pi_i)\colon\tau(S,\ll_\mathrm{pw})\to\tau(S_i,\ll)\cong S_i
\]
for all $i\in \mathbf{I}$; see the proof of \autoref{thm:Cucomplete}.
Set
\[
\Sigma := \Big\{ [(\vect{x}_t)_{t\leq 0}] \in \tau( S, \ll_\mathrm{pw} ) : \vect{x}_t \in \Sigma_0 \text{ for all } t<0 \Big\}.
\]
Then $\Sigma$ is a downward hereditary subset of $\tau( S, \ll_\mathrm{pw} )$ that is closed under passing to suprema of increasing sequences.
Let $\langle \Sigma \rangle$ denote the ideal of $\tau( S, \ll_\mathrm{pw} )$ generated by $\Sigma$.
Then $(\langle\Sigma\rangle,\Sigma)$ is a scaled $\Cu$-semigroup.
Moreover, for each $i\in \mathbf{I}$ we have $\psi_i(\Sigma)\subseteq\Sigma_i$, which shows that $\psi_i\colon(\langle\Sigma\rangle,\Sigma)\to(S_i,\Sigma_i)$ is a scaled $\Cu$-morphism.
One can show that this defines a limit for $\mathrm{F}$ in $\CatCu_\mathrm{sc}$. We omit the details.
\end{proof}

\begin{thm}
\label{thm:Cupreserves}
The scaled Cuntz semigroup functor preserves products.
\end{thm}
\begin{proof}
We first show how to construct the product in the category $\CatCu_\mathrm{sc}$. This uses as an ingredient the proof of \autoref{pgr:scaledcu}.

Let  $(S_j)_{j\in J}$ be a family of $\Cu$-semigroups and let $(S,\Sigma)$ be their scaled product in $\CatCu_\mathrm{sc}$. To get a concrete description of this object, we first take the set-theoretic product $\prod_{j\in J}\Sigma_j$, which is a downward hereditary subset of $\CatPom\text{-}\prod_{j\in J} S_j$, and set
\[
\Sigma = \Big\{ [(\vect{x}_t)_{t\leq 0}] \in \prod_{j\in J} S_j : \vect{x}_t \in \prod_{j \in J}\Sigma_j \text{ for every } t<0 \Big\}.
\]
Then $S$ is the ideal of $\prod_{j\in J} S_j$ generated by $\Sigma$.
Given $[(\vect{x}_t)_{t\leq 0}] \in \prod_{j\in J} S_j$, we have $[(\vect{x}_t)_{t\leq 0}]\in S$ if and only if for every $t<0$ there exist $\sigma^{(1)},\ldots,\sigma^{(N)}\in\prod_{j\in J} \Sigma_j$ such that $\vect{x}_t\ll_\mathrm{pw}\sigma^{(1)}+\ldots+\sigma^{(N)}$.

If now $(A_j)_{j\in J}$ is a family of C$^*$-algebras, we use again $(S,\Sigma)$ to denote the scaled product of $(\Cu_\mathrm{sc}(A_j))_{j\in J}$ with $\Sigma\subseteq S\subseteq \prod_{j\in J}\Cu(A_j)$ as defined above.
Set $A=\prod_{j\in J}A_j$.
Then the map $\Phi\colon \Cu(A)\to \prod_{j\in J} \Cu(A_j)$ defined in
\autoref{rem:UnivPropProd}
is an order-embedding by \autoref{prp:PhiInjective}. Using the strategy in the proof of \autoref{prop:stable} with additional care, one can show that the image of $\Phi$ is $S$ and moreover it identifies the scale of $\Cu(A)$ with $\Sigma$:
\[
\Cu_\mathrm{sc}(A) = (\Cu(A), \Sigma_{A}) \cong \prod_{j\in J} (\Cu(A_j),\Sigma_{A_j}) = (S,\Sigma).\qedhere
\]
\end{proof}

To close this section, we turn our attention to ultraproducts. First, we give a categorical definition.

\begin{df}
\label{pgr:ultraprCat}
Let $\mathrm{\mathbf{C}}$ be a category that has products and inductive limits, let $J$ be a set, let $\filter$ be an ultrafilter on $J$, and let $(X_j)_{j\in J}$ be a family of objects in $\mathrm{\mathbf{C}}$.
Given a subset $K\subseteq J$ and $j\in K$, we write
$\pi_{j,K}\colon \prod_{k\in K} X_k\to X_j$ for the canonical
quotient map.
Given subsets $G\subseteq K\subseteq J$, the universal property of the product gives a morphism
\[
\varphi_{G,K}\colon\prod_{j\in K}X_j \to \prod_{j\in G}X_j,
\]
such that $\pi_{j,K}=\pi_{j,G}\circ\varphi_{G,K}$ for each $j\in G$.

Ordering the elements of $\filter$ by reversed inclusion, we have that $\filter$ is upward directed, and thus the objects $\prod_{j\in K}X_j$ for $K\in\filter$, and morphisms $\varphi_{G,K}$ for $K,G\in\filter$ with $K\supseteq G$, define an inductive system indexed over $\filter$. The inductive limit of this system is called the (categorical) \emph{ultraproduct} of $(X_j)_{j\in J}$ along $\filter$:
\[
\prod_{\filter} X_j := \varinjlim_{K\in\filter} \prod_{j\in K}X_j.
\]
We let $\pi_\filter\colon \prod_{j\in J}X_j \to \prod_\filter X_j$ denote the natural morphism to the inductive limit.
%
\end{df}


Let $\mathrm{\mathbf{C}}$ and $\mathrm{\mathbf{D}}$ be categories with products and inductive limits, and let $\mathrm{F}\colon\mathrm{\mathbf{C}}\to\mathrm{\mathbf{D}}$ be a functor that preserves inductive limits and products. 
Then for any set $J$, any ultrafilter $\filter$ on $J$, and any family $(X_j)_{j\in J}$ of objects in $\mathrm{\mathbf{C}}$, there is a natural isomorphism
\[
\Phi_\filter\colon \mathrm{F}\Big(\prod_\filter X_j\Big) \to \prod_\filter \mathrm{F}(X_j).
\]

Let $(A_j)_{j\in J}$ be a family of C$^*$-algebras and let $\filter$ be an ultrafilter on $J$. Set $A=\prod_{\mathcal{U}}A_j$. The discussion above, combined with the properties of the categories $\CatCu$ and $\CatCu_\mathrm{sc}$, as well as the functors $\Cu$ and $\Cu_{\mathrm{sc}}$, yield natural (scaled) $\CatCu$-morphisms
\[
\Phi_\filter\colon \Cu(A) \to \prod_\filter \Cu(A_j) \text{ and }
\Phi_{\filter,\mathrm{sc}}\colon \Cu_\mathrm{sc}(A) \to \prod_\filter \Cu_\mathrm{sc}(A_j).
\]
Applying \autoref{prop:stable} and \autoref{thm:Cupreserves}, we obtain the following results:

\begin{prop}
\label{prp:CuUltraprodStable}
Given an ultrafilter $\filter$ on a set $J$ and a family $(A_j)_{j\in J}$ of stable C$^*$-algebras,
the map $\Phi_\filter\colon \Cu\big(\prod_\filter A_j\big) \to \prod_\filter \Cu(A_j)$ is an isomorphism.
\end{prop}

\begin{thm}
\label{thm:CuUltraprod}
The scaled Cuntz semigroup functor preserves ultraproducts. In other words, given an ultrafilter $\filter$ on a set $J$ and a family $(A_j)_{j\in J}$ of C$^*$-algebras, the map $\Phi_{\filter,\mathrm{sc}}\colon
\Cu_\mathrm{sc}\big(\prod_\filter A_j\big) \to \prod_\filter \Cu_\mathrm{sc}(A_j)$ is an isomorphism.
\end{thm}

A different view on ultraproducts, more akin to the usual construction for C$^*$-algebras, can also be given in this setting. We give the statement without proof, and refer the reader to \cite[Section 7]{AntPerThi_ultraproducts_2020} for further details.

\begin{thm}
\label{pgr:diffultra}
Let $\filter$ be an ultrafilter on a set $J$, and let $(S_j)_{j\in J}$ be a family of $\Cu$-semigroups.
Set
\[
\mathrm{c}_\filter\big( (S_j)_j \big) := \Big\{ [(\vect{x}_t)_{t\leq 0}] \in \prod_{j\in J} S_j\colon \supp( \vect{x}_t ) \notin\filter \text{ for each } t<0  \Big\}.
\]
Then $\mathrm{c}_\filter\big( (S_j)_{j\in J} \big)$ is an ideal in $\prod_{j\in J} S_j$, and there is a natural isomorphism 
\[
\prod_{\filter} S_j \cong \Big(\prod_{j\in J}S_j\Big) / \mathrm{c}_\filter((S_j)_{j\in J}).
\]

%
%


If now $(A_j)_{j\in J}$ is a family of C$^*$-algebras and $(S,\Sigma)$ is the scaled product of $(\Cu_\mathrm{sc}(A_j))_{j\in J}$ with $\Sigma\subseteq S\subseteq \prod_{j\in J}\Cu(A_j)$ as in \autoref{pgr:scaledcu}, we have a natural isomorphism 
\[\Cu\Big( \prod_\filter A_j \Big) \cong S/\left( S\cap \mathrm{c}_\filter\big((\Cu(A_j))_{j\in J}\big) \right).\]
In particular, the map $\Phi_\filter$ from \autoref{pgr:ultraprCat} is an order-embedding that identifies $\Cu\big(\prod_\filter A_j\big)$ with the image of $S$ under the map $\pi_\filter$ from \autoref{pgr:diffultra}.
\end{thm}

The computation of the Cuntz semigroup of ultraproducts of C$^*$-algebras in terms of the semigroups of the individual algebras (by means of the scaled product as described above) will be important in studying when the completion of the so-called limit traces is dense in the trace simplex of an ultraproduct; see \cite{AntPerRobThi_traces_2022}.

\section{Outlook}

In this final section, we give a sample of problems in
the area of Cuntz semigroups which we think are likely to guide 
the research in this field in the upcoming years. Some of the 
problems are very difficult and should be regarded as long-term 
goals, while other ones are more tangible. Also, some of 
the problems are stated in rather vague terms, while other ones 
are fairly concrete. Some of the problems below were 
suggested in the Cuntz semigroup workshop which took place in
September 2022 in Kiel, Germany.

As explained in \autoref{cor:ABnotiso}, Toms used the Cuntz semigroup
to distinguish two simple, separable, unital, nuclear C$^*$-algebras
with identical Elliott invariants, thus providing a counterexample
to the original formulation of the Elliott conjecture. The following
is thus a natural task:

\begin{pbm}\label{pbm:Cu(Toms)}
For the non $\mathcal{Z}$-stable C$^*$-algebra $A$ constructed by
Toms (see \autoref{cor:ABnotiso}), compute $\Cu(A)$.
\end{pbm}

Since Toms' construction is based on that of Villadsen, one should also attempt the following:

\begin{pbm}\label{pbm:Cu(Villadsen)}
For a Villadsen algebra $A$ of the first 
type as in \cite{Vil_simple_1998}, compute $\Cu(A)$.
\end{pbm}

Some Villadsen algebras of the first type are $\mathcal{Z}$-stable, and 
in those cases the computation of the Cuntz semigroup is given
by \autoref{thm:CuAtimesZ}. Both \autoref{pbm:Cu(Toms)} and 
\autoref{pbm:Cu(Villadsen)} are rather vague, and one 
concrete problem one should attempt is the computation of 
the dimensions of these Cuntz semigroups, in the sense 
of \cite{ThiVil_covering_2022, ThiVil_coveringII_2021}.

It was shown in \cite{TomWin_strongly_2007} that any two unital homomorphisms from 
$\mathcal{Z}$ into a $\mathcal{Z}$-stable C$^*$-algebra are 
approximately unitarily equivalent, and in particular they must
agree at the level of the functor Cu. It would be interesting to find 
an explicit example showing that this fails in the non 
$\mathcal{Z}$-stable case. Therefore, we suggest:

\begin{pbm}
Let $A$ be a non $\mathcal{Z}$-stable C$^*$-algebra as in either 
\autoref{pbm:Cu(Toms)} or \autoref{pbm:Cu(Villadsen)}. Are there 
two distinct Cu-morphisms $\Cu(\mathcal{Z})\to \Cu(A)$ preserving
unit classes? 
\end{pbm}

Another class of C$^*$-algebras for which it would be very interesting
to compute the Cuntz semigroup is that of reduced group C$^*$-algebras, 
particularly for C$^*$-simple groups. While obtaining an explicit 
computation may be out of reach, it should be possible to obtain
some structural information. As a first step, we propose
the following problem, which is a Cuntz semigroup version
of Blackadar's fundamental comparison property
\cite{Bla_comparison_1988} for projections in $C^*_\mathrm{r}(\mathbb{F}_n)$:

\begin{pbm} 
Compute $\Cu(C^*_\mathrm{r}(\mathbb{F}_n))$ for $n\geq 2$, or at least determine
whether it is almost unperforated. 
\end{pbm}

As it turns out, it is not easy to find examples of stably finite
C$^*$-algebras whose Cuntz semigroups fail to be weakly cancellative.
There exist commutative examples, but we do not know if simple
ones exist as well.

\begin{qst}
Does there exist a simple, stably finite C$^*$-algebra whose 
Cuntz semigroup does not have weak cancellation?
\end{qst}

If a C$^*$-algebra as above exists, then its stable rank will 
necessarily be greater than 1 by \autoref{wc}.

A $\Cu$-semigroup is said to be \emph{almost divisible} 
if, given $x,x'\in S$ with $x'\ll x$, then for all $n\in\N$ there is $y\in S$ such that $ny\leq x$ and $x'\leq (n+1)y$. If $A$ is a $\mathcal{Z}$-stable C$^*$-algebra, then $\Cu(A)$ is almost divisible, as was shown by R\o rdam in \cite{Ror_stable_2004}; this is proved
very similar to \autoref{thm:RordamStrComp}. 
The question above is related to the following algebraic question, raised in \cite{AntPerThi_tensor_2018}:

\begin{qst}
Under what additional axioms (besides (O5) and (O6)) is a simple $\Cu$-semigroup which is also almost unperforated and almost divisible necessarily weakly cancellative?
\end{qst}

We now turn to connections to classification. 
As mentioned in the introduction, the Cuntz semigroup has been
successfully used to classify interesting classes of nonsimple
C$^*$-algebras. The latest and most general result in this direction is 
due to Leonel Robert \cite{Rob_classification_2012}, and there is
reason to believe that the class of C$^*$-algebras that Robert 
considered is the largest that can be classified solely in
terms of $\Cu$. 

\begin{qst} 
Are there variations of the Cuntz semigroup that can be used to 
classify larger classes of non-simple C$^*$-algebras?
\end{qst}

The question above is very vague, but it is motivated by the 
invariant $\Cu^\sim$ considered by Robert in \cite{Rob_classification_2012} and by Robert-Santiago 
in \cite{RobSan_revised_2021}, which is necessary to obtain 
classification in the non-unital setting. Another direction in
which this question can be interpreted is by trying to incorporate
K$_1$-information into the invariant. Some steps in this direction
have been made by Cantier in \cite{Can_unitary_2021} (see also \cite{Can_ideal_2021}).

As we saw in \autoref{thm:CuAtimesZ}, 
in the simple, separable, nuclear, finite $\mathcal{Z}$-stable setting, the 
Elliott invariant $\Ell$ contains the same information as 
the pair $(\Cu, \mathrm{K}_1)$. In particular, two simple,
separable, unital, nuclear, finite $\mathcal{Z}$-stable C$^*$-algebras $A$ 
and $B$
satisfying the UCT are isomorphic if and only if $\Cu(A)\cong \Cu(B)$
and K$_1(A)\cong$ K$_1(B)$. 
The modern approach to 
classification focuses on the classification of homomorphisms,
and it 
would thus be interesting to know to what extent the Cuntz semigroup can 
be used in this setting (for some results in this direction, see \cite{Can_class_2022}). Since the invariant needed to classify
homomorphisms is larger than $\Ell$, in particular including
algebraic K$_1$-data, the following seems like a natural question:

\begin{qst}
To what extent can the algebraic K$_1$-group be 
recovered from the Cuntz semigroup?
\end{qst}

The Cuntz semigroup is expected to be useful for classification 
also beyond the $\mathcal{Z}$-stable setting. A positive answer
to the following question would be a significant breakthrough in
the area.

\begin{qst} 
Can one use $(\Cu(A), \KK_1(A))$ to classify a class of 
simple, nuclear C$^*$-algebras bigger than the one in
\autoref{thm:classification}?\end{qst}

The study of group actions on C$^*$-algebras is a very fruitful one.
Some of the most recent research in the area has shown that 
Cuntz semigroup techniques are extremely powerful in this setting (as proved in 
\cite{GarGefKraNar_classifiability_2022, GarGefNarVac_dynamical_2022,
GarGefKraNarVac_tracial_2022,
BosPerZacWu_aemeasure_2023, BosPerZacWu_aedyn_2023}),
thus suggesting that the theory of group
actions on Cuntz semigroups (as developed in \cite{BosPerZacWu_dynamicalCuntz_2022, BosPerZacWu_typesemigroup_2023}) may lead to interesting constructions:

\begin{pbm}
Develop the theory of crossed products and Rokhlin properties for group actions on $\Cu$-semigroups.
\end{pbm}

For actions of compact groups, an equivariant version of the 
Cuntz semigroup, resembling equivariant K-theory, has been 
studied in \cite{GarSan_equivariant_2017}. For actions
with the Rokhlin property, the induced dynamical system on 
$\Cu$ has been explored in \cite{GarSan_equivariant_2016, Gar_crossed_2017,Gar_compact_2018}.  

KK-theory is a bivariant joint generalization of K-theory 
and K-homology: for two C$^*$-algebras $A$ and $B$, the 
KK-group KK$(A,B)$ is a natural homotopy equivalence class of 
$(A,B)$-Hilbert bimodules, and it behaves as K-homology in the 
first coordinate and as K-theory in the second. KK-theory provides
a strong link between operator algebras, noncommutative
geometry and index theory. 
 
\begin{pbm}
Use the bivariant version of $\Cu$ introduced in \cite{AntPerThi_abstract_2020,AntPerThi_abstractII_2020} to establish more connections 
 with noncommutative geometry.
\end{pbm}


\end{document}